\documentclass[11pt]{amsart}

\usepackage{amsmath, amssymb, amscd, amsthm,epsfig,epic,eepic}
\usepackage[justify]{ragged2e}

\usepackage[T1]{fontenc}

\usepackage{framed}
\usepackage{multicol}
\usepackage[normalem]{ulem}
\usepackage{booktabs}
\usepackage{multirow}
\useunder{\uline}{\ul}{}
\usepackage{mathrsfs}
\usepackage{mathabx,mathdots}
\usepackage[shortlabels]{enumitem}
\usepackage[new]{old-arrows}

\makeatletter
     \let\oldfootnote\footnote
     \def\footnote{\@ifstar\footnote@star\footnote@nostar}
     \def\footnote@star#1{{\let\thefootnote\relax\footnotetext{#1}}}
     \def\footnote@nostar{\oldfootnote}
     \makeatother

\newcommand{\N}{\mathbf{N}}
\newcommand{\Z}{\mathbf{Z}}
\newcommand{\Q}{\mathbf{Q}}
\renewcommand{\AA}{\mathbf{A}}

\makeatletter
\newcommand*{\defeq}{\mathrel{\rlap{%
                      \raisebox{0.3ex}{$\cdot$}}%
                      \raisebox{-0.3ex}{$\cdot$}}%
                      =}
\makeatother

\newcommand{\proj}{\mathbf{P}}
\newcommand{\Gm}{\mathbf{G}_m}

\newcommand{\Ga}{\mathbf{G}_a}

\newcommand{\dd}{\mathop{}\!\text{d}}

\DeclareMathOperator{\Proj}{Proj}
\DeclareMathOperator{\Hom}{Hom}
\DeclareMathOperator{\Aut}{Aut}
\DeclareMathOperator{\Spec}{Spec}

\DeclareMathOperator{\Stab}{Stab}
\DeclareMathOperator{\Supp}{Supp}

\DeclareMathOperator{\inn}{inn}
\DeclareMathOperator{\Lie}{Lie}
\DeclareMathOperator{\NE}{NE}
\DeclareMathOperator{\GL}{GL}

\DeclareMathOperator{\PSL}{PSL}
\DeclareMathOperator{\PGL}{PGL}
\DeclareMathOperator{\SO}{SO}
\DeclareMathOperator{\Sp}{Sp}
\DeclareMathOperator{\PSp}{PSp}
\DeclareMathOperator{\Spin}{Spin}

\DeclareMathOperator{\Pic}{Pic}
\DeclareMathOperator{\Sym}{Sym}

\DeclareMathOperator{\Grass}{Grass}

\DeclareMathOperator{\diag}{diag}
\DeclareMathOperator{\ad}{ad}

\usepackage{xcolor, color}
\definecolor{links}{HTML}{A93226}
\definecolor{bluetto}{HTML}{1ABC9C}
\definecolor{cite}{HTML}{21618C}
\usepackage{hyperref, cleveref}
\hypersetup{
    colorlinks=true,
    linkcolor=cite,
    citecolor=links,
    filecolor=links,
    urlcolor=links,
    linktocpage=true
}

\usepackage[inner=3cm,outer=3cm,top=2.75cm,bottom=2.75cm]{geometry}

\usepackage[all]{xy}
\usepackage{tikz-cd}

\theoremstyle{plain}

\newtheorem{theorem}{Theorem}[section]
\newtheorem{lemma}[theorem]{Lemma}
\newtheorem{proposition}[theorem]{Proposition}
\newtheorem{corollary}[theorem]{Corollary}
\newtheorem{theorem*}{Theorem}
\newtheorem{proposition*}[theorem*]{Proposition}
\newtheorem*{proposition**}{Proposition}

\theoremstyle{definition}
\newtheorem{definition}[theorem]{Definition}
\newtheorem{remark}[theorem]{Remark}

\newtheorem{example}[theorem]{Example}
\numberwithin{equation}{section}

\begin{document}

\title{PROJECTIVE HOMOGENEOUS VARIETIES OF PICARD RANK ONE IN SMALL CHARACTERISTIC}

\author{Matilde Maccan}
\email{matilde.maccan@univ-rennes.fr}

\footnote*{Keywords: projective homogeneous variety, parabolic subgroup scheme, positive characteristic}
\footnote*{2020 Mathematics Subject Classification:
Primary: 14M17, Secondary: 14L15, 17B20}
\date{\today}

\maketitle

\begin{abstract}
We extend to characteristic $2$ and $3$ the classification of projective homogeneous varieties of Picard group isomorphic to $\Z$, corresponding to parabolic subgroups with maximal reduced subgroup. The latter are all obtained as product of a maximal reduced parabolic with the kernel of a purely inseparable isogeny. This fails in type $G_2$ and characteristic $2$, for which we exhibit an explicit counterexample and show it is the only one, thus completing the classification. We then construct new examples of projective homogeneous varieties of Picard rank two.
\end{abstract}

\tableofcontents

\section*{Introduction}
\label{sec:int}

We work in the setting of affine group schemes of finite type
over an algebraically closed field $k$. Our object of interest are projective varieties over $k$, homogeneous under their automorphism group: a first class of examples is given by flag varieties, whose natural generalisation are quotients of semisimple groups by parabolic subgroups, which are not necessarily reduced.\\
In characteristic zero, the structure of parabolic subgroups is well known: fixing a semisimple group, a Borel subgroup and a maximal torus $G \supset B \supset T$, 
  there is a bijection between parabolic subgroups containing $B$ and subsets of the set of simple roots of $G$: it is a classical fact that a parabolic subgroup is determined by the simple roots forming a basis for the root system of a Levi subgroup. Over a field of positive characteristic,
  parabolic subgroups can be nonreduced hence homogeneous spaces might have very different geometric properties: see for example the computation of the character associated to the canonical bundle in \cite{Lauritzen}, which shows that such varieties are in general not Fano. The easiest example is the hypersurface in $\proj^2 \times \proj^2$ given by the equation $x_0y_0^p+x_1y_1^p + x_2y_2^p=0$, which is homogeneous under $\PGL_3$ and not Fano if $p>3$.\\
  If the characteristic is equal to at least $5$, or if the Dynkin diagram of $G$ is simply laced (types $A_n$, $D_n$, $E_6$, $E_7$ and $E_8$), Wenzel \cite{Wenzel}, Haboush and Lauritzen \cite{HL93} show that all parabolic subgroups of $G$ can be obtained from reduced ones by fattening with Frobenius kernels and intersecting. More precisely, they are all of the form 
  \[G_{m_1}P^{\alpha_1} \cap \ldots \cap G_{m_r} P^{\alpha_r},\]
  where $G_m$ denotes the kernel of the $m$-th iterated Frobenius morphism on $G$ and $P^{\alpha}$ denotes the maximal reduced parabolic subgroup whose Levi subgroup has as basis all simple roots except for $\alpha$. When no assumption on the characteristic is made, we will call them parabolic subgroups \emph{of standard type}. The proof of \cite{Wenzel} relies heavily on the structure constants (defined over $\Z$) relative to a Chevalley basis of the Lie algebra of a simply connected semisimple group. By construction such constant are integers with absolute value strictly less than five: the hypothesis on the characteristic and on the Dynkin diagram guarantees that they do not vanish over $k$. This raises the natural question of how to generalize the classification of parabolic subgroups to characteristic two and three.\\

In this paper we manage to provide an answer to this question concerning the easiest - combinatorially speaking - class of parabolic subgroups, those having maximal reduced part equal to $P^\alpha$ for some simple root $\alpha$ of $G$. These subgroups correspond to homogeneous projective varieties of Picard group isomorphic to $\Z$: as illustrated later in Subsection \ref{sec:divisors}, the Picard group of a homogeneous space $G/P$ is free abelian with rank equal to the number of simple roots of $G$ not belonging to the root system of a Levi subgroup of the reduced part of $P$. Our main result is the following, allowing us to complete the classification in all types and characteristics.

\begin{theorem*}
\label{main}
Let $X$ be a homogeneous projective variety, over an algebraically closed field of any characteristic, with Picard group isomorphic to $\Z$. Then $X$ is either the quotient of a simple adjoint group by a maximal \emph{reduced} parabolic subgroup, or it is isomorphic to $G_2/P_{\mathfrak{l}}$ (this second case can only arise in characteristic two).
\end{theorem*}

In the above statement, $P_{\mathfrak{l}}$ is a certain exotic parabolic subgroup scheme, introduced and studied in \Cref{section_handl} and \Cref{section_Pl}. The proof of \Cref{main} articulates in two different parts: \Cref{final_rank1}, which treats all cases but $G_2$ in characteristic two, and \Cref{main_G2} which completes the classification.\\

The paper is organized as follows. In \Cref{sec1}, we build on previous work of Borel and Tits \cite{BorelTits}, then completed and re-elaborated in \cite{CGP15}, to give a factorisation result for isogenies with simply connected source. This digression is self-contained but motivated by the fact that - in Picard rank one - purely inseparable isogenies will generalize the role of the Frobenius morphism in \cite{Wenzel}. An important ingredient is the so-called \emph{very special isogeny} of a simple simply connected group $G$, which is the quotient
\[\pi_G \colon G \longrightarrow \overline{G}\]
by the unique minimal noncentral subgroup of $G$ with trivial Frobenius. It turns out (as shown in \cite{CGP15}) that when the Dynkin diagram of $G$ has an edge of multiplicity equal to the characteristic, such a subgroup is strictly contained in the Frobenius kernel. In particular, $\pi_G$ acts as a Frobenius morphism on root subgroups associated to short roots, while it is an isomorphism on root subgroups associated to long ones. The factorisation of isogenies reads as follows, where we denote as $F^m_G \colon G \longrightarrow G^{(m)}$ the $m$-th iterated relative Frobenius homomorphism of $G$.

\begin{proposition*}
Let $G$ be a simple and simply connected algebraic group over an algebraically closed field $k$. Let $f \colon G \rightarrow G^\prime$ be an isogeny.\\
Then there exists a unique factorisation of $f$ as
\[
\begin{tikzcd}
f \colon G \arrow[rr, "\pi"] && \overline{G} \arrow[rr, "F_{\overline{G}}^{m}"] && (\overline{G})^{(m)} \arrow[rr, "\rho"]  && G^\prime,
\end{tikzcd}
\]
where $\rho$ is a central isogeny and $\pi$ is either the identity or - when the Dynkin diagram of $G$ has an edge of multiplicity $p$ - the very special isogeny $\pi_G$.
\end{proposition*}

We introduce and prove the first part of \Cref{main} in \Cref{sec2}. As the different behaviour in type $G_2$ confirms, there is no way to prove this result using general geometric arguments nor working over $\Z$, so we proceed by a case-by-case analysis. The proof essentially articulates in three steps: the first one consists of some elementary reductions, showing that it is enough to prove that if $P$ has maximal reduced subgroup and $G$ acts faithfully on $G/P$, then $P$ must be itself reduced. The second step is exploiting the explicit matricial description of the quotient 
\[\Lie G/ \Lie P_{\text{red}},\] seen as a representation of a Levi subgroup of $P_{\text{red}}$. Finally, the last step involves considering some of the structure constants (chosen so that they do not vanish, depending on the characteristic) and concluding using the notion of very special isogeny.\\

Next, we consider the case of characteristic $2$ and type $G_2$, with short simple root $\alpha_1$ and long simple root $\alpha_2$. Perhaps surprisingly, the analogous strategy of proof works when the reduced part is $P^{\alpha_2}$, but fails when considering $P^{\alpha_1}$, due to the vanishing of structure constants. We deduce that there exist exactly two maximal $p$-Lie subalgebras $\mathfrak{h}$ and $\mathfrak{l}$ of $\Lie G$ strictly containing $\Lie P^{\alpha_1}$. 
We describe them explicitly and consider the corresponding subgroups of height one in $G$, which give rise to two new parabolic subgroups denoted $P_{\mathfrak{h}}$ and $P_{\mathfrak{l}}$. Then we study the corresponding homogeneous spaces, by means of the description of $G_2$ as the automorphism group of an octonion algebra, as in \cite{SV} and \cite{Heinloth}. One turns out to be isomorphic to the projective space $\proj^5$, while we realize the other as a hyperplane section of the $\Sp_6$-homogeneous variety of isotropic $3$-dimensional subspaces in a $6$-dimensional vector space (the study of the second variety, carried on in \Cref{section_Pl}, is computational and slightly cumbersome). 
For the sake of brevity, computations concerning the group $G_2$ and its root subgroups can be found in the Appendix [\ref{sec4}]. We conclude this part with the following result (see \Cref{parabolics_G2} for a more detailed statement), which allows in particular to end the proof of \Cref{main}.

\begin{proposition*}
    Let $G$ be of type $G_2$ in characteristic two.\\
    Then the nonreduced parabolic subgroups of $G$ having $P^{\alpha_1}$ as reduced part are either of standard type, or obtained from $P_{\mathfrak{l}}$ and $P_{\mathfrak{h}}$ by pulling back with an iterated Frobenius homomorphism.
\end{proposition*}

We deduce in \Cref{sec3} the desired consequence of \Cref{main}: the statement focuses exclusively on the classification of parabolic subgroups with maximal reduced part, and requires no assumptions on the characteristic of the base field.

\begin{proposition*}
\label{prop3}
Let $G$ be simple and $P$ be a parabolic subgroup of $G$ such that its reduced subgroup is $P^\alpha$ for some simple root $\alpha$. Then there exists an isogeny $\varphi$ with source $G$ such that
\[
P = (\ker \varphi) P^\alpha,
\]
unless $G$ is of type $G_2$ over a field of characteristic $p= 2$ and $\alpha$ is the short simple root.
\end{proposition*}

We prove \Cref{prop3} as well as a criterion to determine when two projective homogeneous spaces with Picard group $\Z$ are isomorphic as varieties. The remaining part of Section 3 is devoted to the display of a family of projective homogeneous spaces of Picard rank two, whose underlying varieties are \emph{not of standard type}, where the last terminology means not isomorphic (as a variety) to some quotient with stabilizer a parabolic subgroup of standard type. 
We will follow the conventions on root systems adopted by Bourbaki \cite{Bourbaki}. The statement is the following:

\begin{proposition*}
    Consider a simple, simply connected group $G$ over an algebraically closed field of characteristic $2$ and distinct simple roots $\alpha$ and $\beta$ such that: either $G$ is of type $B_n$ or $C_n$ and the pair $(\alpha,\beta)$ is of the form $(\alpha_j,\alpha_i)$ with $i < j < n$ or $j=n$ and $i <n-1$, or $G$ is of type $F_4$ and the pair $(\alpha,\beta)$ is one among
    \[
    (\alpha_1,\alpha_4), \quad (\alpha_2,\alpha_1), \quad (\alpha_2,\alpha_4), \quad (\alpha_3,\alpha_1), \quad (\alpha_3,\alpha_4), \quad (\alpha_4,\alpha_1).
    \]
    Then the homogeneous space $X= G/((\ker \pi_G)P^\alpha \cap P^\beta)$
    is \emph{not} of standard type, where $\pi_G$ denotes the very special isogeny of $G$.
\end{proposition*}

The strategy of proof consists first in recalling and precising a few facts on the Białynicki-Birula decomposition of $G$-simple projective varieties, following \cite{Brion02}. Next we specialize the outline of this decomposition to the particular case of homogeneous spaces. This leads to the description of the Picard group and of the group of $1$-cycles on $X=G/P$, as well as the definition of a finite family of contractions on $X$ indexed by the simple roots of $G$ not belonging to the root system of a Levi subgroup of $P_{\text{red}}$. More precisely, the contraction associated to a root $\alpha$ sends all classes of curves to a point, except for those which are numerically equivalent to the unique $B$-invariant curve passing through the image of the base point of $X$ by the reflection with respect to $\alpha$ in the Weyl group.
This construction, together with the results on automorphism groups in \cite{Demazure}, allows us to conclude. We end with a final question concerning the more general classification of parabolic subgroups in characteristic $2$ and $3$.\\

\textbf{Acknowledgments.} I would like to thank my PhD advisors Michel Brion and Matthieu Romagny for their precious guidance and support, as well as Pierre-Emmanuel Chaput, Philippe Gille and David Stewart for the useful suggestions.

\section{Preliminary work on isogenies}
\label{sec1}

\subsection{Setting and notation}
\label{subsection_notations}

In this work, $k$ denotes an algebraically closed field of prime characteristic $p >0$. When $V$ is a finite-dimensional $k$-vector space, we adopt the convention for $\proj(V)$ to be lines in $V$. By parabolic subgroup we always mean parabolic subgroup scheme.\\
Let $G \supset B \supset T$ be respectively a semisimple, simply connected algebraic group over $k$, a Borel subgroup and a maximal torus contained in it. Our aim is to classify all homogeneous projective $G$-varieties, which are quotients of the form $G/P$, where $P$ is a parabolic subgroup of $G$, not necessarily reduced. By conjugacy of the Borel subgroups, we might restrict ourselves to those containing the Borel subgroup $B$, which we call \emph{standard} parabolic subgroups. From now on, every parabolic subgroup will be standard, unless otherwise mentioned. Such a classification has been established in \cite{Wenzel} and \cite{HL93}, under the assumption that either $p \geq 5$ or that the root system of $G$ relative to $T$ is simply laced.\\
Let us list the main notations that are fixed throughout the paper, which mostly agree with those of \cite{Wenzel}. Concerning root systems, we follow conventions from Bourbaki \cite{Bourbaki} :
\begin{itemize}
    \item $\Phi = \Phi(G,T)$ is the root system of the pair $(G,T)$,
    \item $\Phi^+ = \Phi(B,T)$ is the subset of positive roots associated to the Borel subgroup $B$,
    \item $\Delta$ is the corresponding basis of simple roots,
    \item $W = W(G,T) = W(\Phi)$ is the Weyl group of $(G,T)$,
    \item $s_\alpha$ is the reflection associated to the simple root $\alpha \in \Delta$,
    \item $\Supp(\gamma)$ is the set of simple roots which have a nonzero coefficient in the expression of $\gamma \in \Phi$ as linear combination of simple roots,
    \item $B^-$ is the opposite Borel subgroup, with corresponding set of roots being $\Phi \backslash \Phi^+$,
    \item $U_\gamma$ ($\gamma \in \Phi)$ is the root subgroup associated to $\gamma$, with corresponding root homomorphism $u_\gamma \colon \Ga \stackrel{\sim}{\longrightarrow} U_\gamma$,
    \item $P^\alpha$ ($\alpha \in \Delta$) is the maximal reduced parabolic subgroup not containing $U_{-\alpha}$, which is generated by $B$ and $U_{-\beta}$ with $\beta \in \Delta \setminus \{ \alpha \}$,
    \item $F^m_G \colon G \longrightarrow G^{(m)}$ is the $m$-th iterated relative Frobenius homomorphism of $G$,
    \item $G_m \defeq \ker F^m_G$ is the $m$-th Frobenius kernel.
\end{itemize}
Let us recall that the morphism $F^m_G$ is an isogeny since it is surjective with finite kernel. Moreover, the map $\alpha \mapsto P^\alpha$ defines a bijection between simple roots and maximal reduced parabolic subgroups. More generally, under the assumptions of \cite{Wenzel}, there is a bijection
\begin{align}
\label{bijection}
\Hom_{\text{Set}} (\Delta, \N \cup \{ \infty \}) \longrightarrow \{ \text{parabolic subgroups } G \supset P \supset B \}
\end{align}
sending a function $\varphi \colon \Delta \rightarrow \N \cup \{ \infty\}$ to the subgroup scheme $P_\varphi$ defined by the intersection of all maximal reduced parabolics fattened by their corresponding Frobenius kernels
\[
P_\varphi \defeq \bigcap_{\alpha \in \Delta} G_{\varphi(\alpha)} P^\alpha = \bigcap_{\alpha \in \Delta \colon \varphi(\alpha) \neq \infty} G_{\varphi(\alpha)}P^\alpha.
\]
Let us recall that, given a parabolic subgroup $P$, there is always an associated function $\varphi \colon \Phi^+ \longrightarrow \N \cup \{ \infty \} $ (introduced in \cite{Wenzel}) given by the identity
\[
U_{-\gamma} \cap P = u_{-\gamma} (\alpha_{p^{\varphi(\gamma)}}), \quad \gamma \in \Phi^+,
\]
where $\alpha_{p^\infty}$ is understood to be $\Ga$. For example, the parabolic $G_m P^\alpha$ defines the function sending all positive roots to infinity, except for those containing $\alpha$ in their support, which assume value $m$.
\begin{theorem}[Theorem 10, \cite{Wenzel}]
     The parabolic subgroup $P$ is uniquely determined by the function $\varphi$, with no assumption on the characteristic or on the Dyinkin diagram of $G$.
\end{theorem}
Moreover, when $p\geq 5$ or $G$ is simply laced, the function $\varphi$ is itself uniquely determined by its values on $\Delta$ via the equality 
\[
\varphi(\gamma) = \min \{ \varphi (\alpha) \colon \alpha \in \Supp(\gamma) \},
\]
giving the bijection (\ref{bijection}). See \cite[Theorem $14$]{Wenzel} for more details. As we will see later, the last statement does not always hold in small characteristic.

The guiding idea is to mimic the known classification - written in terms of Frobenius kernels - replacing the Frobenius morphism with any noncentral isogeny (see \Cref{classification_rank1}). This motivates the preliminary study and classification of such homomorphisms.


\subsection{Classifying isogenies}

We now classify isogenies between simple algebraic groups, first recalling definitions and the Isogeny Theorem, then introducing the so-called \emph{very special isogeny} $\pi_G$, whose kernel is a certain subgroup of height one defined by short roots - which only exists when the Dynkin diagram has an edge of multiplicity equal to the characteristic - and concluding with the following factorisation result: see \Cref{factorisation_isogenies}.

\begin{proposition**}
Let $G$ be a simple and simply connected algebraic group over $k$. 
Let $f \colon G \rightarrow G^\prime$ be an isogeny. Then there exists a factorisation of $f$ as
\[
\begin{tikzcd}
f \colon G \arrow[rr, "\pi"] && \overline{G} \arrow[rr, "F_{\overline{G}}^{m}"] && (\overline{G})^{(m)} \arrow[rr, "\rho"]  && G^\prime,
\end{tikzcd}
\]
where $\rho$ is a central isogeny and $\pi$ is either the identity or - when the Dynkin diagram of $G$ has an edge of multiplicity $p$ - the very special isogeny $\pi_G$.
\end{proposition**}

\subsubsection{Preliminaries}

We shall start by reviewing what isogenies look like, in particular noncentral ones. First, let us recall some notations and the statement of the Isogeny Theorem, which is proved in detail in \cite{Steinberg99}.

\begin{definition}
Let $(G,T)$ and $(G^\prime,T^\prime)$ be reductive algebraic groups over $k$. An \emph{isogeny} between them is a surjective homomorphism of algebraic groups $f \colon G \rightarrow G^\prime$ having finite kernel, sending the maximal torus $T$ to the maximal torus $T^\prime$. The \emph{degree} of $f$ is the order of its kernel.
\end{definition}

Given an isogeny $f$, there is an induced map between the character groups
\[
\varphi \defeq X(f_{\vert T}) \colon X(T^\prime) \longrightarrow X(T), \quad \chi^\prime \longmapsto \chi^\prime \circ f_{\vert T},
\]
satisfying the conditions :
\begin{enumerate}[(i)]
    \item both $\varphi \colon X(T^\prime) \rightarrow X(T)$ and its dual $\varphi^\vee \colon X^\vee(T) \rightarrow X^\vee(T^\prime)$ are injective,
    \item there exists a bijection $\Phi \leftrightarrow \Phi^\prime$, denoted $\alpha \leftrightarrow \alpha^\prime$, and integers $q(\alpha)$ which are all powers of $p$, such that
\end{enumerate}
\[
\varphi(\alpha^\prime) = q(\alpha)\alpha \quad \text{and} \quad \varphi^\vee (\alpha^\vee) = q(\alpha) \alpha^\vee{}^\prime \quad \text{for all } \alpha \in \Phi.
\]

Geometrically, the integers $q(\alpha)$ arise as follows: the image $f(U_\alpha)$ is a smooth connected unipotent algebraic subgroup of $G^\prime$ which is normalized by $T^\prime$ and isomorphic to the additive group $\Ga$, hence it must be of the form $U_{\alpha^\prime}$ for a unique $\alpha^\prime \in \Phi^\prime$. This gives the bijection; then, using the $T$-action on those two root subgroups, one finds that there exists a constant $c_\alpha \in \Gm$ and an integer $q(\alpha) \in p^\N$ such that
\begin{align}
\label{def_integers}
f(u_\alpha(x)) = u_{\alpha^\prime} (c_\alpha x^{q(\alpha)})
\end{align}
for all $x \in \Ga$.

\begin{definition}
A homomorphism between character groups $\varphi \colon X(T^\prime) \rightarrow X(T)$ satisfying conditions (i) and (ii) is called an \emph{isogeny of root data}.
\end{definition}

\begin{theorem}[\emph{Isogeny Theorem}]
Let $(G,T)$ and $(G^\prime,T^\prime)$ be reductive algebraic groups over $k$. Assume given an isogeny of root data $\varphi  \colon X(T^\prime) \rightarrow X(T)$. Then there exists an isogeny $f \colon (G,T) \rightarrow (G^\prime,T^\prime)$ inducing $\varphi$. Moreover, $f$ is unique up to an inner automorphism $\inn(t)$ for some $t \in T^\prime/Z(G^\prime)$.
\end{theorem}

\begin{proof}
See \cite[1.5]{Steinberg99}.
\end{proof}

For instance, an important class of isogenies is given by the ones having central kernel, which are characterized by the fact that the associated integers $q(\alpha)$ are all equal to $1$: these are not interesting for our purpose of studying parabolic subgroups, since we may restrict ourselves in the classification to the case of a simply connected group (or an adjoint one, depending on the desired properties). The most known example of a noncentral isogeny is an iterated Frobenius homomorphism $F^m$, for which $\alpha^\prime = \alpha$ and all $q(\alpha)$ are equal to $p^m$. Do other isogenies exist? We shall now consider this question.

\subsubsection{Very special isogenies}
\label{subsection_N}

From now on we make the assumption that $G$ is simple. The Weyl group $W=W(G,T)$ acts on roots leaving the integer $q$ invariant: if the Dynkin diagram of $G$ is simply laced, then there is only one orbit, hence all $q(\alpha)$ must assume the same value. This means, by the Isogeny Theorem, that up to inner automorphisms the only noncentral isogenies with source $G$ are iterated Frobenius homomorphisms.

On the other hand, assume that the Dynkin diagram of $G$ has a multiple edge. In this setting, there are two distinct orbits under the action of the Weyl group, corresponding to long and short roots: this allows, considering an isogeny $f \colon (G,T) \longrightarrow (G^\prime,T^\prime)$, for two possibly distinct values of $q(\alpha)$. Let us denote as $\Phi_<$ and $\Phi_>$ the subsets of $\Phi$ consisting of short and long roots respectively, and denote the two integer values as
\begin{align}
\label{qminus}
    q_< \defeq q(\alpha) \, (\alpha \in \Phi_<) \quad \text{and} \quad q_> \defeq q(\alpha) \,  (\alpha \in \Phi_>).
\end{align}
Analogously, we fix the following notation for the direct sum of root spaces associated to roots of a fixed length:
\[
\mathfrak{g}_< \defeq \bigoplus_{\alpha \in \Phi_<} \mathfrak{g}_\alpha = \bigoplus_{\alpha \in \Phi_<} \Lie U_\alpha \quad \text{and} \quad 
\mathfrak{g}_> \defeq \bigoplus_{\alpha \in \Phi_>} \mathfrak{g}_\alpha = \bigoplus_{\alpha \in \Phi_>} \Lie U_\alpha.
\]

We now recall a notion introduced in \cite[Section $7.1$]{CGP15}, based on previous work from Borel and Tits, and some of its properties. Also, let us remark that the assumption we will make is stronger than just asking that the group is not simply laced: to define the following notions, the characteristic needs to be $p=2$ for types $B_n$, $C_n$ and $F_4$, and $p=3$ in type $G_2$. Equivalently, the group $G$ has Dynkin diagram having an edge of multiplicity $p$. From now on, we will call this the \emph{edge hypothesis}. The following result is \cite[Lemma 7.1.2]{CGP15}.

\begin{lemma}
\label{712_conrad}
Let $G$ be simply connected satisfying the edge hypothesis.
Then the vector subspace
\[
\mathfrak{n}_G \defeq \langle \Lie\gamma^\vee (\Gm) \colon \gamma \in \Phi_< \rangle \oplus \mathfrak{g}_<
\]
is a $p$-Lie ideal of $\Lie G$. Moreover, every nonzero $G$-submodule of $\Lie G$ distinct from $\Lie Z(G)$ contains $\mathfrak{n}_G$.
\end{lemma}
By the equivalence of categories between $p$-Lie subalgebras of $\Lie G$ and algebraic subgroups of $G$ of height one, the $p$-Lie ideal $\mathfrak{n}_G$ lifts to a unique normal subgroup of $G$.

\begin{definition}
Let $G$ be simply connected satisfying the edge hypothesis. The algebraic subgroup of height one having $\mathfrak{n}_G$ as Lie algebra is denoted as \emph{$N_G$}.
\end{definition}

In particular, $N_G$ is characterized by being the unique minimal noncentral normal subgroup of $G$ having trivial Frobenius. For more details see \cite[Definition $7.1.3$]{CGP15}. Thus, we are led to consider the homomorphism
\[
\pi_G \colon G \longrightarrow \overline{G} \defeq G/N_G.
\]
Let us remark that this is a noncentral isogeny with corresponding values $q_< =p$ and $q_> = 1$.

\begin{definition}
\label{veryspecial}
With the above notations, the homomorphism $\pi_G$ is called the \emph{very special isogeny} associated to the simple and simply connected algebraic group $G$.
\end{definition}

The following step towards a better understanding of isogenies is the natural generalization of the above notion to the non simply connected case.

\begin{definition}
\label{def_N}
Let $G$ be simple 
 satisfying the edge hypothesis and let $\psi \colon \widetilde{G} \longrightarrow G$ be its simply connected cover. Let $N_{\widetilde{G}}$ be the kernel of the very special isogeny of $\widetilde{G}$ defined just above. We denote as :
\begin{itemize}
    \item $N_G$ its schematic image via the central isogeny $\psi$ ;
    \item $N_{m,\widetilde{G}} \defeq \ker (\pi_{\widetilde{G}^{(m)}} \circ F_{\widetilde{G}}^m) = (F^m_{\widetilde{G}})^{-1}(N_{\widetilde{G}^{(m)}})$, for any $m \geq 1$ ;
    \item $N_{m,G}$ the schematic image of $N_{m,\widetilde{G}}$ via the central isogeny $\psi$.
\end{itemize}
\end{definition}

Let us remark that $N_G$ is nontrivial, noncentral, normal and has trivial Frobenius. Moreover, it is minimal with such properties: let $H \subset N_G$ be another such subgroup, then $\widetilde{H} \defeq \psi^{-1}(H) \cap N_{\widetilde{G}}$ is nontrivial, noncentral, normal and of height one, hence by definition contained in $N_{\widetilde{G}}$. This shows that $N_G = \psi(N_{\widetilde{G}}) \subset \psi(\widetilde{H}) = H$.\\
It is now natural to ask ourselves if such a subgroup is unique, or if we can give an example of it appearing in a natural context. This is shown in \Cref{N_unique} and \Cref{N_SO} below.

Up to this point in this section we have assumed that the Dynkin diagram of $G$ has an edge of multiplicity $p$. What about the other cases not satisfying the edge hypothesis, in particular those which are not treated in \cite{Wenzel}? Let us assume that either $p=3$ and that the group $G$ is simple of type $B_n$, $C_n$ or $F_4$, or that $p=2$ and the group $G$ is simple of type $G_2$. Then an analogous construction to the subgroup $N_G$ cannot be done for the following reason: nontrivial normal subgroups of height one correspond, under the equivalence of categories, to nonzero $p$-Lie ideals of $\Lie G$, which do not exist due to the following result (see \cite[4.4]{Strade}).

\begin{lemma}
\label{Lie_simple}
Let $p=3$ and $G$ be simple of type $B_n$, $C_n$ for some $n \geq 2$, or $F_4$, or let $p=2$ and $G$ simple of type $G_2$. Then $\Lie G$ is simple as a $p$-Lie algebra.
\end{lemma}

\subsubsection{Factorising isogenies} 
Let us start by recalling the following result concerning the factorisation of the Frobenius morphism (see \cite[Proposition $7.1.5$]{CGP15}) :

\begin{proposition}
\label{CGP_7.1.5}
Let $G$ be simple and simply connected 
satisfying the edge hypothesis. Then
\begin{enumerate}[(a)]
    \item There is a factorisation of the Frobenius morphism as
    \[
    \begin{tikzcd}
    F_G \colon (G \arrow[rr, "\pi_G"] && \overline{G} \arrow[rr ,"\overline{\pi}"] && G^{(1)})
    \end{tikzcd}
    \]
    which is the only nontrivial factorisation into isogenies with first step admitting no nontrivial factorisation into isogenies.
    \item The root system $\overline{\Phi}$ of $\overline{G}$ is isomorphic to the dual of the root system of $G$.
    \item The bijection between $\Phi$ and $\overline{\Phi}$ defined by $\pi_G$ exchanges long and short roots: denoting it as $\alpha \leftrightarrow \overline{\alpha}$, if $\alpha$ is long then $\overline{\alpha}$ is short and vice-versa.
    \item In the factorisation of point $(a)$, the map $\overline{\pi}$ is the very special isogeny of $\overline{G}$.
\end{enumerate}
\end{proposition}

In particular, the restriction $({\pi_G})_{\vert U_\alpha} \colon U_\alpha \rightarrow U_{\overline{\alpha}}$ gives an isomorphism whenever $\alpha$ is long and a purely inseparable isogeny of degree $p$ whenever $\alpha$ is short.

\begin{lemma}
\label{q_trivial}
Assume $f \colon G \rightarrow G^\prime$ is a noncentral isogeny with $G$ simply connected and 
satisfying the edge hypothesis. If at least one value of $q(\alpha)$ is equal to $1$, then necessarily $q_> = 1$.
\end{lemma}

\begin{proof}
Let us consider the Lie subalgebra $\Lie(\ker f) \subset \mathfrak{g}$. This is a proper $G$-submodule of the Lie algebra $\mathfrak{g}$ under the adjoint action, which is not contained in $\Lie Z(G)$: by \Cref{712_conrad}, $\Lie (\ker f)$ must contain all of $\mathfrak{g}_<$.
This means that if $\alpha$ is a short root, then $f_{\vert U_\alpha} \colon U_\alpha \rightarrow U_{\alpha^\prime}$ is not an isomorphism: in other words, $q_< \neq 1$.
\end{proof}

\begin{proposition}
\label{factorisation_isogenies}
Let $G$ be a simple and simply connected algebraic group and 
let $f \colon G \rightarrow G^\prime$ be an isogeny. Then there exists a unique factorisation of $f$ as
\[
\begin{tikzcd}
f \colon G \arrow[rr, "\pi"] && \overline{G} \arrow[rr, "F_{\overline{G}}^{m}"] && (\overline{G})^{(m)} \arrow[rr, "\rho"]  && G^\prime,
\end{tikzcd}
\]
where $\rho$ is a central isogeny and $\pi$ is either the identity or - 
when $G$ satisfies the edge hypothesis - the very special isogeny $\pi_G$.
\end{proposition}

\begin{proof}
Let us start by considering the bijection $\Phi \leftrightarrow \Phi^\prime$ and the corresponding integers $q(\alpha)$ associated to the isogeny $f$, as recalled in (\ref{def_integers}).\\
\textbf{Step $1$}: is the isogeny central? This is equivalent to asking whether all integers $q(\alpha)$ are equal to one. If this is the case, then we are done. Next, we will hence assume that at least one value of $q$ is nontrivial.\\
\textbf{Step $2$}: does $p$ divide $q(\alpha)$ for all roots $\alpha$? If the group is simply laced this is always the case, since 
$q$ is constant. If $p=3$ and the group is of type $B_n$, $C_n$ or $F_4$, or if $p=2$ and the group is of type $G_2$, this is also always the case: indeed, there exists at least one $\gamma \in \Phi$ such that $q(\gamma) \neq 1$. Equivalently, the corresponding root space satisfies $\mathfrak{g}_{\gamma} \subset \mathfrak{h} \defeq \Lie (\ker f)$. Since $\mathfrak{h}$ is a nontrivial $p$-Lie ideal of $\Lie G$, it must coincide with all of $\Lie G$ thanks to \Cref{Lie_simple}.\\
In general, if the answer is yes, then the root subspace $\mathfrak{g}_{\alpha}$ is contained in $\Lie (\ker f)$ for all roots. Since the latter is a Lie ideal of $\Lie G$, taking brackets implies that the copy of $\mathfrak{sl}_2$ associated to each root is also contained in $\Lie (\ker f)$, which thus coincides with $\Lie G$. This means in particular that the Frobenius kernel of $G$ is contained in the kernel of $f$, so we can factorise by the Frobenius morphism as follows
\[
\begin{tikzcd}
G \arrow[rrrr, bend left, "f"] \arrow[rr, "F_G"] && G^{(1)} \arrow[rr, "f^\prime"] && G^\prime
\end{tikzcd}
\]
and go back to Step $1$ replacing $f$ by $f^\prime$. Notice that this is possible, since the group $G^{(1)}$ is still simple and simply connected. Moreover, the new integers associated to the isogeny $f^\prime$ are exactly $q(\alpha)/p$, hence their values strictly decrease. After this step, we can hence assume that there are two distinct values $q_<$ and $q_>$ as defined in (\ref{qminus}). In particular, let us remark that in this case $G$ is not simply laced.\\
\textbf{Step $3$}: this step only occurs when the Dynkin diagram of $G$ has an edge of multiplicity $p$; moreover, by \Cref{q_trivial} $q_> = 1$ while $q_<$ is divisible by $p$. This last condition means that for any short root $\gamma$, the root subspaces $\mathfrak{g}_{\gamma}$ and $\mathfrak{g}_{-\gamma}$ are contained in $\Lie (\ker f)$. This implies that
\[
(\mathfrak{sl}_2)_{\gamma} = [\mathfrak{g}_{\gamma},\mathfrak{g}_{-\gamma}] \oplus \mathfrak{g}_{\gamma} \oplus \mathfrak{g}_{-\gamma} = \Lie (\gamma^\vee (\Gm)) \oplus \mathfrak{g}_{\gamma} \oplus \mathfrak{g}_{-\gamma} \subset \Lie (\ker f),
\]
hence, by definition of the subgroup $N_G$ in the simply connected case, we have
\[
\langle \Lie (\gamma^\vee (\Gm)) \colon \gamma \in \Phi_< \rangle \bigoplus_{\gamma \in \Phi_<} \mathfrak{g}_{\gamma} =: \Lie N_G \subset \Lie(\ker f).
\]
Since $N_G$ is of height one, this implies that $N_G \subset \ker f$, so we can factorise by the very special isogeny as follows
\[
\begin{tikzcd}
G \arrow[rrrr, bend left, "f"] \arrow[rr, "\pi_G"] && G^{(1)} \arrow[rr, "f^\prime"] && G^\prime
\end{tikzcd}
\]
and go back to Step $1$. Notice that this is possible, since by \Cref{CGP_7.1.5}, the group $\overline{G}$ is still simple and simply connected. Moreover, we know that the bijection $\Phi \leftrightarrow \overline{\Phi}$ exchanges long and short roots and that ${(\pi_G)}_{\vert U_\alpha}$ is an isomorphism for $\alpha$ long and of degree $p$ when $\alpha$ is short. By denoting as $q^\prime(-)$ the integers associated to the new isogeny $f^\prime$, we then have
\begin{align*}
(q^\prime)_< & = q^\prime(\overline{\alpha}) = q(\alpha) = q_> = 1, & (\alpha \text{ long})\\
(q^\prime)_> & = q^\prime (\overline{\alpha}) = q(\alpha)/p = q_</p, & (\alpha \text{ short})
\end{align*}
so the nontrivial integer strictly decreases after this step.\\
Following this procedure, one will necessarily factorise a finite number of times leading finally to a central isogeny, which is the $\rho$ given in the statement of the proposition. Moreover, we claim that the Frobenius morphism and the very special isogeny - when it is defined - commute, in the following sense: if $G$ is simple and simply connected, then
\[
\pi_{G^{(1)}} \circ F_G = F_{\overline{G}} \circ \pi_G.
\]
To prove this, let us apply the factorisation of the Frobenius morphism given in \Cref{CGP_7.1.5} twice to get
\[
\pi_{G^{(1)}} \circ F_G = \pi_{G^{(1)}} \circ (\pi_{\overline{G}} \circ \pi_G) = (\pi_{G^{(1)}} \circ \pi_{\overline{G}}) \circ \pi_G = F_{\overline{G}} \circ \pi_G.
\]
This means that we can commute $\pi$ with the Frobenius and assume that it is the first morphism (or the middle one, which gives another unique factorisation) in the expression $f = \rho \circ F^m \circ \pi$.
\end{proof}

\begin{remark}
\label{diagram_isogeny}
The above Proposition allows us to associate to any isogeny $f \colon G \rightarrow G^\prime$ between simple algebraic groups a diagram of the form
\[
\begin{tikzcd}
G \arrow[rr, "f"] && G^\prime \\
\widetilde{G} \arrow[rr, "F^m \circ \pi"] \arrow[u, "\psi"] && \widetilde{G^\prime} \arrow[u, "\rho"]
\end{tikzcd}
\]
where $\psi$ is the simply connected cover of $G$ and $\rho$ is central. In particular, notice that the group $\overline{(\widetilde{G})}{}^{(m)}$, which is the target of the morphism $F^m \circ \pi$, is simply connected and $\rho$ is central, thus this group is the simply connected cover of $G^\prime$.
\end{remark}

The first immediate consequence of this factorisation result is the uniqueness of the subgroup $N_G$.

\begin{lemma}
\label{N_unique}
Let $G$ be simple satisfying the edge hypothesis and $H \subset G$ a normal, noncentral subgroup of height one. Then $H$ contains the subgroup $N_G$. In particular, such a subgroup is unique.
\end{lemma}

\begin{proof}
The conclusion clearly holds when $H$ equals the Frobenius kernel of $G$, hence we can assume that $H \neq G_1$. To prove that $N_G \subset H$ it is enough to show that $f(N_G)$ is trivial, where $f$ is the isogeny $G \rightarrow G/H$. Consider the associated diagram given in Remark \ref{diagram_isogeny} :
\[
\begin{tikzcd}
G \arrow[rr, "f"] && G/H \\
\widetilde{G} \arrow[rr, "F^m \circ \pi"] \arrow[u, "\psi"] && \widetilde{G/H} \arrow[u, "\rho"]
\end{tikzcd}
\]
where $\pi$ is either the identity or the very special isogeny of $G$. We want to show that the bottom arrow is necessarily the very special isogeny $\pi_{\widetilde{G}}$. First, the subgroup $H$ is noncentral hence if $m=0$ then $\pi = \pi_{\widetilde{G}}$, otherwise the bottom row would be the identity and $f$ would be central. Moreover, $H \subsetneq G_1 = \ker(F \colon G \rightarrow G^{(1)})$ hence the factorisation of the isogeny $f \circ \psi$ given in \Cref{factorisation_isogenies} does not contain any Frobenius morphism~: this means that $m =0$ so necessarily $\pi$ is the very special isogeny of $\widetilde{G}$. Thus, we can conclude that $ f \circ \psi = \rho \circ \pi_{\widetilde{G}}$ and
\[
f(N_G) = f(\psi(N_{\widetilde{G}})) = \rho(\pi_{\widetilde{G}}(N_{\widetilde{G}})) = 1
\]
as wanted.
\end{proof}

\begin{example}
\label{N_SO}
Let us assume $p=2$ and consider the group $G = \SO_{2n+1}= \SO(k^{2n+1})$ in type $B_n$ with $n \geq 2$, defined as being relative to the quadratic form
\[
Q(x) = x_n^2 + \sum_{i=0}^{n-1} x_ix_{2n-i}
\]
and $G^\prime = \Sp_{2n} = \Sp(k^{2n})$ relative to the skew form 
\[
b(y,y^\prime) = \sum_{i=1}^n y_i y^\prime_{2n+1-i} -y_{2n+1-i}y^\prime_i.
 \]
 Since $G$ fixes the middle vector of the canonical basis $e_n$, it acts on $k^{2n} = k^{2n+1}/ke_n$ and this gives an isogeny 
\begin{align*}
 \varphi \colon G= \SO_{2n+1} \longrightarrow \Sp_{2n} = G^\prime,
\end{align*}
 of degree $2^{2n}$.
Since the target of the isogeny is already simply connected, the diagram of Remark \ref{diagram_isogeny} is as follows :
\[
\begin{tikzcd}
\SO_{2n+1} \arrow[rr, "\varphi"] && \Sp_{2n} \\
\Spin_{2n+1} \arrow[u, "\psi"] \arrow[rru, bend right, "F^m \circ \pi"]
\end{tikzcd}
\]
In particular, since $\psi$ is central - identifying the root systems of a group and of its simply connected cover - the integers $q(-)$ associated to the isogeny $\varphi$ must be the same as those associated to the composition $F^m \circ \pi$. In particular, this implies $m=0$; hence the subgroup 
\[
N_{\SO_{2n+1}} = \ker\varphi = \psi(\ker \pi) = \psi(N_{\Spin_{2n+1}}),
\]
which appears in this natural construction, coincides with the one just defined above. In particular, in this case
    \[
    \Lie N_{\Spin_{2n+1}} = \Lie(\varepsilon_n^\vee(\Gm)) \bigoplus_{\gamma \in \Phi_<} \mathfrak{g}_{\gamma} = \Lie (\varepsilon_n^\vee(\Gm)) \bigoplus_{1\leq i \leq n} \left( \mathfrak{g}_{-\varepsilon_i} \oplus \mathfrak{g}_{\varepsilon_i}. \right) 
    \]
To conclude this example, let us determine explicitly the subgroup $N_{\SO_{2n+1}}= \ker \varphi$ and its Lie algebra, which will be needed later on. A matrix in $\ker \varphi$ is of the form
\[
A = \begin{pmatrix}
& a_0 & \\
1_n & \vdots & 0_n\\
& a_{n-1} &\\
b_0\ldots b_{n-1} & a_n & b_{n+1} \ldots b_{2n}\\
& a_{n+1} &\\
0_n & \vdots & 1_n\\
& a_{2n} &
\end{pmatrix}
\]
and the condition for $A$ to be in $\SO_{2n+1}$ gives
\begin{align*}
& Q(Ax) = a_n^2 x_n + \sum_{j\neq n}b_j^2 x_j^2 +\sum_{i=0}^{n-1} \left( x_ix_{2n-i} + a_{2n-i}x_ix_n +a_i x_{2n-i}x_n \right) \\
= \, & Q(x) = x_n^2 + \sum_{i=0}^{n-1} x_ix_{2n-i},
\end{align*}
which is equivalent to $a_i=0$ for all $i \neq n$, $a_n^2=1$ and $b_i^2= 0$ for all $i$. Moreover, under these conditions $\det A = a_n = 1$, thus we have 
\[
N_{\SO_{2n+1}} = \ker \varphi = \left\{ 
\begin{pmatrix}
1 & 0 & 0\\
b_0\ldots b_{n-1} & 1 & b_{n+1}\ldots b_{2n}\\
0 & 0 & 1
\end{pmatrix} \in \GL_{2n+1} \colon b_i \in \alpha_p \right\} \simeq \alpha_p^{2n}.
\]
Finally, using the equalities in \Cref{roots_B} concerning short roots, we can conclude that $\Lie N_{\SO_{2n+1}} = \mathfrak{g}_<$.
\end{example}

\section{Case of Picard rank one}
\label{sec2}

Let us recall that we are working with a semisimple algebraic group $G \supset B \supset T$ over an algebraically closed field $k$ of characteristic $p>0$, together with a fixed Borel subgroup and a maximal torus contained in it. Our aim is to prove that all projective homogeneous varieties under a $G$-action having Picard group of rank one are isomorphic (as varieties) to homogeneous spaces having reduced stabilizers, in every type except $G_2$ when the characteristic is $p=2$.
Let us remark that, since the Picard rank of $X= G/P$ is equal to the number of simple roots of $G$ not contained in the root system of a Levi subgroup of $P$,
such spaces are realized as quotients $G/P$ such that the reduced subgroup of the stabilizer $P$ is maximal. For a full justification of this assertion, see \Cref{sec:divisors}.\\
The main result is the following :
\begin{theorem}
\label{final_rank1}
Let $X$ be a projective algebraic variety over an algebraically closed field of characteristic $p>0$, homogeneous under a faithful action of a smooth connected algebraic group $H$ and having Picard group isomorphic to $\Z$.\\
Then there is a simple adjoint algebraic group $G$ and a \emph{reduced} maximal parabolic subgroup $P \subset G$ such that $X = G/P$, unless $p=2$ and $H$ is of type $G_2$.
\end{theorem}

The purpose of this Section is to prove the above Theorem: the idea is to do it explicitly case by case, since there seems to be no easy general geometric argument, as the case of type $G_2$ in characteristic two confirms. We proceed as follows: in \Cref{subsec:reductions} we perform elementary reductions to the case where $X= G/P$ with $G$ simple and the characteristic is $2$ or $3$, and we recall some notation and results used in the proof. In \Cref{type_A} we illustrate the strategy of the proof in the simplest case of type $A_{n-1}$. In Sections \ref{sectionCnPn} to \ref{F4} we implement the argument in types $B_n$, $C_n$ and $F_4$. The case of $G_2$, for which the above Theorem fails in characteristic $2$, is then studied separately in \Cref{subsec:G2}.

\subsection{Reductions and notation}
\label{subsec:reductions}

Let us place ourselves under the hypothesis of \Cref{final_rank1} and denote as $H_{\text{aff}}$ the largest connected affine normal subgroup of $H$. By \cite[Theorem 4.1.1]{verde}, there is a canonical isomorphism $X \simeq A \times Y$, where $A$ is an abelian variety and $Y$ is a projective homogeneous variety under a faithful $H_{\text{aff}}$ action. Moreover, $H_{\text{aff}}$ is semisimple and of adjoint type. Under our assumptions, the abelian variety must be a point because otherwise the Picard group of $X$ would not be discrete; more precisely, the hypothesis $\Pic X = \Z$ implies - by the combinatorial description of the Białynicki-Birula decomposition of homogeneous spaces given in \Cref{BB_flag} - that we can assume $H$ to be simple.

After such reductions, it is thus enough to prove the following statement.

\begin{theorem}
\label{final_rank1_weak}
    Let $G$ be a simple adjoint group, not of type $G_2$ when the characteristic is $2$, and $P$ a parabolic subgroup such that $P_{\text{red}}$ is maximal. If $G$ acts faithfully on $X= G/P$, then $P$ is a \emph{reduced} parabolic subgroup.
\end{theorem}
Let us keep notations from Subsection \ref{subsection_notations} and recall for reference the statement of \cite[Theorem $14$]{Wenzel}.

\begin{theorem}
\label{Wenzel}
There is an injective map
\begin{align*}
    \Hom_{\text{Set}} (\Delta, \N \cup \{ \infty \}) & \longrightarrow \{ \text{parabolic subgroups } G \supset P \supset B \}\\
    \varphi & \longmapsto \bigcap_{\alpha \in \Delta \colon \varphi(\alpha) \neq \infty} G_{\varphi(\alpha)}P^\alpha.
\end{align*}
Moreover, if $p \geq 5$ or the Dynkin diagram of $G$ is simply laced, this map is also surjective.
\end{theorem}

\begin{remark}
\label{faithful_action}
Let us start by taking a projective variety $X$ which is homogeneous under the action of a simple group $H$. By replacing such a group with the image $G$ of the morphism $H \rightarrow \underline{\Aut}_X$ (see \Cref{autX} concerning the notation on automorphism groups) we may assume that the action is faithful. 
In particular, this means that there is no normal algebraic subgroup of $G$ contained in $P$. However, we need to be careful in the case-by-case proof because this additional assumption - which is not restrictive on the varieties considered - forces the group $G$ to be of adjoint type.
\end{remark}
Let us place ourselves in the setting of \Cref{final_rank1_weak} and sketch the strategy of the proof: let $P$ be a nonreduced parabolic subgroup such that
\[
P_{\text{red}}= P^\alpha
\]
for some simple positive root $\alpha \in \Delta$; consider $P^\alpha \subsetneq P \subset G $, inducing the corresponding inclusions on Lie algebras: \[
\Lie P^{\alpha} \subsetneq \Lie P \subset \Lie G.
\]
Since we do not have any information a priori on $P$, we  study the quotient 
\[
V_\alpha \defeq \Lie G / \Lie P^\alpha,
\]
considered as a $L^\alpha$-module under the representation given by the adjoint action, where $L^\alpha$ denotes the Levi subgroup defined as the intersection $P^\alpha \cap (P^\alpha)^-$ with the corresponding opposite parabolic subgroup.\\
Let us fix some notation and state a Lemma on structure constants which will be repeatedly used in what follows :

\begin{itemize}
    \item the decomposition of the Lie algebra in weight spaces under the $T$-action is
    \[
    \mathfrak{g} = \Lie G = \Lie T \bigoplus_{\gamma \in \Phi}  \mathfrak{g}_{\gamma}, 
    \]
    \item when $G$ is simply connected, a Chevalley basis of $\Lie G$ is denoted as $\{ X_\gamma, H_\alpha \}_{\gamma \in \Phi, \alpha\in\Delta}$.
     \end{itemize}
    In particular, $\mathfrak{g}_\gamma = \Lie U_\gamma = k X_\gamma$ and $X_\gamma = \dd u_\gamma (1)$, where $u_\gamma$ is the root homomorphism $\Ga \stackrel{\sim}{\longrightarrow} U_\gamma$. Whenever the Dynkin diagram of $G$ is not simply laced, 
    \begin{itemize} 
    \item $\Phi_< \subset \Phi$ and $\Phi_> \subset \Phi$ denote respectively the subsets of short and long roots, whenever a multiple edge appears in the Dynkin diagram,
    \item when $G$ is simply connected 
    and satisfies the edge hypothesis (see \Cref{subsection_N}), $N_G$ denotes the finite group scheme of height one whose Lie algebra is given by
    \[
    \Lie N_G = \langle \, \Lie (\gamma^\vee(\Gm)) \colon \gamma \in \Phi_< \, \rangle  \bigoplus_{\gamma \in \Phi_<} \mathfrak{g}_\gamma,
    \]
    as recalled in \Cref{def_N}.
    \item when $G$ is not simply connected and satisfies the edge hypothesis, $N_G$ denotes the schematic image of $N_{\widetilde{G}}$ via the universal covering map, where $\widetilde{G}$ is the simply connected cover of $G$ - see again \Cref{def_N}.
 \end{itemize}

 Let us recall the following Lemma - see \cite[Chapter VII, $25.2$]{Humphreys} - which allows us to calculate all structure constants with respect to a Chevalley basis of the Lie algebra $\Lie G$, where $G$ is simple and simply connected. 

\begin{lemma}[Chevalley]
\label{lemmaChevalley}
Let $\{ X_\gamma\colon \gamma \in \Phi, \, H_\alpha \colon \alpha \in \Delta\}$ be a Chevalley basis for $\Lie G$, where $G$ is simple and simply connected. Then the resulting structure constants satisfy
\begin{enumerate}[(a)]
    \item $[H_{\alpha},H_{\beta}]= 0$ for all $\alpha,\beta \in \Delta$ ;
    \item $[H_\alpha,X_\gamma] = \langle \alpha, \gamma \rangle X_\gamma $ for all $\alpha \in \Delta$, $\gamma \in \Phi$ ;
    \item $[X_{-\gamma}, X_{\gamma}]$ is a linear combination with integer coefficients of the $H_\alpha$'s ;
    \item $[X_\gamma,X_\delta] = \pm (r+1)X_{\gamma+\delta}$ for all $\delta \neq \pm \gamma$ roots such that the $\delta$-string through $\gamma$ goes from $\gamma-r\delta$ to $\gamma+q\delta$ with $q \geq 1$, i.e. such that $\gamma + \delta$ is still a root ;
    \item $[X_\gamma,X_\delta] = 0$ for all roots $\delta \neq \pm \gamma$ such that $\gamma+\delta$ is not a root.
\end{enumerate}
\end{lemma}


In particular, the Chevalley relation we use the most frequently is $(d)$: it is important to recall that structure constants appearing in such equations are among $\pm 1,\pm 2,\pm 3,\pm 4$, which indicates why problems arise in characteristic $2$ and $3$.\\

The main line of argument to prove \Cref{final_rank1_weak} is the following: we start by considering $X=G/P$ with $G$ adjoint acting faithfully and $P$ nonreduced. Then with some computation on Lie algebras, we show that  - when it is defined - $N_G \subset P$, while otherwise $G_1 \subset P$. In both cases this gives a normal algebraic subgroup of $G$ contained in the stabilizer $P$, which cannot exist due to Remark \ref{faithful_action}.

\subsection{Type $A_{n-1}$}
\label{type_A}

We start with a case whose classification is already covered by \cite{Wenzel} - without needing any assumption on the characteristic of the base field - but which is useful in order to explain the approach used in the other cases below.\\
Let us consider the reductive group $G= \GL_n$ in type $A_{n-1}$, its maximal torus $T$ given by diagonal matrices of the form
\[
t  = \diag(t_1,\ldots,t_n) \in \GL_n
\]
and the Borel subgroup $B$ of upper triangular matrices. Let us denote as $\varepsilon_i \in X(T)$ the character sending $t \mapsto t_i$, for $i = 1, \ldots, n$. Then the root system $\Phi = \Phi(G,T)$ is given by 
\[
\Phi^+ = \{ \varepsilon_i - \varepsilon_j, \, 1 \leq i < j \leq n \},
\]
with basis $\Delta$ consisting of the following roots :
\[
\alpha_1 = \varepsilon_1 - \varepsilon_2, \, \ldots, \, \alpha_{n-1} = \varepsilon_{n-1} - \varepsilon_n.
\]
Finally, assume given a nonreduced parabolic subgroup $P$ such that $P_{\text{red}}= P_m$, where $P_m \defeq P^{\alpha_m}$ denotes the maximal reduced parabolic subgroup associated to the simple positive root $\alpha_m$ for a fixed $1 \leq m < n$. Thus, the Levi subgroup $L_m$ of this reduced parabolic subgroup is a product of a reductive group of type $A_{m-1}$ and one of type $A_{n-m-1}$ :
\[
L_m = \left\{
\begin{pmatrix}
\, \ast \,  & 0 \\ 0 & \,  \ast \,
\end{pmatrix} \right \} \simeq \GL_m \times \GL_{n-m},
\]
and the two factors have as basis of simple roots $\{ \alpha_1, \ldots, \alpha_{m-1}\}$ and $\{ \alpha_{m+1} ,\ldots , \alpha_{n-1}\}$ respectively.\\
Now, let us consider the vector space $V_m = \Lie G/ \Lie P_m$. Since
\[
\{ \gamma \in \Phi^+ \colon \alpha_m \in \Supp(\gamma) \}= \{ \varepsilon_i - \varepsilon_j, \, i \leq m <j\},
\]
the root spaces in $V_m$ are of the form $\mathfrak{g}_{-\varepsilon_i+\varepsilon_j} = kE_{ji}$, for $i \leq m < j$, where $E_{ji}$ denotes the square matrix of order $n$ having all zero entries except the $(j,i)$-th entry which is equal to $1$. Concretely, $V_m$ consists as $L_m$-module of all matrices $M$ of size $(n-m) \times m$. The action of $L_m$ on $V_m$ is given by
\[
(A,B) \cdot M = \begin{pmatrix}
A & 0 \\ 0 & B
\end{pmatrix} \begin{pmatrix}
0 & 0 \\ M & 0
\end{pmatrix} \begin{pmatrix}
A^{-1} & 0 \\ 0 & B^{-1}
\end{pmatrix} = \begin{pmatrix}
0 & 0 \\ BMA^{-1} & 0
\end{pmatrix} = BMA^{-1},
\]
for all $A \in \GL_m, \, B \in \GL_{n-m}$. This just corresponds to the natural action of $\GL_m \times \GL_{n-m}$ on $\Hom_k(k^m,k^{n-m})$. In particular, $V_m$ is an irreducible $L_m$-module.\\
Since $\Lie P / \Lie P_m$ is an $L_m$-submodule of $V_m$ and we assumed $P$ to be nonreduced, this implies $\Lie P = \Lie G$ hence $G_1 \subset P$.
Under our assumptions, by Remark \ref{faithful_action} we get a contradiction. In other words, under the hypothesis of maximality of the reduced subgroup, we find that there are no new varieties other than those of the known classification. In the following subsections we will treat the other cases - not included in Wenzel's article - where two different root lengths are involved.

\begin{remark}
    What does this case correspond to, geometrically, on the level of varieties? We know by \cite{Wenzel} that $P_{\text{red}}=P^{\alpha_m}$ implies $P = G_r P^{\alpha_m}$ for some $r \geq 0$, hence 
    \[
    X = G/G_rP^{\alpha_m} \simeq G^{(r)}/ (P^{\alpha_m})^{(r)} \simeq G/P^{\alpha_m} = \Grass_{m,n}
    \]
    is isomorphic to the Grassmannian of $m$-th dimensional vector subspaces in $k^n$, equipped with the natural $G=\GL_n$-action, twisted by the $r$-th iterated Frobenius morphism. In particular, assuming faithfulness of the action implies $r=0$.
\end{remark}

\subsection{Type $C_n$}
\label{sectionCnPn}

Let us consider the group $\widetilde{G}= \Sp_{2n}$ in type $C_n$, with $n \geq 2$ in characteristic $p=2$ or $3$.
Defining $\widetilde{G}$ as relative to the skew form 
$b(x,y) = \sum_{i=1}^n x_i y_{2n+1-i} - x_{2n+1-i}y_i$ on $k^{2n}$, one has
\[
\widetilde{G} = \left\{ X \in \GL_{2n} \colon \, {}^tX \begin{pmatrix}
0 & \Omega_n \\ -\Omega_n & 0 
\end{pmatrix} X = \begin{pmatrix}
0 & \Omega_n \\ -\Omega_n & 0 
\end{pmatrix} \right\}, \quad \text{where } \Omega_n = \begin{pmatrix}
0 & 0 & 1\\
0 & \iddots & 0\\
1 & 0 & 0
\end{pmatrix}.
\]
Deriving this condition gives as Lie algebra
\begin{align*}
\Lie \widetilde{G} & = \left\{ M \in \mathfrak{gl}_{2n} \colon \, ^tM \begin{pmatrix}
0 & \Omega_n \\ -\Omega_n & 0 \end{pmatrix} + \begin{pmatrix}
0 & \Omega_n \\ -\Omega_n & 0  \end{pmatrix} M = 0 \right\}\\
& = \left\{ \begin{pmatrix}
A & B \\ C & -A^\sharp
\end{pmatrix} \in \mathfrak{gl}_{2n} \colon \, B = B^\sharp \text{ and } C= C^\sharp \right\},
\end{align*}
where for any square matrix $X$ of order $n$ we denote as $X^\sharp$ the matrix $\Omega_n {}^t X \Omega_n$, i.e.
\begin{align}
\label{sharp}
    (X^\sharp)_{i,j}= X_{n+1-j, n+1-i}.
\end{align}

\begin{remark}
\label{roots_C}
Next, let us consider as maximal torus $T$ the one given by diagonal matrices of the form
\[ t = 
\diag(t_1,\ldots,t_n,t_n^{-1},\ldots, t_1^{-1}) \in \GL_{2n}
\]
and denote as $\varepsilon_i \in X^\ast(T)$ the character sending $t \mapsto t_i$, for $i = 1, \ldots, n$. A direct computation gives the following root spaces in $\Lie \widetilde{G}$:
\begin{align*}
 & \mathfrak{g}_{2\varepsilon_i} = k \begin{pmatrix}
 0 & E_{i,n+1-i} \\0 & 0 \end{pmatrix} = k \begin{pmatrix}
0 & E_{ii}\Omega_n \\0 & 0 \end{pmatrix} &\\
& \mathfrak{g}_{-2\varepsilon_i} = k \begin{pmatrix}
0 & 0 \\E_{n+1-i,i} & 0 \end{pmatrix} = k \begin{pmatrix}
0 & 0 \\ \Omega_n E_{ii} & 0 \end{pmatrix}, & 1\leq i \leq n,\\
     & \mathfrak{g}_{\varepsilon_i+\varepsilon_j} = k \begin{pmatrix}
 0 & E_{i,n+1-j}+E_{j,n+1-i} \\ 0 & 0 \end{pmatrix} = k \begin{pmatrix}
 0 & (E_{ij}+E_{ji})\Omega_n \\ 0 & 0 \end{pmatrix}, &\\
& \mathfrak{g}_{-\varepsilon_i-\varepsilon_j} = k \begin{pmatrix}
0 & 0 \\ E_{n+1-i,j}+E_{n+1-j,i} & 0 \end{pmatrix}  = k \begin{pmatrix}
0 & 0 \\ \Omega_n(E_{ij}+E_{ji}) & 0 \end{pmatrix}, & i <j,\\
     & \mathfrak{g}_{\varepsilon_i-\varepsilon_j} = k \begin{pmatrix}
E_{ij} & 0 \\ 0 & -E_{n+1-j,n+1-i} \end{pmatrix}  = k \begin{pmatrix}
E_{ij} & 0 \\ 0 & - E_{ij}^\sharp \end{pmatrix}, &\\
& \mathfrak{g}_{-\varepsilon_i+\varepsilon_j} = k \begin{pmatrix}
E_{ji} & 0 \\ 0 & -E_{n+1-i,n+1-j} \end{pmatrix} k \begin{pmatrix}
E_{ji} & 0 \\ 0 & -E_{ji}^\sharp \end{pmatrix}, & i <j,
\end{align*}
where $E_{ij}$ denotes the square matrix of order $n$ with zero entries except for the $(i,j)$-th which is equal to one.
\end{remark}

The root system $\Phi = \Phi(\widetilde{G},T)$ is thus indeed
\[
\Phi^+ = \{ \varepsilon_i - \varepsilon_j, \, \varepsilon_i + \varepsilon_j, \, 1 \leq i < j \leq n \} \cup \{ 2\varepsilon_i, \, 1 \leq i \leq n \},
\]
having chosen as Borel subgroup the one given by all upper triangular matrices in $\widetilde{G} \subset \GL_{2n}$. The corresponding basis $\Delta$ consists of the following roots :
\begin{align}
\label{roots_Cn}
\alpha_1 = \varepsilon_1 - \varepsilon_2, \, \ldots, \, \alpha_{n-1} = \varepsilon_{n-1} - \varepsilon_n, \, \alpha_n = 2\varepsilon_n.
\end{align}

\subsubsection{Reduced parabolic $P_n$}
Still considering the group $\widetilde{G} = \Sp_{2n}$, denote as $P_n$ the maximal reduced parabolic subgroup associated to the long simple positive root $\alpha_n$: in a more intrinsic way, this subgroup is the stabilizer of an isotropic vector subspace $W \subset V$ of dimension $n$, where $\widetilde{G} = \Sp(V)$. In particular, $W$ is the span of $e_1,\ldots,e_n$, where $(e_i)_{i=1}^{2n}$ denotes the standard basis of $k^{2n}$. Moreover, let us denote as $P_n^-$ the opposite parabolic subgroup and as $L_n$ their common Levi subgroup, so that
\begin{align*}
    & P_n = \Stab(W \subset V)\\
    & P_n^-  = \Stab(W^\ast \subset V)\\
    & L_n =  P_n \cap P_n^- = \GL(W) \simeq \GL_n,
\end{align*}
where $W \oplus W^\ast = V$. Let us also remark that $L$ has root system $\Psi$ given by 
\[
\Psi^+ = \{ \varepsilon_i - \varepsilon_j, \, 1 \leq i < j \leq n \},
\]
corresponding to a reductive group of type $A_{n-1}$ having as basis $\alpha_1, \ldots , \alpha_{n-1}$. 
This can be visualized in the following block decomposition :
\[
L_n = \left\{
\begin{pmatrix}
A & 0 \\ 0 & -(A^{-1})^\sharp
\end{pmatrix} 
\colon \, A \in \GL(W) \simeq \GL_n
\right\} \subset \widetilde{G}.
\]
First, the Lie algebra of $P_n$ is 
\[
\Lie P_n = \Lie B \, \bigoplus_{i <j} \mathfrak{g}_{-\varepsilon_i + \varepsilon_j} = \bigoplus_{i<j} \left( \mathfrak{g}_{\varepsilon_i - \varepsilon_j} \oplus \mathfrak{g}_{-\varepsilon_i + \varepsilon_j} \right) \bigoplus_{i<j} \mathfrak{g}_{\varepsilon_i + \varepsilon_j} \bigoplus_i \mathfrak{g}_{2\varepsilon_i}.
\]
For our purposes it is useful to study the $L_n$-action on the vector space
\[
V_n \defeq \Lie \widetilde{G} / \Lie P_n = \bigoplus_{i<j} \mathfrak{g}_{-\varepsilon_i - \varepsilon_j} \bigoplus_i \mathfrak{g}_{-2\varepsilon_i}.
\]

\begin{lemma}
\label{lem:representation_Cn}
    The $L_n$-module $V_n$ is isomorphic to the dual of the standard representation of $\GL_n$ on $\Sym^2(k^n)$.
\end{lemma}

\begin{proof}
Indeed, the root spaces we are interested in have been computed in \Cref{roots_C}. Those equalities imply that a matrix in $V_n$ is of the form 
\[
\begin{pmatrix}
0 & 0 \\ \Omega_n X & 0
\end{pmatrix}, \quad \text{with } X \in \Sym^2(k^n),
\]
thus the dual action of $A \in \GL_n \simeq L_n$ can be computed as follows:
\[
{}^t\!A^{-1} \cdot X \simeq  \begin{pmatrix}
{}^t\!A^{-1} & 0 \\ 0 & -({}^t\!A)^\sharp
\end{pmatrix} 
\begin{pmatrix}
0 & 0 \\ \Omega_n X & 0
\end{pmatrix} \begin{pmatrix}
{}^t\!A & 0 \\ 0 & -({}^t\!A^{-1})^\sharp
\end{pmatrix} =
\begin{pmatrix}
0 & 0 \\ -\Omega_n AX {}^t\! A & 0
\end{pmatrix} \simeq A X \, {}^t\!A.
\]
This gives the desired isomorphism between the two $\GL_n$-modules.\\
Let us remark that if we are working over a field of characteristic $p=2$, the $L_n$-module $V_n$ contains a simple $L_n$-submodule, namely 
\[
\left\{ \begin{pmatrix}
    &&&&\\
    & 0 && 0&\\
    && c_1 &&\\
    & \iddots &&&\\
    c_n &&& 0& 
\end{pmatrix}, c_i \in k \right\}= \bigoplus_{i=1}^n \mathfrak{g}_{-2\varepsilon_i},
\]
which is isomorphic to the dual of the standard representation of $\GL_n$ on $k^n$, twisted once by the Frobenius morphism.
\end{proof}

\begin{proposition}
\label{LieN}
Assume given a nonreduced parabolic subgroup $P$ such that $P_{\text{red}}= P_n$. Then $\Lie P = \Lie \widetilde{G}$ or $\Lie P = \Lie P_n + \mathfrak{g}_<$. If $p=3$, then necessarily $\Lie P = \Lie \widetilde{G}$.
\end{proposition}

\begin{proof}
Let us consider the nonzero vector space $\Lie P / \Lie P_n$, which is an $L_n$-submodule of $V_n$. The latter being isomorphic to $\Sym^2(k^n)^\ast$ by  \Cref{lem:representation_Cn}, we have that
\begin{enumerate}[(a)]
    \item either $\Lie P / \Lie P_n$ contains all of the weight spaces $\mathfrak{g}_{-2\varepsilon_i}$ associated to long negative roots, 
    \item or it does not contain any of them.
\end{enumerate}

Let us start by (a) and assume $\mathfrak{g}_{-2\varepsilon_i} \subset \Lie P$ for all $i$. In order to prove that $\Lie P = \Lie \widetilde{G}$, it is enough to show that for any $i <j$, the Chevalley vector $X_{-\varepsilon_i - \varepsilon_j}$ also belongs to $\Lie P$. For this, let us consider roots
\begin{align*}
    \gamma = \varepsilon_i-\varepsilon_j, \quad & \text{satisfying } X_\gamma \in \Lie L_n \subset \Lie P,\\
    \delta= -2\varepsilon_i, \quad & \text{satisfying } X_\delta \in \Lie P \text{ by our last assumption. }
\end{align*}
Thus, $\gamma + \delta = -\varepsilon_i - \varepsilon_j$ is still a root while $\delta - \gamma = -3\varepsilon_i -\varepsilon_j$ is not: applying \Cref{lemmaChevalley} gives 
\[
[X_{\varepsilon_i-\varepsilon_j}, X_{-2\varepsilon_i}] = \pm X_{-\varepsilon_i-\varepsilon_j} \in \Lie P
\]
as wanted.

Let us place ourselves in the hypothesis of (b) and assume that no root subspace associated to a negative long root is in $\Lie P$. Since by assumption $ P$ is nonreduced, $\Lie P_n \subsetneq \Lie P$ so there must be at least one short root of the form $-\varepsilon_i-\varepsilon_j$ satisfying $X_{-\varepsilon_i -\varepsilon_j} \in \Lie P$. We will now prove that this implies all short roots $-\varepsilon_l-\varepsilon_m$ for $l<m$ belong to $\Lie P$, hence showing $\Lie P = \Lie P_n + \mathfrak{g}_<$.\\
First, assume $l \neq i,j$ and consider roots
\begin{align*}
    \gamma = -\varepsilon_i-\varepsilon_j, \quad & \text{satisfying } X_\gamma \in \Lie P  \text{ by assumption, }\\
    \delta = -\varepsilon_l + \varepsilon_i, \quad & \text{satisfying } X_\delta \in \Lie L_n \subset \Lie P.
\end{align*}
In this case, $\gamma + \delta = -\varepsilon_l -\varepsilon_j$ is still a root while $\delta -\gamma = -\varepsilon_l +2\varepsilon_i + \varepsilon_j$ is not: applying \Cref{lemmaChevalley} gives
\[
[X_{-\varepsilon_i-\varepsilon_j},X_{-\varepsilon_l+\varepsilon_i}] = \pm X_{-\varepsilon_l -\varepsilon_j} \in \Lie P.
\]
Now, let us fix any $l<m$ satisfying $l,m \neq j$ and consider roots 
\begin{align*}
    \gamma = \varepsilon_j-\varepsilon_m, \quad & \text{satisfying } X_\gamma \in \Lie L_n \subset \Lie P,\\
    \delta = -\varepsilon_l - \varepsilon_j, \quad & \text{satisfying } X_\delta \in \Lie P \text{ by the last step}.
\end{align*}
Thus, $\gamma + \delta = -\varepsilon_l-\varepsilon_m$ is still a root while $\delta -\gamma = -\varepsilon_l-2\varepsilon_j + \varepsilon_m$ is not: applying \Cref{lemmaChevalley} gives
\[
[X_{\varepsilon_j-\varepsilon_m},X_{-\varepsilon_l-\varepsilon_j}] = \pm X_{-\varepsilon_l-\varepsilon_m} \in \Lie P.
\]

If we are working over a field of characteristic $p=3$, the representation of $\GL_n$ acting on $\Sym^2(k^n)$ is already an irreducible one: this means that $V_n$ is an irreducible $L_n$-module. Hence the nonzero submodule $\Lie P / \Lie P_n$ must coincide with all of $V_n$; equivalently, $\Lie P = \Lie \widetilde{G}$ as wanted.
\end{proof}

\begin{proof}(\textbf{of \Cref{final_rank1_weak} in type $C_n$ when $P_{\text{red}}= P_n$})\\
Let $G$ be simple adjoint of type $C_n$ and $X=G/P$ with a faithful $G$-action such that $P_{\text{red}} = P^{\alpha_n}$ and $P$ is nonreduced. Define $\widetilde{P} \subset \widetilde{G} = \Sp_{2n}$ as being the preimage of $P$ in the simply connected cover: it is a nonreduced parabolic subgroup satisfying ${\widetilde{P}}_{\text{red}}= P_n$. When $p=2$, the above Proposition implies that 
\[
\langle \, \Lie (\gamma^\vee(\Gm)) \colon \gamma \in \Phi_< \, \rangle  \oplus \mathfrak{g}_< = \Lie N_{\widetilde{G}} \subset \Lie \widetilde{P},
\]
hence by considering the image in the adjoint quotient we get $N_G \subset P$, which is a contradiction by Remark \ref{faithful_action}. If $p=3$ then the above Proposition implies that $\Lie \widetilde{P} = \Lie \widetilde{G}$, hence the Frobenius kernel satisfies $\widetilde{G}_1 \subset \widetilde{P}$, and its image in the adjoint quotient is a normal subgroup of $G$ contained in $P$, which gives again a contradiction. Therefore in both cases $P$ must be a smooth parabolic.
\end{proof}


\subsubsection{Reduced parabolic $P_m$, $m<n$}

Let us consider again a $k$-vector space $V$ of dimension $2n$ and denote as $\widetilde{G}$ the group $\Sp_{2n}=\Sp(V)$, of type $C_n$ with $n \geq 2$ and $k$ of characteristic $p=2$ or $3$. Its root system has been recalled in (\ref{roots_Cn}). Let us fix an integer $1 \leq m < n$ and consider - keeping the notation recalled at the beginning of this subsection - the maximal reduced parabolic 
\[
P_m \defeq P^{\alpha_m},
\]
associated to the short simple root $\alpha_m$, which is the subgroup scheme stabilizing an isotropic vector subspace of dimension $m$: let us denote the latter as $W$. Then, $P_m$ also stabilizes its orthogonal with respect to the symplectic form on $V$: denoting as $P_m^-$ the opposite parabolic subgroup and as $L_m$ their common Levi subgroup, one finds
\begin{align*}
    P_m & = \Stab(W \subset W^\perp \subset V) = \Stab (W \subset W \oplus U \subset V)\\
    P_m^- &  = \Stab(W^\ast \subset (W^\ast)^\perp \subset V) = \Stab(W \subset W^\ast \oplus U \subset V)\\
    L_m & =  P_m \cap P_m^- = \GL(W) \times \Sp(U) \simeq \GL_m \times \Sp_{2n-2m}.
\end{align*}
In other words, the choice of such a Levi subgroup corresponds to fixing a vector subspace $U$ satisfying $V= W \oplus U \oplus W^\ast$. Let us also remark that $L$ has root system $\Psi$ given by 
\[
\Psi^+ = \{ \varepsilon_i - \varepsilon_j, \, i <j \leq m \} \cup \{ \varepsilon_i - \varepsilon_j, \varepsilon_i+\varepsilon_j, \, m <i<j \} \cup \{ 2\varepsilon_j, \, m<j\}.
\]
This can be visualized in the following block decomposition :
\[
\begin{tikzcd}
L_m = \left\{ 
\begin{pmatrix}
A & 0 & 0\\ 0 & B & 0 \\ 0 & 0 & -(A^{-1})^\sharp
\end{pmatrix} \colon A \in \GL(W), B \in \Sp(U) \right\} \subset P_m = \left\{
\begin{pmatrix}
\ast & \ast & \ast \\ 0 & \ast & \ast \\ 0 & 0 & \ast
\end{pmatrix}\right\} 
\end{tikzcd}
\]

\begin{proposition}
\label{prop:CnPm}
Assume given a nonreduced parabolic subgroup $P$ such that $P_{\text{red}} = P_m$. Then $\Lie P = \Lie \widetilde{G}$ or $\Lie P = \Lie P_m + \mathfrak{g}_<$. If $p=3$, then necessarily $\Lie P = \Lie \widetilde{G}$.
\end{proposition}

\begin{proof}
The Lie algebra of $P_m$ contains all root subspaces except for those associated to negative roots containing $\alpha_m$ in their support, hence
\[
V_m \defeq \Lie \widetilde{G} / \Lie P_m = \left( \bigoplus_{i<j\leq m} \mathfrak{g}_{-\varepsilon_i - \varepsilon_j} \bigoplus_{j\leq m} \mathfrak{g}_{-2\varepsilon_j} \right) \bigoplus_{i \leq m <j} \left( \mathfrak{g}_{-\varepsilon_i-\varepsilon_j} \oplus \mathfrak{g}_{-\varepsilon_i+\varepsilon_j} \right) 
\]
More concretely, since $L_m = \Stab(W) \cap \Stab(W^\ast)$, the Levi subgroup acts on $V_m$ as follows. First, a matrix in 
\[
\left( \bigoplus_{i<j\leq m} \mathfrak{g}_{-\varepsilon_i - \varepsilon_j} \bigoplus_{j\leq m} \mathfrak{g}_{-2\varepsilon_j} \right)
\]
is of the form
\[
\begin{pmatrix}
    0 & 0 & 0\\
    0 & 0 & 0\\
    \Omega_m X & 0 & 0
\end{pmatrix}
\]
with $X \in \Sym^2(W)$, and the $L_m$-action on it is given by
\begin{align*}
(A,B) \cdot X & \simeq \begin{pmatrix}
A & 0 & 0\\ 0 & B & 0 \\ 0 & 0 & -(A^{-1})^\sharp
\end{pmatrix} \begin{pmatrix}
0 & 0 & 0 \\ 0 & 0 & 0 \\ X & 0 & 0
\end{pmatrix} \begin{pmatrix}
A^{-1} & 0 & 0\\ 0 & B^{-1} & 0 \\ 0 & 0 & -A^\sharp
\end{pmatrix}\\
& = 
\begin{pmatrix}
0 & 0 & 0 \\ 0 & 0 & 0 \\ -\Omega_m ({}^t\!A^{-1} XA^{-1}) & 0 & 0
\end{pmatrix}\simeq \, {}^t\!A^{-1} X A^{-1},
\end{align*}
hence this $L_m$-module is isomorphic to the dual of the standard representation of $\GL_m$ acting on $\Sym^2(k^m)$.

Let us assume that the characteristic of the base field is $p=2$: then $\Sym^2(W)$ has an irreducible $L_m$-submodule given by $\oplus_{j \leq m} \mathfrak{g}_{-2\varepsilon_j}$: this proves that, once such a root subspace is contained in $\Lie P$ for some $j\leq m$, then all root subspaces associated to long negative roots are. If $p=3$, then $\Sym^2(W)$ is already an irreducible submodule itself, hence it is either contained in $\Lie P/\Lie P_m$ or has trivial intersection with it.\\
On the other hand, by \Cref{roots_C}, an element of
\[
\bigoplus_{i \leq m <j} \left( \mathfrak{g}_{-\varepsilon_i-\varepsilon_j} \oplus \mathfrak{g}_{-\varepsilon_i+\varepsilon_j} \right) =: M
\]
is of the form
\[
\begin{pmatrix}
    0 & 0 & 0\\Y & 0 & 0\\ 0 & Y^\flat & 0
\end{pmatrix}, \quad \text{where } \, Y^\flat \defeq \Omega_m {}^t\! Y \begin{pmatrix} 0 & \Omega_{n-m}\\ \Omega_{n-m} & 0\end{pmatrix}
\]
with $Y \in \Hom_k(W,U)$. This gives as $L_m$-action
\[
(A,B) \cdot Y \simeq \begin{pmatrix}
A & 0 & 0\\ 0 & B & 0 \\ 0 & 0 & -(A^{-1})^\sharp
\end{pmatrix} \begin{pmatrix}
0 & 0 & 0 \\ Y & 0 & 0 \\ 0 & Y^\flat & 0
\end{pmatrix} \begin{pmatrix}
A^{-1} & 0 & 0\\ 0 & B^{-1} & 0 \\ 0 & 0 & -A^\sharp
\end{pmatrix} \simeq \, BYA^{-1},
\]
because $B$ being an element of $\Sp(U)$ implies
\[
(A^{-1})^\sharp Y^\flat B^{-1} = \Omega_m {}^t\! A^{-1} \, {}^t Y {}^t B \begin{pmatrix} 0 & \Omega_{n-m}\\ -\Omega_{n-m} & 0\end{pmatrix} = (BYA^{-1})^\flat.
\]
Thus, $M$ is isomorphic as an $L_m$-module to the representation
\[
\GL_m \times \Sp_{2n-2m} \curvearrowright \Hom_k(k^m,k^{2n-2m}), \quad (A,B) \cdot Y = BYA^{-1}
\]
This means in particular that $M$ is an irreducible $L_m$-module, since the Weyl group acts transitively on the set of its weights.\\
Now, let us go back to the parabolic subgroup $P$: being nonreduced, $\Lie P / \Lie P_m$ is a nontrivial $L_m$-submodule of $V_m$. We already know that assuming such a quotient to contain $\mathfrak{g}_{-2\varepsilon_j}$ implies it contains all of them, thus we still need three claims to conclude the proof:
\begin{enumerate}[(a)]
    \item assuming $\Lie P / \Lie P_m$ to contain a subspace associated to a long negative root implies it also contains a subspace associated to a short negative root;
    \item assuming it to contain a subspace associated to a short negative root implies it contains all of them;
    \item when $p = 3$, assuming it to contain a subspace associated to a short negative root implies it also contains a subspace associated to a long negative root.
\end{enumerate}
For (a), assume $\mathfrak{g}_{-2\varepsilon_j} \subset \Lie P$ for some $j\leq m$, then consider roots 
\begin{align*}
    \gamma = -2\varepsilon_j, \quad & \text{satisfying } X_\gamma \in \Lie P\\
    \delta = \varepsilon_j-\varepsilon_n, \quad & \text{satisfying } X_\delta \in \Lie B \subset \Lie P.
\end{align*}
Since $\gamma+\delta$ is a root and $\delta-\gamma$ is not, \Cref{lemmaChevalley} yields
\[
[X_{-2\varepsilon_j}, X_{\varepsilon_j-\varepsilon_n}] = \pm X_{-\varepsilon_j-\varepsilon_n} \in \Lie P.
\]
Let us remark that (a) is automatically true when  $p=3$ due to the irreducibility of the $L_m$-module $\Sym^2(W)$, without needing to consider any structure constant.\\
For (b), first assume some $\mathfrak{g}_\eta \subset M$ is also contained in $\Lie P$. Then $M \subset \Lie P$ because of its irreducibility as $L_m$-submodule of $V_m$. Moreover, fixing $i<j\leq m$ and applying \Cref{lemmaChevalley} to $\gamma = -\varepsilon_i-\varepsilon_n$ and $\delta = -\varepsilon_j +\varepsilon_n$,  satisfying $X_\gamma, X_\delta \in M$, we obtain
\[
[X_{-\varepsilon_i-\varepsilon_j}, X_{-\varepsilon_j +\varepsilon_n}] = \pm X_{-\varepsilon_i-\varepsilon_j} \in \Lie P.
\]
Thus (b) holds in this case. On the other hand, let us start by assuming that $\mathfrak{g}_{-\varepsilon_i-\varepsilon_j} \subset \Lie P$ for some $i<j\leq m$. Then, applying \Cref{lemmaChevalley} to $\gamma = -\varepsilon_i -\varepsilon_j$ and $\delta = \varepsilon_j - \varepsilon_n \in \Phi^+$ yields
\[
[X_{-\varepsilon_i-\varepsilon_j}, X_{\varepsilon_j-\varepsilon_n}] = \pm X_{-\varepsilon_i-\varepsilon_n} \in \Lie P
\]
so we conclude that some $\mathfrak{g}_\nu \subset M$ is contained in $\Lie P$ and conclude by the beginning of the proof of (b).\\
For (c) it is enough to use (b) and the irreducibility of the $L_m$-submodule $\Sym^2(W)$ when $p=3$.
\end{proof}

\begin{proof}(\textbf{of \Cref{final_rank1_weak} in type $C_n$ when $P_{\text{red}}= P_m$})\\
Let $G$ be simple adjoint of type $C_n$ and $X=G/P$ with a faithful $G$-action such that $P_{\text{red}}= P^{\alpha_m}$ and $P$ is nonreduced. Define $\widetilde{P} \subset \widetilde{G} = \Sp_{2n}$ as being the preimage of $P$ in the simply connected cover: it is a nonreduced parabolic subgroup satisfying ${\widetilde{P}}_{\text{red}}= P_m$. When $p=2$, \Cref{prop:CnPm} implies that 
\[
\langle \, \Lie (\gamma^\vee(\Gm)) \colon \gamma \in \Phi_< \, \rangle  \oplus \mathfrak{g}_< = \Lie N_{\widetilde{G}} \subset \Lie \widetilde{P},
\]
hence by considering the image in the adjoint quotient we get $N_G \subset P$, which is a contradiction by Remark \ref{faithful_action}. If $p=3$ then \Cref{prop:CnPm} implies that $\Lie \widetilde{P} = \Lie \widetilde{G}$, hence the Frobenius kernel satisfies $\widetilde{G}_1 \subset \widetilde{P}$, and its image in the adjoint quotient is a normal subgroup of $G$ contained in $P$, which gives again a contradiction. Therefore in both cases $P$ must be a smooth parabolic.
\end{proof}


\subsection{Type $B_n$}

The aim of this subsection is to get the same results for the group of type $B_n$, with the help of some of the computations involving structure constants, which we have already done in case of type $C_n$.
\subsubsection{Lie algebra of $\SO_{2n+1}$}
 Before continuing with our proof, let us compute what $\Lie G$ looks like, where $G= \SO_{2n+1}=\SO(k^{2n+1})$ is defined as being relative to the quadratic form 
 \[
 Q(x) = x_n^2 + \sum_{i=0}^{n-1} x_ix_{2n-i},
 \]in order to determine all its root spaces and be able
to make explicit computations with them.
To do this, let us consider as maximal torus $T \subset G$ the one given by diagonal matrices of the form
\[ t = \diag(t_1,\ldots,t_n,1,t_n^{-1},\ldots,t_1^{-1}) \in \GL_{2n+1},
\]
while the Borel subgroup is given by upper triangular matrices in $G$. The Lie algebra is given by all matrices of the form
\[
M = \begin{pmatrix}
& f_0 &\\
A=(a_{ij})_{i,j=1}^n & \vdots & B=(b_{ij})_{i,j=1}^n\\
& f_{n-1} &\\
g_0\ldots g_{n-1} & h & g_{n+1}\ldots g_{2n}\\
& f_{n+1} &\\
C=(c_{ij})_{i,j=1}^n & \vdots & D=(d_{ij})_{i,j=1}^n  \\
& f_{2n} &
\end{pmatrix} \in \mathfrak{g}_{2n+1}
\]
satisfying $Q((1+\epsilon M) x ) = Q(x)$ for all $x \in k^{2n+1}$, where $\epsilon^2=0$. Let us compute
\begin{align*}
    & Q((1+\epsilon M)x) =  \left( x_n + \epsilon (g_0x_0 + \ldots + g_{n-1}x_{n-1} + hx_n + g_{n+1}x_{n+1} + \ldots g_{2n}x_{2n} ) \right)^2 \\ 
    & + \sum_{i=0}^{n-1} \left( x_i + \epsilon \left( \sum_{j=0}^{n-1} a_{ij}x_j + f_i x_n + \sum_{m=0}^{n-1} b_{i,n+1-m}x_{2n-m} \right) \right)\\
    & \cdot \left( x_{2n-i} + \epsilon \left( \sum_{r=0}^{n-1} c_{n+1-i,r} x_r + f_{2n-i}x_n + \sum_{l=0}^{n-1}d_{n+1-i,n+1-l} x_{2n-l}
  \right)\right)\\
   & = Q(x) + \epsilon [ 2hx_n^2 + \sum_{i=0}^{n-1} (f_{2n-i}+2g_i)x_ix_n + \sum_{i=0}^{n-1} (f_i+2g_{2n-i}) x_nx_{2n-i}\\
   & + \sum_{i,m=0}^{n-1} b_{i,n+1-m}x_{2n-m}x_{2n-i} + \sum_{i,r=0}^{n-1} c_{n+1-i,r} x_ix_r + \sum_{i,j=0}^{n-1} (a_{ij}+d_{n+1-j,n+1-i})x_jx_{2n-i}]
\end{align*}
Asking the above quantity to be equal to $Q(x)$ gives the following conditions:
\[
2h=0, \quad f_i = -2g_{2n-i}, \quad f_{2n-i}= g_i, \quad D = -A^\sharp, \quad C = -C^\sharp, \quad B = -B^\sharp,
\]
where we keep the notation (\ref{sharp}). Moreover, the matrices $\Omega_n B$ and $\Omega_n C$ have zero diagonal. Since the group considered is special orthogonal, the last condition on the determinant implies that the trace of the matrix must be zero hence $h= 0$ also in characteristic $2$. The result is thus
\[
\Lie \SO_{2n+1} = \left\{
 \begin{pmatrix}
A & -2\Omega_n w & B\\
{}^t v & 0 & {}^t w\\
C & -2\Omega_n v & -A^\sharp
\end{pmatrix}
\in \mathfrak{gl}_{2n+1} \colon C = -C^\sharp, \, B = -B^\sharp, \, c_{n+1-i,i}= b_{n+1-i,i}= 0
\right\}
\]

\begin{remark}
\label{roots_B}
Denoting, analogously to the type $C_n$, as $\varepsilon_i \in X(T)$ the character $t \mapsto t_i$ for $1\leq i \leq n$, the root spaces are the following :
\begin{align*}
    & \mathfrak{g}_{-\varepsilon_i} = k \begin{pmatrix}
0 & 0 & 0 \\ {}^t e_i & 0 & 0 \\ 0 & -2e_{n+1-i} & 0
\end{pmatrix}, &\\
& \mathfrak{g}_{\varepsilon_i} = k \begin{pmatrix}
0 & -2e_i & 0 \\ 0 & 0 & {}^t e_{n+1-i} \\ 0 & 0 & 0
\end{pmatrix}, & 1\leq i \leq n,\\
     & \mathfrak{g}_{\varepsilon_i+\varepsilon_j} = k \begin{pmatrix}
0 & 0 & (E_{ij}+E_{ji})\Omega_n \\ 0 & 0 & 0 \\ 0 & 0 & 0
\end{pmatrix}, & \\
& \mathfrak{g}_{-\varepsilon_i-\varepsilon_j} = k \begin{pmatrix}
0 & 0 & 0 \\ 0 & 0 & 0 \\ \Omega_n(E_{ij}+E_{ji}) & 0 & 0
\end{pmatrix}, & i <j,\\
     & \mathfrak{g}_{\varepsilon_i-\varepsilon_j} = k \begin{pmatrix}
E_{ij} & 0 & 0 \\ 0 & 0 & 0 \\ 0 & 0 & - E_{ij}^\sharp
\end{pmatrix}, &\\
& \mathfrak{g}_{-\varepsilon_i+\varepsilon_j} = k \begin{pmatrix}
E_{ji} & 0 & 0 \\ 0 & 0 & 0 \\ 0 & 0 & - E_{ji}^\sharp
\end{pmatrix}, & i <j,
\end{align*}
where $e_i$ denotes the standard basis of $k^n$ and $E_{ij}$ the square matrix of order $n$ with all zero entries except for the $(i,j)$-th which is equal to one.
\end{remark}

We thus verify that the root system $\Phi = \Phi (G,T)$ is given by
\[
\Phi^+ = \{ \varepsilon_i - \varepsilon_j, \, \varepsilon_i + \varepsilon_j, \, 1 \leq i < j \leq n \} \cup \{ \varepsilon_i, \, 1 \leq i \leq n \},
\]
with basis $\Delta$ consisting of the following roots :
\begin{align}
\label{roots_Bn}
\alpha_1 = \varepsilon_1 - \varepsilon_2, \, \ldots, \, \alpha_{n-1} = \varepsilon_{n-1} - \varepsilon_n, \, \alpha_n = \varepsilon_n.
\end{align}

\subsubsection{Reduced parabolic $P_n$}
Going back to our setting, let us consider the maximal reduced parabolic subgroup $P_n = P^{\alpha_n}$ associated to the short simple root $\alpha_n$, i.e. the stabilizer of the isotropic vector subspace $W = ke_0 \oplus \cdots \oplus ke_{n-1} \subset V$ of dimension $n$, where $G= \SO(V)$ and $(e_i)_{i=0}^{2n}$ denotes the standard basis of $k^{2n+1}$. Since its Levi subgroup $L_n = P_n \cap P_n^-$ stabilizes both $W$ and its dual $W^\ast = ke_{n+1} \oplus \cdots \oplus ke_{2n}$, we conclude that it is of the form 
 \[
\begin{tikzcd}
L_n = \left\{
\begin{pmatrix}
A & 0 & 0 \\ 0 & 1 & 0 \\ 0 & 0 & (A^{-1})^\sharp
\end{pmatrix} \colon \, A \in \GL(W) \simeq \GL_n
\right\} \subset P_n = \left\{
\begin{pmatrix}
\ast & \ast & \ast \\ 0 & \ast & \ast \\ 0 & 0 & \ast
\end{pmatrix}\right\} \subset G,
\end{tikzcd}
\]
where $V = W \oplus ke_n \oplus W^\ast$. In particular, $L_n$ is isomorphic to $\GL_n$, with root system $\Psi$ given by 
\[
\Psi^+ = \{ \varepsilon_i - \varepsilon_j, \, 1 \leq i < j \leq n \}.
\]

\begin{proposition}
\label{LieN_2}
Assume given a nonreduced parabolic subgroup $P$ such that $P_{\text{red}} = P_n$. Then $\Lie P = \Lie G$ or $\Lie P = \Lie P_n + \mathfrak{g}_<$. If $p=3$, then necessarily $\Lie P = \Lie G$.
\end{proposition}

\begin{proof}
First, by definition of $P_n$ its Lie algebra is given by
\[
\Lie P_n = \Lie L_n \, \bigoplus_{i<j} \mathfrak{g}_{\varepsilon_i+\varepsilon_j} \bigoplus_i \mathfrak{g}_{\varepsilon_i},
\]
Since $P$ is assumed to be nonreduced, $\Lie P_n \subsetneq \Lie P$ hence :
\begin{enumerate}[$(1)$]
    \item either there is some $i$ such that $\mathfrak{g}_{-\varepsilon_i} \subset \Lie P$, 
    \item or there is some $i<j$ such that $\mathfrak{g}_{-\varepsilon_i-\varepsilon_j} \subset \Lie P$.
\end{enumerate}

Let us start by assuming $(1)$ and fix such an index $i$. To show that all other $\mathfrak{g}_{-\varepsilon_j}$ are then contained in $\Lie P$, let us consider the $L_n$-module 
\[
V_n \defeq \Lie G / \Lie P_n = \bigoplus_{i<j} \mathfrak{g}_{-\varepsilon_i - \varepsilon_j} \bigoplus_i \mathfrak{g}_{-\varepsilon_i}.
\]

By \Cref{roots_B}, a matrix in $\bigoplus_{i=1}^n \mathfrak{g}_{-\varepsilon_i}$ is of the form 
\[
\begin{pmatrix}
    0 & 0 & 0\\
    {}^t v & 0 & 0\\
    0 & -2\Omega_n v & 0
\end{pmatrix}
\]
for $v \in k^n$, and the dual $L_n$-action on it is given by
\begin{align}
\label{lemma:6.3.2}
{}^t\!A^{-1} \cdot v & = \begin{pmatrix}
{}^t\!A^{-1} & 0 & 0 \\ 0 & 1 & 0 \\ 0 & 0 & {}^t\! A^\sharp
\end{pmatrix}
\begin{pmatrix}
    0 & 0 & 0\\
    {}^t v & 0 & 0\\
    0 & -2\Omega_n v & 0
\end{pmatrix}
\begin{pmatrix}
{}^t\! A & 0 & 0 \\ 0 & 1 & 0 \\ 0 & 0 & ({}^t\! A^{-1})^\sharp
\end{pmatrix} \\ & =
\begin{pmatrix}
0 & 0 & 0 \\ {}^t(Av) & 0 & 0 \\ 0 & -2\Omega_n Av & 0
\end{pmatrix} \simeq Av
\end{align}
In particular, $\bigoplus_{i=1}^n \mathfrak{g}_{-\varepsilon_i}$ is a simple $L_n$-module, isomorphic to the dual of the standard representation of $\GL_n$ on $k^n$. Thus, if a root subspace associated to some $-\varepsilon_i$ is contained in $\Lie P$, all of the $\mathfrak{g}_{-\varepsilon_j}$ are too.\\
Let us assume instead that $(2)$ holds: then, by repeating the same exact reasoning done in case (b) of the preceding subsection, we show that $\Lie P $ contains all weight spaces associated to long roots. This is due to the fact that the argument above only involves roots of the form $\pm (\varepsilon_l \pm \varepsilon_m)$. Moreover, assume $i \neq n$ and consider roots
\begin{align*}
    \gamma = \varepsilon_n, \quad & \text{satisfying } X_\gamma \in \Lie L_n \subset \Lie P\\
    \delta = -\varepsilon_i-\varepsilon_n, \quad & \text{satisfying } X_\delta \in \Lie P \text{ by our last assumption. }
\end{align*}
Thus, $\gamma + \delta = -\varepsilon_i$ is still a root while $\delta - \gamma = -\varepsilon_i -2\varepsilon_n$ is not: applying \Cref{lemmaChevalley} gives 
\[
[X_{\varepsilon_n}, X_{-\varepsilon_i-\varepsilon_n}] = \pm X_{-\varepsilon_i} \in \Lie P.
\]
In conclusion, when $p=2$ we have shown that condition $(2)$ implies $\Lie P = \Lie G$, while assuming condition $(1)$ to be true and $(2)$ to be false implies $\Lie P = \Lie P_n + \mathfrak{g}_<$.\\
If $p=3$ then the above reasoning still holds; the only remark that we need to add is that $\mathfrak{g}_< \subset \Lie P$ implies that there is a long negative root $\nu$ satisfying $\mathfrak{g}_\nu \subset \Lie P / \Lie P_n$. For this, let us consider roots
\[
\gamma = -\varepsilon_1 \text{ and } \delta = -\varepsilon_n, \text{ satisfying } X_\gamma,X_\delta \in \Lie P \text{ by our last assumption}.
\]
Thus, $\gamma+\delta = -\varepsilon_1-\varepsilon_n$ is still a root, $\gamma-\delta = -\varepsilon_1+\varepsilon_n$ is too, while $\gamma-2\delta = -\varepsilon_1+2\varepsilon_n$ is not: applying \Cref{lemmaChevalley} gives
\[
[X_{-\varepsilon_1}, X_{-\varepsilon_n} ] = \pm 2 X_{-\varepsilon_1-\varepsilon_n}, \quad \text{hence } X_{-\varepsilon_1-\varepsilon_n} \in \Lie P.
\]
Clearly, this last step of the proof would not work under the hypothesis $p =2$.
\end{proof}

\begin{proof}(\textbf{of \Cref{final_rank1_weak} in type $B_n$ when $P_{\text{red}}= P_n$})\\
Let $G$ be simple adjoint of type $B_n$ and $X=G/P$ with a faithful $G$-action such that $P_{\text{red}}= P_n = P^{\alpha_n}$ and $P$ is nonreduced. When $p=2$, the above Proposition, together with the computation of Example \ref{N_SO}, imply that 
\[
\mathfrak{g}_< = \Lie N_G \subset \Lie P,
\]
hence we get $N_G \subset P$, which is a contradiction by Remark \ref{faithful_action}. When $p=3$, the above Proposition implies that $\Lie P = \Lie G$, hence the Frobenius kernel satisfies $G_1\subset P$, which gives again a contradiction. Therefore in both cases $P$ must be a smooth parabolic.
\end{proof}

\begin{remark}
\label{rem_lifting}
A small additional remark is needed in order to have a uniform statement later on, since this is the only case where the group $G$ is not simply connected: let $\psi \colon \widetilde{G} = \Spin_{2n+1} \longrightarrow G = \SO_{2n+1}$ be the quotient morphism and consider a nonreduced parabolic subgroup $P \subset \widetilde{G}$ such that $P_{\text{red}} = P^{\alpha_n}$. The above reasoning implies that $\psi(P)$ either contains $N_G$ - when such a subgroup is defined - or it contains the Frobenius kernel $G_1$. In particular, $P$ contains a normal noncentral subgroup of height one, namely $P \cap \psi^{-1}(N_G)$ or $P \cap \psi^{-1}(G_1)$.
\end{remark}

\subsubsection{Reduced parabolic $P_m$, $m<n$}

Let us consider again a $k$-vector space $V$ of dimension $2n+1$ and denote as $G$ the group $\SO_{2n+1} = \SO(V)$, of type $B_n$ with $n \geq 2$ and $k$ of characteristic $p=2$ or $3$. Moreover, let us consider the maximal reduced parabolic subgroup
\[
P_m \defeq P^{\alpha_m}
\]
associated to a long simple root $\alpha_m$ for some $m<n$, keeping notations from (\ref{roots_Bn}). This subgroup is the stabilizer of an isotropic vector subspace $W = ke_0 \oplus \cdots \oplus ke_{m-1} \subset V$ of dimension $m$, where $(e_i)_{i=0}^{2n}$ denotes the standard basis of $k^{2n+1}$. Since its Levi subgroup $L_m = P_m \cap P_m^-$ stabilizes both $W$ and its dual $W^\ast = ke_{2n-m+1} \oplus \cdots \oplus ke_{2n}$, we conclude that it is of the form 
 \[
\begin{tikzcd}
L_m = \left\{ 
\begin{pmatrix}
A & 0 & 0\\ 0 & B & 0 \\ 0 & 0 & (A^{-1})^\sharp
\end{pmatrix} \colon A \in \GL(W), B \in \SO(U) \right\} \subset P_m = \left\{
\begin{pmatrix}
\ast & \ast & \ast \\ 0 & \ast & \ast \\ 0 & 0 & \ast
\end{pmatrix}\right\} 
\end{tikzcd}
\]
where $V = W \oplus U \oplus W^\ast$. In particular, $ L_m \simeq \GL_m \times \SO_{2n-2m+1}$ with root system $\Psi$ given by 
\[
\Psi^+ = \{ \varepsilon_i - \varepsilon_j, \, i <j \leq m \} \cup \{ \varepsilon_i - \varepsilon_j, \varepsilon_i+\varepsilon_j, \, m <i<j \} \cup \{ \varepsilon_j, \, m<j\}.
\]

\begin{proposition}
\label{LieN_3}
Assume given a nonreduced parabolic subgroup $P$ such that $P_{\text{red}} = P_m$. Then $\Lie P = \Lie G$ or $\Lie P = \Lie P_m + \mathfrak{g}_<$.  If $p=3$, then necessarily $\Lie P = \Lie G$.
\end{proposition}

\begin{proof}
The Lie algebra of $P_m$ contains all root subspaces except for those associated to negative roots containing $\alpha_m$ in their support, hence
\[
V_m \defeq \Lie G/ \Lie P_m = \left( \bigoplus_{i<j\leq m} \mathfrak{g}_{-\varepsilon_i - \varepsilon_j} \bigoplus_{j\leq m} \mathfrak{g}_{-\varepsilon_j} \right) \bigoplus_{i \leq m <j} \left( \mathfrak{g}_{-\varepsilon_i-\varepsilon_j} \oplus \mathfrak{g}_{-\varepsilon_i+\varepsilon_j} \right) 
\]
The analogous computations as those in the proofs of \Cref{prop:CnPm} and (\ref{lemma:6.3.2}) imply that, as $L_m$-modules, 
\begin{enumerate}[(1)]
    \item $\bigoplus_{j\leq m} \mathfrak{g}_{-\varepsilon_j}$ is isomorphic to the dual of the standard representation of $\GL_n$ on $k^m$, hence it is in particular a simple $L_m$-submodule of $V_m$ ;
    \item $\bigoplus_{i \leq m <j} \left( \mathfrak{g}_{-\varepsilon_i-\varepsilon_j} \oplus \mathfrak{g}_{-\varepsilon_i+\varepsilon_j} \right)$ is isomorphic to the following representation, which gives a second irreducible $L_m$-submodule of $V_m$ :
\end{enumerate}
\[
\GL_m \times \SO_{2n-2m+1} \curvearrowright \Hom_k(k^m,k^{2n-2m+1}), \quad (A,B) \cdot Y = BYA^{-1}.
\]
Now, first assume $\mathfrak{g}_{-\varepsilon_l} \subset \Lie P$ for some $l\leq m$. Then $\bigoplus_{j\leq m} \mathfrak{g}_{-\varepsilon_j}$ is contained in $\Lie P$, since $\Lie P / \Lie P_m$ is a nontrivial $L_m$-submodule of $V_m$. Hence in this case $\mathfrak{g}_< \subset \Lie P$.\\
The only other possibility is to start by assuming that $\mathfrak{g}_\gamma \subset \Lie P$ for some long negative root $\gamma$ containing $\alpha_m$ in its support. Then one can repeat the same exact reasoning of point (b) in the proof of \Cref{prop:CnPm}, since it involves only roots of the form $\pm (\varepsilon_l \pm \varepsilon_m)$ with $l <m$, to conclude that all root subspaces associated to long negative roots are also contained in $\Lie P$. To conclude that, in this case, $\Lie P = \Lie G$, it suffices to apply \Cref{lemmaChevalley} to $\gamma = -\varepsilon_1 - \varepsilon_m$ and $\delta = \varepsilon_m$, which gives
\[
[X_{-\varepsilon_1-\varepsilon_m}, X_{\varepsilon_m}] = \pm X_{-\varepsilon_1} \in \Lie P
\]
as wanted.\\
Up to this point everything holds in both characteristic $p=2$ and $3$. To conclude it is enough to show that, when $p=3$, if $\mathfrak{g}_< \subset \Lie P$ then there is a long negative root $\nu$ satisfying $\mathfrak{g}_\nu \subset \Lie P / \Lie P_m$. For this, let us consider roots
\[
\gamma = -\varepsilon_1 \text{ and } \delta = -\varepsilon_n, \text{ satisfying } X_\gamma,X_\delta \in \Lie P \text{ by our last assumption}.
\]
Thus, $\gamma+\delta = -\varepsilon_1-\varepsilon_n$ is still a root, $\gamma-\delta = -\varepsilon_1+\varepsilon_n$ is too, while $\gamma-2\delta = -\varepsilon_1+2\varepsilon_n$ is not: applying \Cref{lemmaChevalley} gives
\[
[X_{-\varepsilon_1}, X_{-\varepsilon_n} ] = \pm 2 X_{-\varepsilon_1-\varepsilon_n}, \quad \text{hence } X_{-\varepsilon_1-\varepsilon_n} \in \Lie P
\]
as wanted.
\end{proof}

\begin{proof}(\textbf{of \Cref{final_rank1_weak} in type $B_n$ when $P_{\text{red}}= P_m$})\\
Let $G$ be simple adjoint of type $B_n$ and $X=G/P$ with a faithful $G$-action such that $P_{\text{red}}= P^{\alpha_m}$ and $P$ is nonreduced. When $p=2$ the above Proposition, together with Example \ref{N_SO}, imply that 
\[
\mathfrak{g}_< = \Lie N_G \subset \Lie P,
\]
hence we get $N_G \subset P$, which is a contradiction by Remark \ref{faithful_action}. When $p=3$, the above Proposition implies that $\Lie P = \Lie G$, hence the Frobenius kernel satisfies $G_1\subset P$, which gives again a contradiction. Therefore in both cases $P$ must be a smooth parabolic.
\end{proof}

\begin{remark}
\label{rem_lifting2}
As in Remark \ref{rem_lifting} above, we can conclude that if $P \subset \Spin_{2n+1}$ is a nonreduced parabolic subgroup satisfying $P_{\text{red}} = P^{\alpha_m}$, then it contains a normal noncentral subgroup of height one.
\end{remark}

\subsection{Type $F_4$}
\label{F4}

Let us consider a simple group $G$ with root system $F_4$ over an algebraically closed field $k$ of characteristic $p=2$ or $3$. Following notations from \cite{Bourbaki}, a basis $\Delta$ of its root system $\Phi$ is given by
\[
\alpha_1= \varepsilon_2-\varepsilon_3, \quad \alpha_2 = \varepsilon_3-\varepsilon_4, \quad \alpha_3 = \varepsilon_4, \quad \alpha_4 = \frac{1}{2}(\varepsilon_1-\varepsilon_2-\varepsilon_3-\varepsilon_4),
\]
satisfying the relations
\[
\vert\vert \alpha_1 \vert\vert^2 = \vert\vert \alpha_2 \vert \vert^2=2, \quad \vert\vert \alpha_3 \vert\vert^2 = \vert\vert \alpha_4 \vert \vert^2=1
\]
and
\begin{align}
\label{F4_relations}
    (\alpha_1,\alpha_2)= (\alpha_2,\alpha_3)=-1, \quad \! (\alpha_1,\alpha_3) = (\alpha_1,\alpha_4)= (\alpha_2,\alpha_4) = 0, \quad\!  (\alpha_3,\alpha_4)= -\frac{1}{2}.
\end{align}
Let us denote the associated maximal reduced parabolic subgroups as $P_i \defeq P^{\alpha_i}$, for $i \in \{ 1,2,3,4\}$. Let us also recall that, when $p=2$, $N_G \subset G$ is the unique subgroup of height one such that
\[
\Lie N_G = \Lie \alpha_3^\vee(\Gm) \oplus \Lie \alpha_4^\vee(\Gm) \oplus \mathfrak{g}_<,
\]
where the short positive roots are
\begin{align*}
& \alpha_3, \,\, \alpha_4, \,\, \alpha_2+\alpha_3, \,\,  \alpha_3+\alpha_4, \,\, \alpha_1+\alpha_2+\alpha_3, \,\, \alpha_2+2\alpha_3+\alpha_4,\\ & \alpha_2+\alpha_3+\alpha_4,\,\, \alpha_1+2\alpha_2+3\alpha_3+2\alpha_4, \,\, \alpha_1+\alpha_2+\alpha_3+\alpha_4, \\ & \alpha_1+\alpha_2+2\alpha_3+\alpha_4,\,\,
 \alpha_1+2\alpha_2+2\alpha_3+\alpha_4, \,\,
 \alpha_1+2\alpha_2+3\alpha_3+\alpha_4. 
\end{align*}

\begin{proposition}
\label{LieN_6}
Assume given a nonreduced parabolic subgroup $P$ such that $P_{\text{red}}= P_i$ for some $i$. Then $\Lie P = \Lie G$ or $\Lie P = \Lie P_i + \mathfrak{g}_<$. If $p=3$, then necessarily $\Lie P = \Lie G$.
\end{proposition}

\begin{proof}
Before starting a case-by-case analysis, let us denote as $s_i$, for $i = 1,2,3,4$, the reflection associated to the simple root $\alpha_i$, i.e.
\begin{align}
\label{F4_reflections}
    s_i (\gamma) = \gamma - 2\frac{(\alpha_i,\gamma)}{(\alpha_i,\alpha_i)} \alpha_i, \qquad \text{for all } \gamma \in \Phi.
\end{align}

\textbf{Case $P_\text{red}= P_1$}.\\
Let us assume that $P_{\text{red}}= P_1$ and denote as $L_1 \defeq P_1 \cap P_1^-$ the Levi subgroup: its root system is of type $C_3$ with basis consisting of short roots $\alpha_4$, $\alpha_3$ and the long root $\alpha_2$. Moreover, $L_1$ acts on the vector space
\[
V_1 \defeq \Lie G / \Lie P_1 = \bigoplus_{\gamma \in \Gamma_1} \mathfrak{g}_{-\gamma},
\]
where $\Gamma_1$ is the subset of all positive roots satisfying $\alpha_1 \in \Supp(\gamma)$. As usual, let us consider the nonzero vector subspace $W_1 \defeq \Lie P / \Lie P_1 $, which is a $L_1$-submodule of $V_1$: the set of its weights, which we denote $\Omega_1$, must be stable under the reflections $s_2$, $s_3$ and $s_4$. Our aim is to show that
\begin{align}
\label{claim1}
    \text{either} \quad  \Omega_1 = \Gamma_1 \cap \Phi_< \quad \text{or} \quad \Omega_1 = \Gamma_1 :
\end{align}
in other words, either $W_1 = \oplus_{\gamma \in \Gamma_1 \cap \Phi_<} \mathfrak{g}_{-\gamma}$ or $W_1=V_1$.\\
First, let us show that the Weyl group $W(L_1,T)= \langle s_2,s_3,s_4\rangle$ acts transitively on 
\begin{align*}
\Gamma_1 \cap \Phi_< = \{ & \alpha_1+\alpha_2+\alpha_3, \, \alpha_1+\alpha_2+\alpha_3+\alpha_4, \, \alpha_1+\alpha_2+2\alpha_3+\alpha_4, \, \alpha_1+2\alpha_2+2\alpha_3+\alpha_4,\\
& \alpha_1+2\alpha_2+3\alpha_3+\alpha_4, \, \alpha_1+2\alpha_2+3\alpha_3+2\alpha_4\} :
\end{align*}
this implies that either $\Gamma_1 \cap \Phi_< \subset \Omega_1$ or $(\Gamma_1 \cap \Phi_<)\cap \Omega_1 = \emptyset$. The following computations follow directly from (\ref{F4_relations}) and (\ref{F4_reflections}) :
\begin{align*}
    & s_4(\alpha_1+\alpha_2+\alpha_3) = \alpha_1 + \alpha_2+\alpha_3+\alpha_4,\\
    & s_3(\alpha_1+\alpha_2+\alpha_3+\alpha_4) = \alpha_1+\alpha_2 +2\alpha_3+\alpha_4,\\
    & s_2(\alpha_1+\alpha_2+2\alpha_3+\alpha_4) = \alpha_1+2\alpha_2+2\alpha_3+\alpha_4,\\
    & s_3(\alpha_1+2\alpha_2+2\alpha_3+\alpha_4) = \alpha_1+2\alpha_2+3\alpha_3+\alpha_4,\\
    & s_4(\alpha_1+2\alpha_2+3\alpha_3+\alpha_4) = \alpha_1+2\alpha_2+3\alpha_3+2\alpha_4.
\end{align*}
Next, let us show that $W(L_1,T)$ acts transitively on 
\begin{align*}
(\Gamma_1 \cap \Phi_>) \backslash \{ \widetilde{\alpha} \} = \{ & \alpha_1, \, \alpha_1+\alpha_2, \, \alpha_1+\alpha_2+2\alpha_3, \, \alpha_1+2\alpha_2+2\alpha_3, \, \alpha_1+\alpha_2+2\alpha_3+2\alpha_4,\\
& \alpha_1+2\alpha_2+2\alpha_3+2\alpha_4, \, \alpha_1+2\alpha_2+4\alpha_3+2\alpha_4, \, \alpha_1+3\alpha_2+4\alpha_3+2\alpha_4 \},
\end{align*}
where $\widetilde{\alpha} \defeq 2\alpha_1 +3\alpha_2 + 4\alpha_3+ 2\alpha_4$ is the highest root. Let us remark that $\widetilde{\alpha}$ is indeed fixed by the Weyl group of $L_1$: this is due to the fact that it is the only root whose coefficient of $\alpha_1$ is $2$ instead of $1$. Again, the transitivity of the action is proved by direct computation :
\begin{align*}
    & s_2(\alpha_1) = \alpha_1+\alpha_2,\\
    & s_3(\alpha_1+\alpha_2) = \alpha_1+\alpha_2+2\alpha_3,\\
    & s_2(\alpha_1+\alpha_2+2\alpha_3) = \alpha_1 + 2\alpha_2 + 2\alpha_3,\\
    & s_4 (\alpha_1+2\alpha_2+2\alpha_3) = \alpha_1+2\alpha_2+2\alpha_3+2\alpha_4,\\
    & s_3(\alpha_1+2\alpha_2+2\alpha_3+2\alpha_4) = \alpha_1+2\alpha_2+4\alpha_3+2\alpha_4, \\
    & s_1(\alpha_1+2\alpha_2+2\alpha_3+2\alpha_4) = \alpha_1 + \alpha_2 + 2\alpha_3+2\alpha_4,\\
    & s_2(\alpha_1+2\alpha_2+4\alpha_3+2\alpha_4) = \alpha_1+3\alpha_2+4\alpha_3+2\alpha_4.
\end{align*}
Thus, either $(\Gamma_1 \cap \Phi_>) \backslash \{ \widetilde{\alpha} \} \subset \Omega_1$ or $((\Gamma_1 \cap \Phi_>) \backslash \{ \widetilde{\alpha} \}) \cap \Omega_1 = \emptyset$. Next, we show that $\widetilde{\alpha} \in \Omega_1$ if and only if $(\Gamma_1 \cap \Phi_>) \backslash \{ \widetilde{\alpha} \} \subset \Omega_1$.
\begin{itemize}
    \item Assume that $\mathfrak{g}_{-\widetilde{\alpha}} \subset W_1$. Then applying \Cref{lemmaChevalley} to $\gamma = -\widetilde{\alpha}$ and $\delta = \alpha_1 +2\alpha_2 +2\alpha_3$ give
    \[
    [X_{-\widetilde{\alpha}}, X_{\alpha_1+2\alpha_2+2\alpha_3}] = \pm X_{-\alpha_1-\alpha_2-2\alpha_3-2\alpha_4} \in \Lie P,
    \]
    since $\gamma+\delta$ is a root while $\gamma-\delta = -3\alpha_1-5\alpha_2-6\alpha_3-2\alpha_4$ is not. This implies that the long root $\alpha_1+\alpha_2+2\alpha_3+2\alpha_4$ belongs to $\Omega_1$
    \item  Assume that $(\Gamma_1 \cap \Phi_>) \backslash \{ \widetilde{\alpha} \} \subset \Omega_1$. In particular,
    \[
    \mathfrak{g}_{-\alpha_1-2\alpha_2-2\alpha_3} \oplus \mathfrak{g}_{-\alpha_1-\alpha_2-2\alpha_3-2\alpha_4} \subset \Lie P.
    \]
    Thus, we can apply \Cref{lemmaChevalley} to $\gamma = -\alpha_1-2\alpha_2-2\alpha_3$ and $\delta = -\alpha_1-\alpha_2-2\alpha_3-2\alpha_4$ to get
    \[
    [X_{-\alpha_1-2\alpha_2-2\alpha_3}, X_{-\alpha_1-\alpha_2-2\alpha_3-2\alpha_4}] = \pm X_{-\widetilde{\alpha}} \in \Lie P,
    \]
    since $\gamma+\delta$ is a root while $\gamma-\delta = -\alpha_2+2\alpha_4$ is not.
    \end{itemize}
The last step in order to prove (\ref{claim1}) consists in showing that $(\Gamma_1\cap \Phi_>) \subset \Omega_1$ implies $(\Gamma_1 \cap \Phi_<) \cap \Omega_1 \neq \emptyset$ which, by the above reasoning, means $\Gamma_1 = \Omega_1$. By our assumption, the long root $\gamma = -\alpha_1-2\alpha_2-2\alpha_3 -2\alpha_4$ satisfies $\mathfrak{g}_\gamma \subset \Lie P$. Setting $\delta = -\alpha_3$ and applying \Cref{lemmaChevalley} gives
\[
[X_{-\alpha_1-2\alpha_2-2\alpha_3 -2\alpha_4}, X_{-\alpha_3}] = \pm X_{-\alpha_1-2\alpha_2-3\alpha_3-2\alpha_4} \in \Lie P,
\]
since $\gamma+\delta$ is a root while $\gamma-\delta = -\alpha_1-2\alpha_2-\alpha_3-2\alpha_4$ is not. This concludes the first case.
 \vskip 10 pt
 
\textbf{Case $P_\text{red}= P_2$}.\\
Let us assume that $P_{\text{red}}= P_2$ and fix the analogous notation as above: $L_2 \defeq P_2 \cap P_2^-$ acts on 
\[
W_2 \defeq \Lie P / \Lie P_2 = \bigoplus_{\gamma \in \Omega_2} \mathfrak{g}_{-\gamma} \subset V_2 \defeq \Lie G / \Lie P_2 = \bigoplus_{\gamma \in \Gamma_2} \mathfrak{g}_{-\gamma}
\]
and its set of weights $\Omega_2$ must be stable under the action of the Weyl group $W(L_2,T) = \langle s_1, s_3, s_4\rangle$. Our aim is to show that
\begin{align}
\label{claim2}
    \text{either} \quad  \Omega_2 = \Gamma_2 \cap \Phi_< \quad \text{or} \quad \Omega_2 = \Gamma_2.
\end{align}
First, let us consider the partition of $\Gamma_2$ as disjoint union of the following subsets :
\begin{align*}
    \Sigma_1 \defeq \{ & \alpha_1+3\alpha_2+4\alpha_3+2\alpha_4, \, \widetilde{\alpha} \},\\
    \Sigma_2 \defeq \{ & \alpha_1+2\alpha_2+2\alpha_3, \, \alpha_1+2\alpha_2+2\alpha_3+2\alpha_4, \, \alpha_1+2\alpha_2+4\alpha_3+2\alpha_4 \},\\
    \Sigma_3 \defeq \{ & \alpha_1+2\alpha_2+2\alpha_3+\alpha_4, \, \alpha_1+2\alpha_2+3\alpha_3+\alpha_4, \, \alpha_1+2\alpha_2+3\alpha_3+2\alpha_4 \},\\
    \Sigma_4 \defeq \{ & \alpha_2+\alpha_3, \, \alpha_1+\alpha_2+\alpha_3, \, \alpha_1+\alpha_2+\alpha_3+\alpha_4, \, \alpha_1+\alpha_2+2\alpha_3+\alpha_4,\\
    & \alpha_2+2\alpha_3+\alpha_4, \, \alpha_2+\alpha_3+\alpha_4\},\\
    \Sigma_5 \defeq \{ & \alpha_2+2\alpha_3+\alpha_4, \, \alpha_2+2\alpha_3, \, \alpha_1+\alpha_2, \, \alpha_2, \, \alpha_1+\alpha_2+2\alpha_3, \, \alpha_1+\alpha_2+2\alpha_3+2\alpha_4 \}.
\end{align*}

Notice that $\Sigma_1 \cup \Sigma_2 \cup \Sigma_5 = \Gamma_2 \cap \Phi_>$ and $\Sigma_3 \cup \Sigma_4 = \Gamma_2 \cap \Phi_<$, so the root lengths once again come into play. Moreover, $\Sigma_1$, $\Sigma_2 \cup \Sigma_3$ and $\Sigma_4 \cup \Sigma_5$ are indeed stable under the action of $W(L_2,T)$, since their elements have coefficient $3$, $2$ and $1$ respectively with respect to the simple root $\alpha_2$. Now, the following computations prove that :
\begin{itemize}
    \item $\Sigma_1$ is stable by $W(L_2,T)$ :
    \[
    s_1(\alpha_1+3\alpha_2+4\alpha_3+2\alpha_4 ) = \widetilde{\alpha} ;
    \]
    \item $\Sigma_2$ is stable by $W(L_2,T)$ :
    \begin{align*}
        & s_4 (\alpha_1+2\alpha_2+2\alpha_3) = \alpha_1+2\alpha_2+2\alpha_3+2\alpha_4,\\
        & s_3 (\alpha_1+2\alpha_2+2\alpha_3+2\alpha_4) = \alpha_1+2\alpha_2+4\alpha_3+2\alpha_4 ;
    \end{align*}
    \item $\Sigma_3$ is stable by $W(L_2,T)$: 
    \begin{align*}
        & s_3(\alpha_1+2\alpha_2+2\alpha_3+\alpha_4) = \alpha_1+2\alpha_2+3\alpha_3+\alpha_4,\\
        & s_4(\alpha_1+2\alpha_2+3\alpha_3+2\alpha_4) = \alpha_1+2\alpha_2+3\alpha_3+2\alpha_4;
    \end{align*}
    \item $\Sigma_4$ is stable by $W(L_2,T)$ :
    \begin{align*}
       & s_1(\alpha_2+\alpha_3) = \alpha_1+\alpha_2+\alpha_3,\\
       & s_4 (\alpha_1+\alpha_2+\alpha_3) = \alpha_1+\alpha_2+\alpha_3+\alpha_4,\\
       & s_1 (\alpha_1+\alpha_2+\alpha_3+\alpha_4) = \alpha_2+\alpha_3+\alpha_4,\\
       & s_3(\alpha_2+\alpha_3+\alpha_4) = \alpha_2+2\alpha_3+\alpha_4,\\
       & s_1(\alpha_2+2\alpha_3+\alpha_4)= \alpha_1+\alpha_2+2\alpha_3+\alpha_4;
    \end{align*}
    \item $\Sigma_5$ is stable by $W(L_2,T)$ :
    \begin{align*}
        & s_4 (\alpha_2+2\alpha_3+2\alpha_4) = \alpha_2+2\alpha_3,\\
        & s_3 (\alpha_2+2\alpha_3) = \alpha_1+\alpha_2,\\
        & s_1(\alpha_1+\alpha_2) = \alpha_2 \quad \text{and} \quad s_3(\alpha_1+\alpha_2) = \alpha_1+\alpha_2+2\alpha_3,\\
        & s_4(\alpha_1+\alpha_2+2\alpha_3) = \alpha_1+\alpha_2+2\alpha_3+2\alpha_4.
    \end{align*}
\end{itemize}
Thus, for $j = 1,\ldots, 5$, we have shown that $\Sigma_j \cap \Omega_2 \neq \emptyset$ implies that $\Sigma_j \subset \Omega_2$. Next, we prove the following claims by using \Cref{lemmaChevalley} on structure constants :
\begin{enumerate}[(a)]
   \item $\Sigma_1 \subset \Omega_2$ implies that $\Sigma_2 \subset \Omega_2$,
    \item $\Sigma_2 \subset \Omega_2$ implies that $\Sigma_5 \subset \Omega_2$,
     \item $\Sigma_5 \subset \Omega_2$ implies that $\Sigma_2 \subset \Omega_2$,
      \item $\Sigma_2 \cup \Sigma_5 \subset \Omega_2$ implies that $\Sigma_1 \subset \Omega_2$,
       \item $\Sigma_3 \subset \Omega_2$ implies that $\Sigma_4 \subset \Omega_2$,
        \item $\Sigma_4 \subset \Omega_2$ implies that $\Sigma_3 \subset \Omega_2$,
         \item $\Sigma_2 \subset \Omega_2$ implies that $\Sigma_3 \subset \Omega_2$.
\end{enumerate}
The parabolic subgroup $P$ being non-reduced by assumption, the set $\Omega_2$ is nonempty hence, once these implications are proved, it must be either all of $\Gamma_2$ or $\Sigma_3 \cup \Sigma_4 = \Gamma_2 \cap \Phi_<$, which proves (\ref{claim2}).

(a): By assumption $\mathfrak{g}_{-\alpha_1-3\alpha_2-4\alpha_3-2\alpha_4} \subset \Lie P$. Set $\gamma = -\alpha_1-3\alpha_2-4\alpha_3-2\alpha_4$ and $\delta = \alpha_2,$ then $\gamma-\delta = -\alpha_1-4\alpha_2-4\alpha_3-2\alpha_4$ is not a root hence
\[
[X_\gamma,X_\delta] = \pm X_{-\alpha_1-2\alpha_2-4\alpha_3-2\alpha_4} \in \Lie P
\]
so $\alpha_1+2\alpha_2+4\alpha_3+2\alpha_4 \in \Sigma_2 \cap \Omega_2$.

(b): By assumption $\mathfrak{g}_{-\alpha_1-2\alpha_2-4\alpha_3-2\alpha_4} \subset \Lie P$. Set $\gamma = -\alpha_1-2\alpha_2-4\alpha_3-2\alpha_4$ and $\delta = \alpha_2+2\alpha_3$, then $\gamma-\delta = -\alpha_1-3\alpha_2-6\alpha_3-2\alpha_4$ is not a root hence
\[
[X_\gamma,X_\delta] = \pm X_{-\alpha_1-\alpha_2-2\alpha_3-2\alpha_4} \in \Lie P
\]
so $\alpha_1+\alpha_2+2\alpha_3+2\alpha_4 \in \Sigma_5 \cap \Omega_2$.

(c): By assumption $\mathfrak{g}_{-\alpha_1-\alpha_2} \oplus \mathfrak{g}_{-\alpha_2-2\alpha_3} \subset \Lie P$. Set $\gamma = -\alpha_1-\alpha_2$ and $\delta = -\alpha_2-2\alpha_3$, then $\gamma-\delta = -\alpha_1-2\alpha_3$ is not a root hence
\[
[X_\gamma,X_\delta] = \pm X_{-\alpha_1-2\alpha_2-2\alpha_3} \in \Lie P
\]
so $\alpha_1+2\alpha_2+2\alpha_3 \in \Sigma_2 \cap \Omega_2$.

(d): By assumption $\mathfrak{g}_{-\alpha_1-\alpha_2-2\alpha_3-2\alpha_4} \oplus \mathfrak{g}_{-\alpha_1-2\alpha_2-2\alpha_3} \subset \Lie P$. Set $\gamma = -\alpha_1-\alpha_2-2\alpha_3-2\alpha_4$ and $\delta = -\alpha_1-2\alpha_2-2\alpha_3$, then $\gamma-\delta = \alpha_2-2\alpha_4$ is not a root hence
\[
[X_\gamma,X_\delta] = \pm X_{-\widetilde{\alpha}} \in \Lie P
\]
so $\widetilde{\alpha} \in \Sigma_1 \cap \Omega_2$.

(e): By assumption $\mathfrak{g}_{-\alpha_1-2\alpha_2-2\alpha_3-\alpha_4} \subset \Lie P$. Set $\gamma = -\alpha_1-2\alpha_2-2\alpha_3-\alpha_4$ and $\delta = \alpha_2$, then $\gamma-\delta = -\alpha_1-3\alpha_2-2\alpha_3-\alpha_4$ is not a root hence
\[
[X_\gamma,X_\delta] = \pm X_{-\alpha_1-\alpha_2-2\alpha_3-\alpha_4} \in \Lie P
\]
so $\alpha_1+\alpha_2+2\alpha_3+\alpha_4 \in \Sigma_4 \cap \Omega_2$.

(f): By assumption $\mathfrak{g}_{-\alpha_1-\alpha_2-\alpha_3-\alpha_4} \oplus \mathfrak{g}_{-\alpha_2-2\alpha_3-\alpha_4} \subset \Lie P$. Set $\gamma = -\alpha_1-\alpha_2-\alpha_3-\alpha_4$ and $\delta = -\alpha_2-2\alpha_3-\alpha_4$, then $\gamma-\delta = -\alpha_1+\alpha_3$ is not a root hence
\[
[X_\gamma,X_\delta] = \pm X_{-\alpha_1-2\alpha_2-3\alpha_3-2\alpha_4} \in \Lie P
\]
so $\alpha_1+2\alpha_2+3\alpha_3+2\alpha_4 \in \Sigma_3 \cap \Omega_2$.

(g): By assumption $\mathfrak{g}_{-\alpha_1-2\alpha_2-2\alpha_3-2\alpha_4} \subset \Lie P$. Set $\gamma = -\alpha_1-2\alpha_2-2\alpha_3-2\alpha_4$ and $\delta = -\alpha_3$, then $\gamma-\delta = -\alpha_1-2\alpha_2-\alpha_3-2\alpha_4$ is not a root hence
\[
[X_\gamma,X_\delta] = \pm X_{-\alpha_1-2\alpha_2-3\alpha_3-2\alpha_4} \in \Lie P
\]
so $\alpha_1+2\alpha_2+3\alpha_3+2\alpha_4 \in \Sigma_3 \cap \Omega_2$.
\vskip10 pt
\textbf{Case $P_\text{red}= P_3$}.\\
Let us assume that $P_{\text{red}}= P_3$ and fix the analogous notation as above: $L_3 \defeq P_3 \cap P_3^-$ acts on 
\[
W_3 \defeq \Lie P / \Lie P_3 = \bigoplus_{\gamma \in \Omega_3} \mathfrak{g}_{-\gamma} \subset V_3 \defeq \Lie G / \Lie P_3 = \bigoplus_{\gamma \in \Gamma_3} \mathfrak{g}_{-\gamma}
\]
and its set of weights $\Omega_3$ must be stable under the action of the Weyl group $W(L_3,T) = \langle s_1, s_2, s_4\rangle$. Our aim is to show that
\begin{align}
\label{claim3}
    \text{either} \quad  \Omega_3 = \Gamma_3 \cap \Phi_< \quad \text{or} \quad \Omega_3 = \Gamma_3.
\end{align}
First, let us consider the partition of $\Gamma_3$ as disjoint union of the following subsets :
\begin{align*}
    \Lambda_1 \defeq & \{ \alpha_1+2\alpha_2+4\alpha_3+2\alpha_4, \, \alpha_1+3\alpha_2+4\alpha_3+2\alpha_4, \, \widetilde{\alpha} \},\\
    \Lambda_2  \defeq &  \{ \alpha_1+2\alpha_2+2\alpha_3, \, \alpha_1+2\alpha_2+2\alpha_3+2\alpha_4, \, \alpha_1+\alpha_2+2\alpha_3+2\alpha_4, \, \alpha_2+2\alpha_3+2\alpha_4,\\
    & \alpha_2+2\alpha_3, \alpha_1+\alpha_2+2\alpha_3\},\\
    \Lambda_3 \defeq & \{ \alpha_1+2\alpha_2+3\alpha_3+2\alpha_4, \, \alpha_1+2\alpha_2+3\alpha_3+\alpha_4 \}, \\
    \Lambda_4 \defeq & \{ \alpha_1+2\alpha_2+2\alpha_3+\alpha_4, \, \alpha_1+\alpha_2+2\alpha_3+\alpha_4, \, \alpha_2+2\alpha_3+\alpha_4 \},\\
    \Lambda_5 \defeq & \{ \alpha_2+\alpha_3+\alpha_4, \, \alpha_1+\alpha_2+\alpha_3+\alpha_4, \, \alpha_1+\alpha_2+\alpha_3, \, \alpha_2+\alpha_3, \, \alpha_3, \, \alpha_3+\alpha_4 \}.
\end{align*}
Notice that $\Lambda_1 \cup \Lambda_2 = \Gamma_3 \cap \Phi_>$ and $\Lambda_3 \cup \Lambda_4 \cup \Lambda_5 = \Gamma_3 \cap \Phi_<$; moreover, as in the preceding case, let us remark that $\Lambda_1$, $\Lambda_3$, $\Lambda_2 \cup \Lambda_4$ and $\Lambda_5$ are stable under $W(L_3,T)$ because their elements have as coefficient respectively $4$, $3$, $2$ and $1$ with respect to the simple root $\alpha_3$. Now let us prove by direct computation that :
\begin{itemize}
    \item $\Lambda_1$ is stable by $W(L_3,T)$ :
    \begin{align*}
    & s_2(\alpha_1+2\alpha_2+4\alpha_3+2\alpha_4) = \alpha_1+3\alpha_2+4\alpha_3+2\alpha_4,\\
    & s_1 (\alpha_1+3\alpha_2+4\alpha_3+2\alpha_4) = \widetilde{\alpha};
    \end{align*}
    \item $\Lambda_2$ is stable by $W(L_3,T)$ :
    \begin{align*}
        & s_4(\alpha_1+2\alpha_2+2\alpha_3) = \alpha_1+2\alpha_2+2\alpha_3+2\alpha_4,\\
        & s_2(\alpha_1+2\alpha_2+2\alpha_3+2\alpha_4) = \alpha_1+\alpha_2+2\alpha_3+2\alpha_4,\\
        & s_1(\alpha_1+\alpha_2+2\alpha_3+2\alpha_4) = \alpha_2+2\alpha_3+2\alpha_4,\\
        & s_4(\alpha_2+2\alpha_3+2\alpha_4) = \alpha_2+2\alpha_3,\\
        & s_1(\alpha_2+2\alpha_3) = \alpha_1+\alpha_2+2\alpha_3;
    \end{align*}
    \item $\Lambda_3$ is stable by $W(L_3,T)$ :
    \[
    s_4(\alpha_1+2\alpha_2+3\alpha_3+2\alpha_4) = \alpha_1+2\alpha_2+3\alpha_3+\alpha_4;
    \]
    \item $\Lambda_4$ is stable by $W(L_3,T)$ :
    \begin{align*}
        & s_2(\alpha_1+2\alpha_2+2\alpha_3+\alpha_4) = \alpha_1+\alpha_2+2\alpha_3+\alpha_4,\\
        & s_1(\alpha_1+\alpha_2+2\alpha_3+\alpha_4) = \alpha_2+2\alpha_3+\alpha_4;
    \end{align*}
    \item $\Lambda_5$ is stable by $W(L_3,T)$ :
    \begin{align*}
        & s_1(\alpha_2+\alpha_3+\alpha_4) = \alpha_1+\alpha_2+\alpha_3+\alpha_4,\\
        & s_4(\alpha_1+\alpha_2+\alpha_3+\alpha_4) = \alpha_1+\alpha_2+\alpha_3,\\
        & s_1(\alpha_1+\alpha_2+\alpha_3) = \alpha_2+\alpha_3,\\
        & s_2(\alpha_2+\alpha_3) = \alpha_3,\\
        & s_4(\alpha_3) = \alpha_3+\alpha_4.
    \end{align*}
\end{itemize}
Thus, for $j = 1,\ldots, 5$, we have shown that $\Lambda_j \cap \Omega_3 \neq \emptyset$ implies that $\Lambda_j \subset \Omega_3$. Next, we need to prove the following claims by using \Cref{lemmaChevalley} on structure constants :
\begin{enumerate}[(a)]
    \item $\Lambda_1 \subset \Omega_3$ implies that $\Lambda_2 \subset \Omega_3$,
    \item $\Lambda_2 \subset \Omega_3$ implies that $\Lambda_1 \subset \Omega_3$,
    \item $\Lambda_3 \subset \Omega_3$ implies that $\Lambda_4 \subset \Omega_3$,
    \item $\Lambda_4 \subset \Omega_3$ implies that $\Lambda_5 \subset \Omega_3$,
    \item $\Lambda_5 \subset \Omega_3$ implies that $\Lambda_4 \subset \Omega_3$,
    \item $\Lambda_4 \cup \Lambda_5 \subset \Omega_3$ implies that $\Lambda_3 \subset \Omega_3$,
    \item $\Lambda_1 \subset \Omega_3$ implies that $\Lambda_3 \subset \Omega_3$.
\end{enumerate}
The parabolic subgroup $P$ being non-reduced by assumption, the set $\Omega_3$ is nonempty hence, once these implications are proved, it must be either all of $\Gamma_3$ or $\Lambda_3 \cup \Lambda_4 \cup \Lambda_5 = \Gamma_3 \cap \Phi_<$, which proves (\ref{claim3}).

(a): By assumption $\mathfrak{g}_{-\widetilde{\alpha}} \subset \Lie P$. Set $\gamma = -\widetilde{\alpha}$ and $\delta = \alpha_1+2\alpha_2+2\alpha_3$, then $\gamma-\delta = -3\alpha_1+5\alpha_2+6\alpha_3+2\alpha_4$ is not a root hence
\[
[X_\gamma,X_\delta] = \pm X_{-\alpha_1-\alpha_2-2\alpha_3-2\alpha_4} \in \Lie P
\]
so $\alpha_1+\alpha_2+2\alpha_3+2\alpha_4 \in \Lambda_2 \cap \Omega_3$.

(b): By assumption $\mathfrak{g}_{-\alpha_1-2\alpha_2-2\alpha_3} \oplus \mathfrak{g}_{-\alpha_1-\alpha_2-2\alpha_3-2\alpha_4} \subset \Lie P$. Set $\gamma = -\alpha_1-2\alpha_2-2\alpha_3$ and $\delta = -\alpha_1-\alpha_2-2\alpha_3-2\alpha_4$, then $\gamma-\delta = -\alpha_2+2\alpha_4$ is not a root hence
\[
[X_\gamma,X_\delta] = \pm X_{-\widetilde{\alpha}} \in \Lie P
\]
so $\widetilde{\alpha} \in \Lambda_1 \cap \Omega_3$.

    (c): By assumption $\mathfrak{g}_{-\alpha_1-2\alpha_2-3\alpha_3-2\alpha_4} \subset \Lie P$. Set $\gamma = -\alpha_1-2\alpha_2-3\alpha_3-2\alpha_4$ and $\delta = \alpha_3+\alpha_4 \in \Phi^+$, then $\gamma-\delta = -\alpha_1-2\alpha_2-4\alpha_3-3\alpha_4$ is not a root hence
\[
[X_\gamma,X_\delta] = \pm X_{-\alpha_1-2\alpha_2-2\alpha_3-\alpha_4} \in \Lie P
\]
so $\alpha_1+2\alpha_2+2\alpha_3+\alpha_4 \in \Lambda_4 \cap \Omega_3$.

(d): By assumption $\mathfrak{g}_{-\alpha_2-2\alpha_3-\alpha_4} \subset \Lie P$. Set $\gamma = -\alpha_2-2\alpha_3-\alpha_4$ and $\delta = \alpha_3 \in \Phi^+$, then $\gamma-\delta = -\alpha_2-3\alpha_3-\alpha_4$ is not a root hence
\[
[X_\gamma,X_\delta] = \pm X_{-\alpha_2-\alpha_3-\alpha_4} \in \Lie P
\]
so $\alpha_2+\alpha_3+\alpha_4 \in \Lambda_5 \cap \Omega_3$.

(e): By assumption $\mathfrak{g}_{-\alpha_3-\alpha_4} \oplus \mathfrak{g}_{-\alpha_2-\alpha_3} \subset \Lie P$. Set $\gamma = -\alpha_3-\alpha_4$ and $\delta = -\alpha_2-\alpha_3$, then $\gamma-\delta = -\alpha_2+\alpha_4$ is not a root hence 
\[
[X_\gamma, X_\delta] = \pm X_{-\alpha_2-2\alpha_3-\alpha_4} \in \Lie P
\]
so $\alpha_2+2\alpha_3+\alpha_4 \in \Lambda_4 \cap \Omega_3$.

(f): By assumption $\mathfrak{g}_{-\alpha_1-\alpha_2-\alpha_3-\alpha_4} \oplus \mathfrak{g}_{-\alpha_2-2\alpha_3-\alpha_4}\subset \Lie P$. Set $\gamma = -\alpha_1-\alpha_2-\alpha_3-\alpha_4$ and $\delta = -\alpha_2-2\alpha_3-\alpha_4$, then $\gamma-\delta = -\alpha_1-\alpha_3$ is not a root hence 
\[
[X_\gamma, X_\delta] = \pm X_{-\alpha_1-2\alpha_2-3\alpha_3-2\alpha_4} \in \Lie P
\]
so $\alpha_1+2\alpha_2+3\alpha_3+2\alpha_4 \in \Lambda_3 \cap \Omega_3$.

(g): By assumption $\mathfrak{g}_{-\alpha_1-2\alpha_2-4\alpha_3-2\alpha_4} \subset \Lie P$. Set $\gamma = -\alpha_1-2\alpha_2-4\alpha_3-2\alpha_4$ and $\delta = \alpha_3 \in \Phi^+$, then $\gamma-\delta = -\alpha_1-2\alpha_2-5\alpha_3-2\alpha_4$ is not a root hence
\[
[X_\gamma,X_\delta]= \pm X_{-\alpha_1-2\alpha_2-3\alpha_3-2\alpha_4} \in \Lie P
\]
so $\alpha_1-2\alpha_2+3\alpha_3+2\alpha_4 \in \Lambda_3 \cap \Omega_3$.

\vskip 10 pt
\textbf{Case $P_\text{red}= P_4$}.\\
Let us assume that $P_{\text{red}}= P_4$ and fix the analogous notation as above: the Levi subgroup $L_4 \defeq P_4 \cap P_4^-$, which is of type $B_3$, acts on 
\[
W_4 \defeq \Lie P / \Lie P_4 = \bigoplus_{\gamma \in \Omega_4} \mathfrak{g}_{-\gamma} \subset V_4 \defeq \Lie G / \Lie P_4 = \bigoplus_{\gamma \in \Gamma_4} \mathfrak{g}_{-\gamma}
\]
and its set of weights $\Omega_4$ must be stable under the action of the Weyl group $W(L_4,T) = \langle s_1, s_2, s_3 \rangle$. Our aim is to show that
\begin{align}
\label{claim4}
    \text{either} \quad  \Omega_4 = \Gamma_4 \cap \Phi_< \quad \text{or} \quad \Omega_4 = \Gamma_4.
\end{align}
Let $\beta \defeq \alpha_1+2\alpha_2+3\alpha_3+2\alpha_4$ and consider, as in the first case of this proof, the action of $W(L_4,T)$ on 
\begin{align*}
    (\Gamma_4 \cap \Phi_<) \backslash \{ \beta, \alpha_4\} \defeq \{ & \alpha_3+\alpha_4, \, \alpha_2+\alpha_3+\alpha_4, \, \alpha_1+\alpha_2+\alpha_3+\alpha_4, \, \alpha_1+\alpha_2+2\alpha_3+\alpha_4,\\
    &  \alpha_1+2\alpha_2+2\alpha_3+\alpha_4, \, \alpha_1+2\alpha_2+3\alpha_3+\alpha_4, \, \alpha_2+2\alpha_3+\alpha_4\},
\end{align*}
which is transitive because 
\begin{align*}
    & s_2(\alpha_3+\alpha_4) = \alpha_2+\alpha_3+\alpha_4\\
    & s_1(\alpha_2+\alpha_3+\alpha_4) = \alpha_1+\alpha_2+\alpha_3+\alpha_4,\\
    & s_3(\alpha_1+\alpha_2+\alpha_3+\alpha_4) = \alpha_1+\alpha_2+2\alpha_3+\alpha_4,\\
    & s_2(\alpha_1+\alpha_2+2\alpha_3+\alpha_4) = \alpha_1+2\alpha_2+2\alpha_3+\alpha_4,\\
    & s_1(\alpha_1+\alpha_2+2\alpha_3+\alpha_4) = \alpha_2+2\alpha_3+\alpha_4,\\
    & s_3(\alpha_1+2\alpha_2+2\alpha_3+\alpha_4)= \alpha_1+2\alpha_2+3\alpha_3+\alpha_4,
\end{align*}
and the same action on 
\begin{align*}
    \Gamma_4 \cap \Phi_> = \{ & \alpha_2+2\alpha_3+2\alpha_4, \, \alpha_1+\alpha_2+2\alpha_3+2\alpha_4, \, \alpha_1+2\alpha_2+2\alpha_3+2\alpha_4, \\ & \alpha_1+2\alpha_2+4\alpha_3+2\alpha_4,\,
     \alpha_1+3\alpha_2+4\alpha_3+2\alpha_4, \, \widetilde{\alpha}\},
\end{align*}
which is also transitive because 
\begin{align*}
    & s_1(\alpha_2+2\alpha_3+2\alpha_4) = \alpha_1+\alpha_2+2\alpha_3+2\alpha_4,\\
    & s_2(\alpha_1+\alpha_2+2\alpha_3+2\alpha_4) = \alpha_1 +2\alpha_2+2\alpha_3+2\alpha_4,\\
    & s_3(\alpha_1+2\alpha_2+2\alpha_3+2\alpha_4) = \alpha_1+2\alpha_2+4\alpha_3+2\alpha_4,\\
    & s_2(\alpha_1+2\alpha_2+4\alpha_3+2\alpha_4) = \alpha_1+3\alpha_2+4\alpha_3+2\alpha_4,\\
    & s_1(\alpha_1+3\alpha_2+4\alpha_3+2\alpha_4) = \widetilde{\alpha}.
\end{align*}
Next, we prove the following claims using \Cref{lemmaChevalley} on structure constants :
\begin{enumerate}[(a)]
    \item $\Gamma_4 \cap \Phi_> \subset \Omega_4$ implies that $\beta \in \Omega_4$,
    \item $\beta \in \Omega_4$ implies that $(\Gamma_4 \cap \Phi_<) \backslash \{ \beta, \alpha_4 \} \subset \Omega_4$,
    \item $(\Gamma_4 \cap \Phi_<) \backslash \{ \beta, \alpha_4 \} \subset \Omega_4$ implies that $\alpha_4 \in \Omega_4$,
    \item $\alpha_4 \in \Omega_4$ implies that $(\Gamma_4 \cap \Phi_<) \backslash \{ \beta, \alpha_4 \} \subset \Omega_4$,
    \item $(\Gamma_4 \cap \Phi_<) \backslash \{ \beta \} \subset \Omega_4$ implies that $\beta \subset \Omega_4$.
\end{enumerate}
The parabolic subgroup $P$ being non-reduced by assumption, the set $\Omega_4$ is nonempty hence, once these implications are proved, it must me either all of $\Gamma_4$ or $\Gamma_4 \cap \Phi_<$, which proves (\ref{claim4}).

(a): By assumption $\mathfrak{g}_{-\alpha_2-2\alpha_3-2\alpha_4} \subset \Lie P$ and $\mathfrak{g}_{-\alpha_1-\alpha_2-\alpha_3} \in \Lie L_4 \subset \Lie P$. Set $\gamma = -\alpha_2-2\alpha_3-2\alpha_4$ and $\delta = -\alpha_1-\alpha_2-\alpha_3$, then $\gamma -\delta = \alpha_1-\alpha_3-2\alpha_4$ is not a root hence 
\[
[X_\gamma,X_\delta] = \pm X_{-\beta} \in \Lie P
\]
so $\beta \in \Omega_4$.

(b): By assumption $\mathfrak{g}_{-\beta} \subset \Lie P$. Set $\gamma = -\beta$ and $\delta = \alpha_1+\alpha_2+\alpha_3+\alpha_4 \in \Phi^+$, then $\gamma-\delta = -2\alpha_1-3\alpha_2-4\alpha_3-3\alpha_4$ is not a root hence
\[
[X_\gamma,X_\delta] = \pm X_{-\alpha_2-2\alpha_3-\alpha_4} \in \Lie P
\]
so $\alpha_2+2\alpha_3+\alpha_4 \in ((\Gamma_4 \cap \Phi_<)\backslash \{ \beta,\alpha_4\}) \cap \Omega_4$.

(c): By assumption $\mathfrak{g}_{-\alpha_3-\alpha_4} \subset \Lie P$. Set $\gamma = -\alpha_3-\alpha_4$ and $\delta = \alpha_3 \in \Phi^+$, then $\gamma-\delta = -2\alpha_3-\alpha_4$ is not a root hence
\[
[X_\gamma,X_\delta] = \pm X_{-\alpha_4} \in \Lie P
\]
so $\alpha_4 \in \Omega_4$.

(d): By assumption $\mathfrak{g}_{-\alpha_4} \subset \Lie P$ and $\mathfrak{g}_{-\alpha_3} \subset \Lie L_4 \subset \Lie P$. Set $\gamma = -\alpha_4$ and $\delta = -\alpha_3$, then $\gamma-\delta = \alpha_3-\alpha_4$ is not a root hence 
\[
[X_\gamma, X_\delta] = \pm X_{-\alpha_3-\alpha_4} \in \Lie P
\]
so $\alpha_3+\alpha_4 \in ((\Gamma_4 \cap \Phi_<)\backslash \{ \beta,\alpha_4\}) \cap \Omega_4$.

(e): By assumption $\mathfrak{g}_{-\alpha_4} \oplus \mathfrak{g}_{-\alpha_1-2\alpha_2-3\alpha_3-\alpha_4} \subset \Lie P$. Set $\gamma = -\alpha_4$ and $\delta = -\alpha_1-2\alpha_2-3\alpha_3-\alpha_4$, then $\gamma -\delta = \alpha_1+2\alpha_2+3\alpha_3$ is not a root hence
\[
[X_\gamma,X_\delta] = \pm X_{-\beta} \in \Lie P
\]
so $\beta \in \Omega_4$.

\textbf{Conclusion}: up to this point all computations hold in both characteristic $p=2$ and $3$. To conclude our proof when $p=3$, one more step - which works simultaneously for all cases $i = 1,2,3,4$ - is necessary in order to conclude that $\Omega_i = \Gamma_i$. That is, we want to show that $(\Gamma_i \cap \Phi_<) \subset \Omega_i$ implies $(\Gamma_i \cap \Phi_>) \cap \Omega_i \neq \emptyset$. By assumption, $\mathfrak{g}_{-\alpha_1-2\alpha_2-3\alpha_3-2\alpha_4} \subset \Lie P$. Set $\gamma = -\alpha_1-2\alpha_2-3\alpha_3-2\alpha_4$ and $\delta = \alpha_3 \in \Phi^+$, then $\gamma+\delta$ and $\gamma-\delta$ are still roots while $\gamma-2\delta = -\alpha_1-2\alpha_2-5\alpha_3-2\alpha_4$ is not, hence
\[
[X_\gamma, X_\delta] = \pm 2 X_{-\alpha_1-2\alpha_2-4\alpha_3-2\alpha_4} \in \Lie P \quad \text{hence }  X_{-\alpha_1-2\alpha_2-4\alpha_3-2\alpha_4} \in \Lie P,
\]
so that $-\alpha_1-2\alpha_2-4\alpha_3-2\alpha_4 \in (\Gamma_i \cap \Phi_>) \cap \Omega_i$ as wanted. 
\end{proof}

\begin{proof}(\textbf{of \Cref{final_rank1_weak} in type $F_4$})\\
Let $G$ be simple of type $F_4$ and $X=G/P$ with a faithful $G$-action such that $P_{\text{red}}$ is maximal and $P$ is nonreduced. When $p=2$, \Cref{LieN_6} implies that $\mathfrak{g}_<  \subset \Lie P$, 
hence we get $\Lie N_G \subset \Lie P$ and $N_G \subset P$ by the equivalence of categories, which is a contradiction by Remark \ref{faithful_action}. When $p=3$, the above Proposition implies that $\Lie P = \Lie G$, hence the Frobenius kernel satisfies $G_1\subset P$, which gives again a contradiction. Therefore in both cases $P$ must be a reduced parabolic.
\end{proof}

\subsection{Type $G_2$}
\label{subsec:G2}

The last non-simply laced Dynkin diagram we have to consider is of type $G_2$. In this case, things behave as expected when the reduced parabolic subgroup is $P^{\alpha_2}$, the one associated with the long simple root $\alpha_2$, or when the characteristic is $p=3$: the proof follows the same strategy as in types $B_n$, $C_n$ and $F_4$.\\
This still leaves out the case of a nonreduced parabolic subgroup satisfying $P_{\text{red}}=P^{\alpha_1}$ in characteristic $2$, where $\alpha_1$ denotes the short simple root. Under such assumptions, we find two maximal $p$-Lie subalgebras 
\[
\mathfrak{h} \defeq \Lie P^{\alpha_1} \oplus \mathfrak{g}_{-2\alpha_1-\alpha_2} \quad \text {and} \quad
\mathfrak{l} \defeq \Lie P^{\alpha_1} \oplus \mathfrak{g}_{-\alpha_1} \oplus \mathfrak{g}_{-\alpha_1-\alpha_2},
\]containing $\Lie P^{\alpha_1}$. 
Let $H$ and $L$ be of height one with Lie algebra respectively equal to $\mathfrak{g}_{-2\alpha_1-\alpha_2}$ and $\mathfrak{g}_{-\alpha_1} \oplus \mathfrak{g}_{-\alpha_1-\alpha_2}$, and define
\[
P_\mathfrak{h} \defeq \langle H, P^{\alpha_1} \rangle \quad \text{and} \quad P_\mathfrak{l} \defeq \langle L, P^{\alpha_1} \rangle.
\]
This gives rise to two parabolic subgroups which have as reduced subgroup a maximal one, but cannot be described as $(\ker \varphi)P^{\alpha_1}$ for some isogeny $\varphi$ with source $G$. We then move on to investigate the corresponding homogeneous spaces, which we describe by using the Chevalley description of $G$ as automorphism group of an octonion algebra.\\
The main result is the following, which completes the classification of \Cref{main}.

\begin{theorem}
\label{main_G2}
    Let $G$ be of type $G_2$ in characteristic two and let $P$ be a nonreduced parabolic subgroup of $G$ having $P^{\alpha_1}$ as reduced part.\\
    Then one of the three following cases holds:
    \begin{itemize}
        \item $P$ is of standard type and $X \simeq G/P^{\alpha_1}$ is isomorphic to a quadric $Q$ in $\proj^6$;
        \item $P$ is obtained from $P_{\mathfrak{h}}$ by pull back via an iterated Frobenius morphism and $X \simeq G/P_{\mathfrak{h}}$ is isomorphic to $\proj^5$;
        \item $P$ is obtained from $P_{\mathfrak{l}}$ by pull back via an iterated Frobenius morphism and $X \simeq G/P_{\mathfrak{l}}$ is isomorphic to a hyperplane section of $\Sp_6/P^{\alpha_3}$.
   \end{itemize}
\end{theorem}

Let us recall for reference the following result: see \cite[Theorem 1]{Demazure}, reformulated here under the stronger hypothesis of $k$ being an algebraically closed field. It will be needed to conclude the type $G_2$ case, as well as later on, when dealing with higher Picard ranks.

\begin{theorem}
\label{demazure77}
Let $H^\prime$ be a semisimple adjoint group over $k$ and $Q^\prime$ a reduced parabolic subgroup of $H^\prime$. Then the natural homomorphism
\[
H^\prime \longrightarrow H\defeq \underline{\Aut}^0_{H^\prime/Q^\prime}
\]
is an isomorphism in all but the three following cases: 
\begin{enumerate}[(a)]
    \item $H^\prime$ is of type $C_n$ for some $n \geq 2$ and $Q^\prime= P^{\alpha_1}$ is associated to the first short simple root: in this case the automorphism group $H$ is smooth simple adjoint of type $A_{2n-1}$; 
    \item $H^\prime$ is of type $B_n$ for some $n \geq 2$ and $Q^\prime = P^{\alpha_n}$ is associated to the short simple root: in this case the automorphism group $H$ is smooth simple adjoint of type $D_{n+1}$; 
    \item $H^\prime$ is of type $G_2$ and $Q^\prime = P^{\alpha_1}$: in this case the automorphism group $H$ is smooth simple adjoint of type $B_3$.
\end{enumerate}
\end{theorem}

With a slight change of notation compared to Demazure, we call the three pairs $(H,Q)$ in the cases $(a)$, $(b)$ and $(c)$ of the Theorem \emph{exceptional}, while $(H^\prime,Q^\prime)$ is called the \emph{associated} pair to the exceptional one.

\begin{remark}
\label{autX}
In order to be clear let us recall what we mean by automorphism group, both in \Cref{demazure77} and in the rest of the paper. For a proper algebraic scheme $X$ over a perfect field $k$, let us consider the functor 
\[
\underline{\Aut}_X \colon (\mathbf{Sch}/k)_{\text{red}} \longrightarrow \mathbf{Grp}, \quad T \longmapsto \Aut_T(X_T),
\]
sending a reduced $k$-scheme $T$ to the group of automorphisms of $T$-schemes of $X\times_k T$. By \cite[Theorem 3.6]{MatsumuraOort} this functor is represented by a reduced group scheme $\underline{\Aut}_X$ which is locally of finite type over $k$. We denote as $\underline{\Aut}_X^0$ its connected component of the identity, which is a smooth algebraic group.
\end{remark}

\subsubsection{What works as expected}

Let us consider a group $G$ with root system of type $G_2$ over a field $k$ of characteristic $p = 2$ or $3$. Following notations from \cite{Bourbaki}, the elements of $\Phi^+$ are 
\[
\alpha_1, \quad \alpha_1+\alpha_2, \quad 2\alpha_1+\alpha_2, \quad 3\alpha_1 +\alpha_2, \quad \alpha_2, \quad 3\alpha_1+2\alpha_2. 
\]
In particular, let us consider as elements of the basis $\Delta$ the short root $\alpha_1$ and the long root $\alpha_2$; then denote $P_1 \defeq P^{\alpha_1}$ and $P_2\defeq P^{\alpha_2}$ the associated maximal reduced parabolic subgroups. 
\smallskip

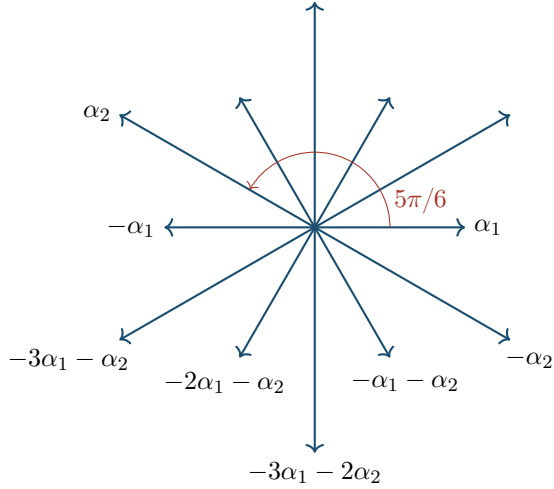
\begin{center}
\begin{tikzpicture}
    \foreach\ang in {60,120,...,360}{
     \draw[->,cite!80!black,thick] (0,0) -- (\ang:2cm);
    }
    \foreach\ang in {30,90,...,330}{
     \draw[->,cite!80!black,thick] (0,0) -- (\ang:3cm);
    }
    \draw[links,->](1,0) arc(0:150:1cm)node[pos=0.14,right,scale=0.9]{$5\pi/6$};
    \node[anchor= west,scale=0.9] at (2,0) {$\alpha_1$};
    \node[anchor= east,scale=0.9] at (-2,0) {$- \alpha_1$};
    \node[anchor=south west,scale=0.9] at (-3.2,1.3) {$\alpha_2$};
    \node[anchor=north east,scale=0.9] at (3.3,-1.5) {$- \alpha_2$};
    \node[anchor=north west,scale=0.9] at (-4.2,-1.5) {$- 3\alpha_1-\alpha_2$};
    \node[anchor=north,scale=0.9] at (-1.2,-1.8) {$-2\alpha_1-\alpha_2$};
    \node[anchor=north,scale=0.9] at (1.2,-1.8) {$-\alpha_1-\alpha_2$};
    \node[anchor=north,scale=0.9] at (0,-3) {$-3\alpha_1-2\alpha_2$};
  \end{tikzpicture}
\end{center}
\smallskip

Let us recall that, when $p=3$, $N_G \subset G$ is in this case the unique subgroup of height one such that
\[
\Lie N_G = \Lie \alpha_1^\vee(\Gm) \oplus \mathfrak{g}_{\alpha_1} \oplus \mathfrak{g}_{-\alpha_1} \oplus \mathfrak{g}_{\alpha_2+2\alpha_1} \oplus \mathfrak{g}_{-\alpha_2-2\alpha_1} \oplus \mathfrak{g}_{\alpha_1+\alpha_2} \oplus \mathfrak{g}_{-\alpha_1-\alpha_2}.
\]

\begin{proposition}
\label{LieN_5}
Assume given a nonreduced parabolic $P$ such that $P_{\text{red}}= P_1$ (with $p=3)$ or $P_{\text{red}} = P_2$ (with $p=2$ or $3$). Then $\Lie P = \Lie G$ or $\Lie P = \Lie P_{\text{red}} + \mathfrak{g}_<$. If $p=2$, then necessarily $\Lie P = \Lie G$.
\end{proposition}

\begin{remark}
We can conclude that \Cref{final_rank1_weak} holds in this case as follows: let $G$ be simple of type $G_2$ and $X=G/P$ with a faithful $G$-action such that $P_{\text{red}}$ is maximal, satisfies the hypothesis of \Cref{LieN_5}, and such that $P$ is nonreduced. The above Proposition implies that 
\[
 \Lie \alpha_1^\vee(\Gm) \oplus \mathfrak{g}_< = \Lie N_G \subset \Lie P,
\]
hence we get $N_G \subset P$, which is a contradiction by Remark \ref{faithful_action}. Therefore $P$ must be a smooth parabolic.
\end{remark}

\begin{proof}
\textbf{Case $P_\text{red}= P_1$}.\\
Let us assume that $P_{\text{red}} = P_1$ and that the characteristic is $p=3$. The Levi subgroup $L_1 \defeq P_1 \cap P_1^-$ has root system $\{ \pm \alpha_2 \}$ and acts on the vector space 
\[
V_1 \defeq \Lie G / \Lie P_1 = \mathfrak{g}_{-\alpha_1} \oplus \mathfrak{g}_{-\alpha_1-\alpha_2} \oplus \mathfrak{g}_{-2\alpha_1-\alpha_2} \oplus \mathfrak{g}_{-3\alpha_1 -\alpha_2} \oplus \mathfrak{g}_{-3\alpha_1-2\alpha_2}.
\]
Now, let us look at the nonzero vector subspace $W_1 \defeq \Lie P / \Lie P_1$, which is in particular an $L_1$-submodule of $V_1$. Thus, the set of its weights must be stable under the reflection $s_{\alpha_2}$. This means, by a direct computation, that
\begin{align}
\label{eq:longG2}
\mathfrak{g}_{-3\alpha_1-2\alpha_2} \subset W_1 \quad & \Longleftrightarrow \quad \mathfrak{g}_{-3\alpha_1-\alpha_2} \subset W_1,\\
\mathfrak{g}_{-\alpha_1-\alpha_2} \subset W_1 \quad & \Longleftrightarrow \quad \mathfrak{g}_{-\alpha_1} \subset W_1.
\end{align}

Let us assume first that $\mathfrak{g}_{-\alpha_1-\alpha_2} \oplus \mathfrak{g}_{-\alpha_1} \subset W_1$. Then, applying \Cref{lemmaChevalley} to $\gamma = -\alpha_1-\alpha_2$ and $\delta = -\alpha_1$ gives
\[
[X_{-\alpha_1-\alpha_2}, X_{-\alpha_1}] = \pm 2X_{-2\alpha_1-\alpha_2}, \quad \text{hence } X_{-2\alpha_1-\alpha_2} \in \Lie P,
\]
since $\gamma + \delta$ and $\gamma - \delta $ are roots while $\gamma-2\delta = \alpha_1-\alpha_2$ is not.
Conversely, assuming $\mathfrak{g}_{-2\alpha_1-\alpha_2} \subset W_1$ and considering roots $\gamma = -2\alpha_1-\alpha_2$ and $\delta = \alpha_1$  yields
\[
[X_{-2\alpha_1-\alpha_2}, X_{\alpha_1}] = \pm 2X_{-\alpha_1-\alpha_2}, \quad \text{hence } X_{-\alpha_1-\alpha_2} \in \Lie P.
\]
In other words, we have showed that whenever a root subspace associated to a short negative root is contained in $W_1$, the other two are too. \\
To conclude this first case, it is enough to show that 
\[
\mathfrak{g}_{-3\alpha_1-2\alpha_2} \oplus \mathfrak{g}_{-3\alpha_1-\alpha_2} \subset W_1 \quad \text{implies that} \quad \mathfrak{g}_{-2\alpha_1-\alpha_2} \subset W_1.
\]
This can be done by considering roots $\gamma = -3\alpha_1-2\alpha_2$ and $\delta = \alpha_1+\alpha_2$, for which $\gamma+\delta$ is a root but $\gamma-\delta  = -4\alpha_1 -3\alpha_2$ is not, hence 
\[
[X_{-3\alpha_1-2\alpha_2}, X_{\alpha_1+\alpha_2}] = \pm X_{-2\alpha_1 -\alpha_2} \in \Lie P.
\]

\textbf{Case $P_\text{red}= P_2$}.\\
Moving on to the second case, let us assume that $P_{\text{red}} = P_2$. The Levi subgroup $L_2 \defeq P_2 \cap P_2^-$ has root system $\{ \pm \alpha_1 \}$ and acts on the vector space 
\[
V_2 \defeq \Lie G / \Lie P_1 = \mathfrak{g}_{-\alpha_2} \oplus \mathfrak{g}_{-\alpha_1-\alpha_2} \oplus \mathfrak{g}_{-2\alpha_1-\alpha_2} \oplus \mathfrak{g}_{-3\alpha_1 -\alpha_2} \oplus \mathfrak{g}_{-3\alpha_1-2\alpha_2}.
\]
Now, let us look at the nonzero vector subspace $W_2 \defeq \Lie P / \Lie P_2$, which is in particular an $L_2$-submodule of $V_2$. Thus, the set of its weights must be stable under the reflection $s_{\alpha_1}$. This means, by a direct computation, that
\begin{align}
\label{equivalence}
\mathfrak{g}_{-\alpha_1-\alpha_2} \subset W_2 \quad & \Longleftrightarrow \quad \mathfrak{g}_{-2\alpha_1-\alpha_2} \subset W_2,\\
\mathfrak{g}_{-3\alpha_1-\alpha_2} \subset W_2 \quad & \Longleftrightarrow \quad \mathfrak{g}_{-\alpha_1} \subset W_2.
\end{align}
The equivalence (\ref{equivalence}) already implies that once a root subspace associated to a short negative root is contained in $W_2$, the only other one is too.\\
If $p=3$, to conclude it suffices to show that $\mathfrak{g}_{-\gamma} \subset W_2$ for some long root $\gamma  \in \Phi^+$ implies $W_2=V_2$ i.e. $\Lie P = \Lie G$. First,
\[
[X_{-3\alpha_1-2\alpha_2}, X_{\alpha_2}] = \pm X_{-3\alpha_1-\alpha_2},
\]
because $(-3\alpha_1 -2\alpha_2)- \alpha_2$ is not a root, and conversely
\[
[X_{-3\alpha_1-\alpha_2}, X_{-\alpha_2}] = \pm X_{-3\alpha_1-2\alpha_2},
\]
because $(-3\alpha_1-\alpha_2)-(-\alpha_2)$ is not a root. Finally, 
\[
[X_{-3\alpha_1-\alpha_2}, X_{\alpha_1}] = \pm X_{-2\alpha_1-\alpha_2},
\]
because $(-3\alpha_1-\alpha_2)-\alpha_1$ is not a root. This proves that in this case $W_2=V_2$.\\
If $p=2$, one more step must be added: assume that $\mathfrak{g}_{-2\alpha_1-\alpha_2} \oplus \mathfrak{g}_{-\alpha_1-\alpha_2} \subset W_2$, then 
\[
[X_{-2\alpha_1-\alpha_2},X_{-\alpha_1}] = \pm X_{-3\alpha_1-\alpha_2}, \quad \text{hence } X_{-3\alpha_1-\alpha_2} \in \Lie P,
\]
because $(-2\alpha_1-\alpha_2) +\alpha_1$, $(-2\alpha_1-\alpha_2) +2\alpha_1$ are roots, while $(-2\alpha_1-\alpha_2) +3\alpha_1$ is not. This last remark, together with the above computations shows that when $p=2$ necessarily $\Lie P = \Lie G$.
\end{proof}

\subsubsection{What does not}
\label{section_handl}

The only case yet to consider is the following: the characteristic is $p=2$, the group $G$ is of type $G_2$ and $P$ is a nonreduced parabolic subgroup satisfying $P_{\text{red}}= P^{\alpha_1}$, the reduced parabolic associated to the short simple root, whose Levi subgroup has root system $\{\pm \alpha_2\}$. Let us place ourselves in this setting: by repeating the same reasoning as above, we can conclude only a weaker statement.

\begin{lemma}
\label{lem:addlong}
Assume that one of the two root subspaces associated to $-3\alpha_1-2\alpha_2$ and $-3\alpha_1-\alpha_2$ is contained in $\Lie P$. Then $\Lie P = \Lie G$.
\end{lemma}

\begin{proof}
    By (\ref{eq:longG2}), we have that both root subspaces are in $\Lie P$. Then  considering roots $\gamma = -3\alpha_1-2\alpha_2$, $\delta = \alpha_1+\alpha_2$ and $\delta^\prime = 2\alpha_1+\alpha_2$ yields
\[
[X_{\gamma}, X_\delta] = \pm X_{-2\alpha_1-\alpha_2} \in \Lie P \quad \text{and} \quad [X_{\gamma}, X_{\delta^\prime}] = \pm X_{-\alpha_1-\alpha_2} \in \Lie P,
\]
because $\gamma-\delta$ and $\gamma-\delta^\prime$ are not roots. This means that if one long root is added then we have to add everything else. 
\end{proof}

The same reasoning applied to short roots fails, due to the vanishing of structure constants in characteristic $2$. More precisely, we can identify two Lie subalgebras strictly containing $\Lie P^{\alpha_1}$, which cannot be Lie ideals since $\Lie G$ is a simple $p$-Lie algebra (see \Cref{Lie_simple})
 as follows: define the following vector subspaces 
\begin{align}
\label{def:h}
\mathfrak{h} & \defeq \Lie P^{\alpha_1} \oplus \mathfrak{g}_{-2\alpha_1-\alpha_2} = \Lie B \oplus \mathfrak{g}_{-\alpha_2} \oplus \mathfrak{g}_{-2\alpha_1-\alpha_2} ;\\
\label{def:l}
\mathfrak{l} & \defeq \Lie P^{\alpha_1} \oplus \mathfrak{g}_{-\alpha_1} \oplus \mathfrak{g}_{-\alpha_1-\alpha_2} = \Lie B \oplus \mathfrak{g}_{-\alpha_2} \oplus \mathfrak{g}_{-\alpha_1} \oplus \mathfrak{g}_{-\alpha_1-\alpha_2}.
\end{align}

\begin{lemma}
\label{h_and_l}
With the above notation, $\mathfrak{h}$ and $\mathfrak{l}$ are two $p$-Lie subalgebras of $\Lie G$.
\end{lemma}

\begin{proof}
Let $\{ X_\gamma \colon \gamma \in \Phi, H_{\alpha_1}, H_{\alpha_2}\}$ be a Chevalley basis of $\Lie G$. First, using \Cref{lemmaChevalley} we can calculate a few structure constants which are then useful in the rest of the proof: 
\begin{center}
\begin{tabular}{ |c|c|c|c|c|c|}
\hline
 & $\ad (X_{-2\alpha_1-\alpha_2})$ & $\ad (X_{-\alpha_1})$ & $\ad(X_{-\alpha_1-\alpha_2)}$ & $\ad (X_{\alpha_1})$ & $\ad (X_{2\alpha_1+\alpha_2})$\\
\hline
$X_{\alpha_1}$ & $0$ & $ \in \Lie T$ & $X_{-\alpha_2}$ & $0$ & $X_{3\alpha_1+\alpha_2}$\\
\hline
$X_{3\alpha_1+\alpha_2}$ & $X_{\alpha_1}$ & $X_{2\alpha_1+\alpha_2}$ & $0$ & $0$ & $0$\\
\hline
$X_{2\alpha_1+\alpha_2}$ & $\in \Lie T$ & $0$ & $0$ & $X_{3\alpha_1+\alpha_2}$ & $0$\\
\hline
$X_{3\alpha_1+2\alpha_2}$ & $X_{\alpha_1+\alpha_2}$ & $0$ & $X_{2\alpha_1+\alpha_2}$ & $0$ & $0$\\
\hline
$X_{\alpha_1+\alpha_2}$ & $0$ & $X_{\alpha_2}$ & $\in \Lie T$ & $0$ & $X_{3\alpha_1+2\alpha_2}$\\
\hline
 $X_{\alpha_2}$ & $0$ & $0$ & $X_{-\alpha_1}$ & $X_{\alpha_1+\alpha_2}$ & $0$\\
\hline
$X_{-\alpha_1}$ & $X_{-3\alpha_1-\alpha_2}$ & $0$ & $0$ & $\in \Lie T$ & $0$\\
\hline
$X_{-3\alpha_1-\alpha_2}$ & $0$ & $0$ & $0$ & $X_{-2\alpha_1 -\alpha_2}$ & $X_{-\alpha_1}$\\
\hline
 $X_{-2\alpha_1-\alpha_2}$ & $0$ & $X_{-3\alpha_1-\alpha_2}$ & $X_{-3\alpha_1-2\alpha_2}$ & $0$ & $\in \Lie T$\\
\hline
 $X_{-3\alpha_1-2\alpha_2}$ & $0$ & $0$ & $0$ & $0$ & $X_{-\alpha_1-\alpha_2}$\\
\hline
 $X_{-\alpha_1-\alpha_2}$ & $X_{-3\alpha_1-2\alpha_2}$ & $0$ & $0$ & $X_{-\alpha_2}$ & $0$\\
\hline
 $X_{-\alpha_2}$ & $0$ & $X_{-\alpha_1-\alpha_2}$ & $0$ & $0$ & $0$\\
\hline
\end{tabular}
\end{center}

Let us verify that $\mathfrak{h}$ is a Lie subalgebra. Since we know that $\Lie P^{\alpha_1}$ is one, it is enough to show that $[\mathfrak{g}_{-2\alpha_1-\alpha_2}, \Lie P^{\alpha_1}] \subset \mathfrak{h}$. \Cref{lemmaChevalley} implies that
\[
[\mathfrak{g}_{-2\alpha_1-\alpha_2}, \Lie T] = [X_{-2\alpha_1-\alpha_2}, \Lie T] \subset \mathfrak{g}_{-2\alpha_1-\alpha_2} \subset \mathfrak{h}.
\]
Moreover, the first column of the above table shows that
\[
[\mathfrak{g}_{-2\alpha_1-\alpha_2}, \mathfrak{g}_\gamma] = k[X_{-2\alpha_1-\alpha_2}, X_\gamma] \subset \mathfrak{h},
\]
for all roots $\gamma$ whose root subspace is contained in $\Lie P^{\alpha_1}$.\\
Analogously, let us prove that $\mathfrak{l}$ is a Lie subalgebra: for this, it is enough to show that 
\[
[\mathfrak{g}_{-\alpha_1}, \Lie P^{\alpha_1}], [\mathfrak{g}_{-\alpha_1-\alpha_2}, \Lie P^{\alpha_1}], [\mathfrak{g}_{-\alpha_1}, \mathfrak{g}_{-\alpha_1-\alpha_2}] \subset \mathfrak{l}.
\]
First, \Cref{lemmaChevalley} implies that 
\begin{align*}
    [\mathfrak{g}_{-\alpha_1}, \Lie T] = [X_{-\alpha_1}, \Lie T] \subset g_{-\alpha_1} \subset \mathfrak{l} \, ;\\
    [\mathfrak{g}_{-\alpha_1-\alpha_2}, \Lie T] = [X_{-\alpha_1-\alpha_2}, \Lie T] \subset g_{-\alpha_1-\alpha_2} \subset \mathfrak{l}.
\end{align*}
Moreover, the second and third column in the above table show that
\[
[\mathfrak{g}_{-\alpha_1}, \mathfrak{g}_\gamma] = k[X_{-\alpha_1}, X_\gamma] \quad \text{and} \quad [\mathfrak{g}_{-\alpha_1-\alpha_2}, \mathfrak{g}_\gamma] = k[X_{-\alpha_1-\alpha_2},X_\gamma]
\]
are both subspaces of $\mathfrak{l}$, for all roots $\gamma$ whose root subspace is contained in $\Lie P^{\alpha_1}$.\\
To conclude there is still left to show that $\mathfrak{h}$ and $\mathfrak{l}$ are stable by the $p$-mapping (recall that by assumption $p=2)$, knowing that $\Lie P^{\alpha_1}$ is. In other words, setting $Y_\gamma$ equal to the image of $X_\gamma$ by the $p$-mapping, we want to prove that $Y_{-2\alpha_1-\alpha_2} \in \mathfrak{h}$ and that $Y_{-\alpha_1}, Y_{-\alpha_1-\alpha_2} \in \mathfrak{l}$.\\
To do this, let 
\[
Y_{-2\alpha_1-\alpha_2} = H + \sum_{\delta \in \Phi} a_\delta X_\delta, \quad \text{for some } a_\delta \in k, \, H \in \Lie T.
\]
It is enough to show that $a_{-\alpha_1} = a_{-3\alpha_1-\alpha_2}= a_{-3\alpha_1-2\alpha_2} = a_{-\alpha_1-\alpha_2}=0$. By the properties of the $p$-mapping, we have that $\ad (Y_\gamma) = \ad^2(X_\gamma)$ for any root $\gamma$. Using that $[X_{-2\alpha_1-\alpha_2}, X_{\alpha_1}]$ vanishes (see table), we have:
\begin{align*}
0 = & \ad (X_{-2\alpha_1-\alpha_2}) ([X_{-2\alpha_1-\alpha_2}, X_{\alpha_1}]) = \ad^2(X_{-2\alpha_1-\alpha_2}) (X_{\alpha_1})\\
= & \ad (Y_{-2\alpha_1-\alpha_2})(X_{\alpha_1}) = [H, X_{\alpha_1}] + \sum_{\delta \in \Phi} a_\delta[X_\delta,X_{\alpha_1}].
\end{align*}
Expanding all brackets using the fourth column of the above table gives that, for some $a \in k$,
\[
0 = aX_{\alpha_1} + a_{2\alpha_1-\alpha_2}X_{3\alpha_1+\alpha_2} + a_{\alpha_2} X_{\alpha_1+\alpha_2} + a_{-\alpha_1} H_{\alpha_1} + a_{-3\alpha_1-\alpha_2} X_{-2\alpha_1-\alpha_2} + a_{-\alpha_1-\alpha_2} X_{-\alpha_2},
\]
which implies in particular $a_{-\alpha_1}= a_{-3\alpha_1-\alpha_2}= a_{-\alpha_1-\alpha_2}= 0$. Moreover, $[X_{-2\alpha_1-\alpha_2},X_{\alpha_1}]$ also vanishes, hence we have
\begin{align*}
0 = & \ad (X_{-2\alpha_1-\alpha_2}) ([X_{-2\alpha_1-\alpha_2}, X_{\alpha_1+\alpha_2}]) = \ad^2(X_{-2\alpha_1-\alpha_2}) (X_{\alpha_1+\alpha_2})\\
= & \ad (Y_{-2\alpha_1-\alpha_2})(X_{\alpha_1+\alpha_2}) = [H, X_{\alpha_1+\alpha_2}] + \sum_{\delta \in \Phi} a_\delta[X_\delta,X_{\alpha_1+\alpha_2}].
\end{align*}
Writing this with respect to the Chevalley basis gives 
\[a_{-3\alpha_1-2\alpha_2} [X_{-3\alpha_1-2\alpha_2}, X_{\alpha_1+\alpha_2}] = a_{-3\alpha_1-2\alpha_2} X_{-2\alpha_1-\alpha_2}\] as the only term in $X_{-2\alpha_1-\alpha_2}$, meaning that the coefficient $a_{-3\alpha_1-2\alpha_2}$ also vanishes, as wanted: thus we can conclude that $\mathfrak{h}$ is a $p$-Lie subalgebra of $\Lie G$.\\
Analogously, let
\[
Y_{-\alpha_1} = H^\prime + \sum_{\delta \in \Phi} b_\delta X_\delta, \quad \text{for some } b_\delta \in k, H^\prime \in \Lie T,
\]
and as before we aim to show that $b_{-3\alpha_1-\alpha_2} = b_{-2\alpha_1-\alpha_2} = b_{-3\alpha_1-2\alpha_2}= 0$. Using that $[X_{-\alpha_1},X_{2\alpha_1+\alpha_2}]$ vanishes (see table), we have
\begin{align*}
    0 = & \ad (X_{-\alpha_1}) ([X_{-\alpha_1},X_{2\alpha_1+\alpha_2}]) = \ad^2(X_{-\alpha_1}) ([X_{-\alpha_1}, X_{2\alpha_1+\alpha_2}])\\
    = & \ad(Y_{-\alpha_1}) (X_{2\alpha_1+\alpha_2}) = [H^\prime, X_{2\alpha_1+\alpha_2}] + \sum_{\delta \in \Phi} b_\delta [X_\delta, X_{2\alpha_1+\alpha_2}].
\end{align*}
Expanding all brackets using the last column of the above table gives that, for some $b \in k$ and some $H^{\prime\prime} \in \Lie T$,
\begin{align*}
0 & =  bX_{2\alpha_1+\alpha_2} + b_{\alpha_1}X_{3\alpha_1+\alpha_2} + b_{\alpha_1+\alpha_2} X_{3\alpha_1+2\alpha_2} + b_{-3\alpha_1-\alpha_2} X_{-\alpha_1} + b_{-2\alpha_1-\alpha_2} H^{\prime\prime}\\ & + b_{-3\alpha_1-2\alpha_2}X_{-\alpha_1-\alpha_2}.
\end{align*}
In particular, this proves that $b_{-3\alpha_1-\alpha_2}= b_{-3\alpha_1-2\alpha_2}= 0$. Moreover, $[X_{-\alpha_1},X_{-\alpha_1-\alpha_2}]$ also vanishes, hence we have
\begin{align*}
    0 = & \ad (X_{-\alpha_1}) ([X_{-\alpha_1},X_{-\alpha_1-\alpha_2}]) = \ad^2(X_{-\alpha_1}) ([X_{-\alpha_1}, X_{-\alpha_1-\alpha_2}])\\
    = & \ad(Y_{-\alpha_1}) (X_{-\alpha_1-\alpha_2}) = [H^\prime, X_{-\alpha_1-\alpha_2}] + \sum_{\delta \in \Phi} b_\delta [X_\delta, X_{-\alpha_1-\alpha_2}].
\end{align*}
Expanding this with respect to the Chevalley basis gives $b_{-2\alpha_1-\alpha_2}[X_{-2\alpha_1-\alpha_2},X_{-\alpha_1-\alpha_2}] = b_{-2\alpha_1-\alpha_2}X_{-3\alpha_1-2\alpha_2}X_{-3\alpha_1-2\alpha_2}$ as the only term in $X_{-3\alpha_1-2\alpha_2}$, meaning that the coefficient $b_{-2\alpha_1-\alpha_2}$ also vanishes: this proves that $Y_{-\alpha_1} \in \mathfrak{l}$. \\
To prove that $Y_{-\alpha_1-\alpha_2}$ is also in $\mathfrak{l}$, an analogous computation, symmetric with respect to the reflection $s_{\alpha_2}$, can be done.
Finally, we can conclude that $\mathfrak{l}$ is a $p$-Lie subalgebra.
\end{proof}

\begin{corollary}
\label{cor:hl}
 The $p$-Lie subalgebras of $\Lie G$ containing strictly $\Lie P^{\alpha_1}$ are exactly $\mathfrak{h}$ and $\mathfrak{l}$.
\end{corollary}

\begin{proof}
 Let us consider a $p$-Lie subalgebra $\Lie P^{\alpha_1} \subsetneq \mathfrak{s} \subset \Lie G$, meaning that there is some positive root $\gamma \neq \alpha_1$ such that $\mathfrak{g}_{-\gamma}$ is contained in $\mathfrak{s}$. By \Cref{lem:addlong}, if $\gamma$ is long then $\mathfrak{s}= \Lie G$, so we can assume $\gamma$ to be short. To do this, let us remark that by \Cref{lemmaChevalley} we have
 \begin{align}
 \label{bracket_G2}
     [X_{-\alpha_1},X_{-2\alpha_1-\alpha_2}] = X_{-3\alpha_1-\alpha_2},
 \end{align}
 because $-\alpha_1-(-2\alpha_1-\alpha_2))$ and $-\alpha_1-2(-2\alpha_1-\alpha_2))$ are roots while $-\alpha_1-3(-2\alpha_1-\alpha_2))$ is not, hence the structure constant is $3=1$. If $\gamma= \alpha_1$, by symmetry with respect to the Weyl group  $\{ \pm s_{\alpha_2}\}$ of the Levi factor of $P^{\alpha_1}$ we have that $\mathfrak{g}_{-\alpha_1-\alpha_2}$ is also contained in $\mathfrak{s}$, hence either $\mathfrak{s}= \mathfrak{l}$ or it also contains $\mathfrak{g}_{-2\alpha_1-\alpha_2}$. The equality (\ref{bracket_G2}) together with \Cref{lem:addlong} then imply $\mathfrak{s}= \Lie G$. The same reasoning applies when starting by  $\gamma = -\alpha_1-\alpha_2$. On the other hand, starting by $\gamma= 2\alpha_1+\alpha_2$ implies that either $\mathfrak{s}= \mathfrak{h}$, or it contains also $\mathfrak{g}_{-\alpha_1}\oplus \mathfrak{g}_{-\alpha_1-\alpha_2}$, from which we conclude again - by (\ref{bracket_G2}) and \Cref{lem:addlong} - that $\mathfrak{s}= \Lie G$.
\end{proof}

\begin{definition}
\label{def:PhPl}
Let us fix the following notation for the rest of this Section:
\begin{enumerate}[(1)]
    \item $H\defeq (U_{-2\alpha_1-\alpha_2})_1$ is the subgroup of height one such that $\Lie H = \mathfrak{g}_{-2\alpha_1-\alpha_2}$, i.e. $\mathfrak{h} = \Lie P^{\alpha_1} \oplus \Lie H$ ;
    \item $L\defeq (U_{-\alpha_1})_1 \cdot (U_{-\alpha_1-\alpha_2})_1$ is the subgroup of height one such that $\Lie L = \mathfrak{g}_{-\alpha_1}\oplus \mathfrak{g}_{-\alpha_1-\alpha_2}$, i.e. $\mathfrak{l} = \Lie P^{\alpha_1} \oplus \Lie L$ ;
    \item $P_{\mathfrak{h}}$ the parabolic subgroup generated by $P^{\alpha_1}$ and $H$; 
    \item $P_{\mathfrak{l}}$ the parabolic subgroup generated by $P^{\alpha_1}$ and $L$.
\end{enumerate}
\end{definition}

Let us notice that $\mathfrak{g}_{-\alpha_1}$ and $\mathfrak{g}_{-\alpha_1-\alpha_2}$ commute, so that $L$ is the direct product of the Frobenius kernels defining it. 

\begin{remark}
The two parabolic subgroups $P_{\mathfrak{h}}$ and $P_{\mathfrak{l}}$ are \emph{exotic} in the sense that they cannot be of the form $(\ker \varphi) P^\alpha$ for some isogeny $\varphi$, since when $p=2$ the only noncentral isogenies in type $G_2$ are iterated Frobenius homomorphisms (see \Cref{factorisation_isogenies}). 
\end{remark} 

In the following part we investigate what the homogeneous spaces having as stabilizer respectively $P_{\mathfrak{h}}$ and $P_{\mathfrak{l}}$ are isomorphic to, as varieties.


\subsubsection{Parabolic $P_{\mathfrak{h}}$}

\begin{proposition}
    \label{prop:QtoP5}
    Let $G$ be simple of type $G_2$ in characteristic $p=2$ and $P_{\mathfrak{h}}$ the parabolic subgroup of \Cref{def:PhPl}. Then the quotient morphism $G/P^{\alpha_1} \longrightarrow G/P_{\mathfrak{h}}$ is the natural projection
    \[
   \proj^6 \supset Q\defeq \{ x^2_3 +x_2x_4+x_1x_5+x_0x_6 = 0\} \rightarrow \proj^5, \quad [x_0 \colon \ldots \colon x_6] \mapsto [x_0 \colon x_1 \colon x_2 \colon x_4 \colon x_5 \colon x_6].
    \]
    In particular, the homogeneous space $G/P_{\mathfrak{h}}$ is isomorphic as a variety to $\proj^5= \PSp_6/P^{\alpha_1}$.
\end{proposition}

In order to construct this morphism, we will see the group $G$ as the automorphism group of an octonion algebra - see the Appendix for more details - which is
\[
\mathbb{O} = \left\{ (u,v) \colon u,v \text{ are } 2\times 2 \text{ matrices}\right\},
\]
with basis $(e_{11},e_{12},e_{21},e_{22},f_{11},f_{12},f_{21},f_{22})$, recalled in Subsection \ref{G2_description}, 
with unit $e=(1,0) = e_{11}+e_{22}$, and which is equipped with a norm
\[
q(u,v) = \det (u) + \det(v).
\]
An embedding of the group $G_2$ into $\Sp_6$ can be seen as follows: let us consider its action on the vector space
\[
V \defeq e^\perp = \{ (u,v) \colon \det(1+u) +\det(u) = 1\}
\]
as in (\ref{e_orthogonal}). Since $p= 2$, we have that $e \in V$ hence the group $G$ also acts on the quotient $W\defeq V/ke$, which has dimension $6$. By (\ref{T_typeG}) in the Appendix, a maximal torus $T$ of $G$ - with respect to the basis $(f_{12},f_{11},e_{12},e_{21},f_{22},f_{21})$ of $W$ - is given by
\[
{\Gm}^2 \ni (a,b) \longmapsto \diag(a,a^{-1}b^{-1},a^2b, a^{-2}b^{-1}, ab,a^{-1}) = t \in T \subset \GL_6.
\]
Let us recall that the basis of simple roots we fix is $\alpha_1(t) \defeq a$ and $\alpha_2(t) \defeq b$, hence $V$ has the following decomposition in weight spaces :
\begin{align*}
V_0 = ke, \, & V_{\alpha_1} = kf_{12}, \, V_{-\alpha_1}= kf_{21}, \, V_{\alpha_1+\alpha_2} = kf_{22},\\
& V_{-\alpha_1-\alpha_2}= kf_{11}, \, V_{2\alpha_1+\alpha_2} = ke_{12}, \, V_{-2\alpha_1-\alpha_2} = ke_{21}.
\end{align*}
This way, $T$ can be identified with the maximal torus in \cite[page 13]{Heinloth}: in Heinloth's description of the embedding $G \subset \Sp_6$ in characteristic $2$, given by the action on 
\[
W = V/ke = W_{\alpha_1} \oplus W_{-\alpha_1-\alpha_2} \oplus W_{2\alpha_1+\alpha_2} \oplus W_{-2\alpha_1-\alpha_2} \oplus W_{\alpha_1+\alpha_2} \oplus W_{-\alpha_1},
\]
the group $G$ is generated by the two following copies of $\GL_2$ :
\[
\theta_1 \colon A \longmapsto \begin{pmatrix} A && \\ & A^{(1)}\det A^{-1} & \\ && A\end{pmatrix} \quad \text{and} \quad \theta_2 \colon B \longmapsto \begin{pmatrix} \det B^{-1} &&& \\ & B && \\ && B & \\ &&& \det B\end{pmatrix},
\]
where $A^{(1)}$ denotes the Frobenius twist applied to $A$.

\begin{lemma}
    When considering the action of $G$ on $\proj(V) = \proj^6$, we have
    \[
    \Stab_G([V_{2\alpha_1+\alpha_2}]) = P^{\alpha_1}.
    \]
\end{lemma}

\begin{proof}
    First, let us prove that $P^{\alpha_1}$, which is generated by $T$, $U_{\pm \alpha_2}$ and $U_{\alpha_1}$, fixes $V_{2\alpha_1+\alpha_2} = ke_{12}$. Clearly the torus does; moreover, the computation of the respective actions of $u_{-\alpha_2}(\lambda)$, $u_{\alpha_2}(\lambda)$  and $u_{\alpha_1}(\lambda)$ on $V$, done in \Cref{rootsubspaces:G2}, shows that all three fix \[[e_{12}] = [0\colon 0 \colon 1 \colon 0 \colon 0 \colon 0 \colon 0].\]
    This proves that $P^{\alpha_1} \subset S \defeq \Stab_G([V_{2\alpha_1+\alpha_2}])$. To prove the reverse inclusion, let us remark that no nontrivial subgroup of $U_{-\alpha_1}$ and of $U_{-2\alpha_1-\alpha_2}$ fixes $[e_{12}]$: again by Remark \ref{rootsubspaces:G2}, we have
    \[
        u_{-\alpha_1}(\lambda) \cdot e_{12} = e_{12} + \lambda f_{22} \quad \text{and} \quad u_{-2\alpha_1-\alpha_2}(\lambda) \cdot e_{12} = e_{12} + \lambda^2 e_{21},
    \]
    thus $U_{-\alpha_1} \cap S = U_{-2\alpha_1-\alpha_2} \cap S = 1$.\\
    At this point, we know that $\Lie P^{\alpha_1} \subset S$, hence by \Cref{cor:hl}, $\Lie S$ is either equal to $\Lie P^{\alpha_1}$, to  $\mathfrak{h}$, to $\mathfrak{l}$ or to $\Lie G$. However, $U_{\alpha_1} \cap S = 1$ means $\mathfrak{g}_{-\alpha_1} $ is not contained in $\Lie S$, hence the latter cannot be equal to $\mathfrak{l}$ nor to $\Lie G$. Analogously, $U_{-2\alpha_1-\alpha_2} \cap S = 1$ means $\mathfrak{g}_{2\alpha_1+\alpha_2}$ is not contained in $\Lie S$, hence $\Lie S $ cannot be equal to $\mathfrak{h}$. This means that $\Lie S = \Lie P^{\alpha_1}$ hence $S= P^{\alpha_1}$ as wanted.
\end{proof}

We can now conclude part of the proof of \Cref{prop:QtoP5}. First, let us recall that we are working with the basis $(f_{12}, f_{11}, e_{12}, e, e_{21}, f_{22}, f_{21})$ on $V$, giving homogeneous coordinates $[x_0 \colon \cdots \colon x_6]$ on $\proj(V)$: the norm $q$ hence becomes
\[
q(x) = x^2_3 +x_2x_4+x_1x_5+x_0x_6,
\]
and its zero locus in $\proj^6$ is the quadric $Q$ of the Proposition. The point $[e_{12}]$ belongs to $Q$ while $[e]$ does not, and the quotient $W=V/ke$ corresponds to the projection $\proj^6\backslash \{[e]\} \longrightarrow \proj^5$. Moreover, we have
\[
\begin{tikzcd}
G/P^{\alpha_1} = G/\Stab_G([V_{2\alpha_1+\alpha_2}]) = G \cdot [e_{12}] \arrow[rr, hookrightarrow] && Q
\end{tikzcd}
\]
Since both are smooth irreducible projective of dimension $5$, they coincide. In particular, 
\[
\underline{\Aut}^0_{G/P^{\alpha_1}} = \underline{\Aut}^0_{Q} = \SO(V) = \SO_7
\]
is of type $B_3$, as stated in \Cref{demazure77}.\\
What is left to prove is that $G/P_{\mathfrak{h}} \simeq \proj^5$: to do this, we look at the action of $G$ on $W$.

\begin{lemma}
    When considering the action of $G$ on $\proj(W) = \proj^5$, we have
    \[
    \Stab_G([W_{2\alpha_1+\alpha_2}]) = P_{\mathfrak{h}}.
    \]
\end{lemma}

\begin{proof}
    Let $S^\prime$ be the stabilizer. From the above Lemma we know that $P^{\alpha_1}$ fixes $[V_{2\alpha_1+\alpha_2}]$, hence it also fixes $[W_{2\alpha_1+\alpha_2}]$. Moreover, by Remark (\ref{rootsubspaces:G2}) we have
    \[
    u_{-2\alpha_1-\alpha_2}(\lambda) \cdot [e_{12}] = [0 \colon 0 \colon 1 \colon \lambda^2 \colon 0 \colon 0] \quad \text{and} \quad u_{-\alpha_1}(\lambda) \cdot [e_{12}] = [0 \colon 0 \colon 1 \colon 0 \colon \lambda \colon 0],
    \]
    meaning that $U_{-\alpha_1} \cap S^\prime =1$, while
    \[
    H = u_{-2\alpha_1-\alpha_2}(\alpha_p) = U_{-2\alpha_1-\alpha_2} \cap S^\prime.
    \]
    In particular, this yields that on one side, $P_{\mathfrak{h}} \subset S^\prime$ hence $\mathfrak{h} \subset \Lie S$, and on the other side, $\mathfrak{g}_{-\alpha_1}$ is not contained in $\Lie S^\prime$. In particular by \Cref{cor:hl} $\Lie S^\prime = \mathfrak{h}$ and the only positive root $\gamma$ satisfying $1 \subsetneq U_{-\gamma} \cap S^\prime \subsetneq U_{-\gamma}$ is $-2\alpha_1-\alpha_2$, hence by \cite{Wenzel}
    \[
    U_{S^\prime}^- = \prod_{\gamma \in \Phi^+ \colon U_{-\gamma} \nsubset S^\prime} (U_{-\gamma} \cap S^\prime) = U_{-2\alpha_1-\alpha_2} \cap S^\prime = H,
    \]
    where $U_P^-$, following Wenzel's notation, denotes the infinitesimal unipotent subgroup given by the intersection of a parabolic subgroup $P$ with the unipotent radical of the opposite of $P_{\text{red}}$ with respect to the Borel $B$.
    Thus, we can conclude that $S^\prime = U_{S^\prime}^- \cdot S^\prime_{\text{red}} = H \cdot P^{\alpha_1}$, and the latter must coincide with $P_{\mathfrak{h}}$ by definition.
\end{proof}

\begin{corollary}
\label{lem:UPminus_H}
We have $P_{\mathfrak{h}} = H \cdot P^{\alpha_1}$. More precisely,
\[ 
U_{P_{\mathfrak{h}}}^- = P_{\mathfrak{h}} \cap R_u^-(P^{\alpha_1}) = P_{\mathfrak{h}} \cap U_{-2\alpha_1-\alpha_2} = H.
\]
\end{corollary} 
Now, let us consider the embedding
\[
\begin{tikzcd}
G/P_{\mathfrak{h}} = G/\Stab_G([W_{2\alpha_1+\alpha_2}]) = G \cdot [e_{12}] \arrow[rr, hookrightarrow] && \proj(W) = \proj^5 \end{tikzcd}.
\]
As before, since both are smooth irreducible projective of dimension $5$, they coincide. This gives as quotient map 
\begin{align}
    \label{proj}
    G/P^{\alpha_1} = Q \longrightarrow G/ (H\cdot P^{\alpha_1}) = \proj^5
\end{align}
the projection from $[e]$, which has degree $2$ equal to the order of $H$.


\subsubsection{Parabolic $P_{\mathfrak{l}}$}
\label{section_Pl}

Let us consider the homogeneous space $G/P_{\mathfrak{l}}$ and show that one can realize it in a concrete way using octonions. More precisely, considering the action of $G_2 \subset \Sp_6$ on $W=V/ke$, the parabolic subgroup $P_{\mathfrak{l}}$ is the stabilizer of a $3$-dimensional isotropic vector subspace of $W$, spanned by the root spaces associated to the short positive roots (see \Cref{end_X}). To do this, let us consider $\eta \defeq f_{12} \wedge f_{22} \wedge e_{12}$ as element of $\proj(\Lambda^3 V)$ and $\overline{\eta}$ the element of $\proj(\Lambda^3 W)$ given by the images in $W$ of the three vectors.

\begin{lemma}
\label{stabilizer}
    Let $G$ be simple of type $G_2$ in characteristic $p=2$ and $P_{\mathfrak{l}}$ the parabolic subgroup of \Cref{def:PhPl}. When considering the action of $G$ on $\proj(\Lambda^3 V)$ and $\proj(\Lambda^3 W)$ respectively, we have
    \[
    \Stab_G(\eta) = P^{\alpha_1} \quad \text{and} \quad \Stab_G(\overline{\eta}) = P_{\mathfrak{l}}.
    \]
\end{lemma}

\begin{proof}
Let us denote as $S$ and $S^{\prime\prime}$ the above stabilizers.\\
    First, let us prove that $P^{\alpha_1}$, which is generated by $T$, $U_{\pm \alpha_2}$ and $U_{\alpha_1}$, fixes the subspace $kf_{12}  \oplus kf_{22} \oplus ke_{12} \subset V$, whose elements are of the form $(w_0,0,w_2,0,0,w_5,0)$. The computations of Remark \ref{rootsubspaces:G2} in the Appendix give us the following :
    \begin{align*}
        u_{\alpha_2}(\lambda) \cdot (w_0,0,w_2,0,0,w_5,0) & = (w_0,0,w_2,0,0,\lambda w_0+w_5,0),\\
        u_{-\alpha_2}(\lambda) \cdot (w_0,0,w_2,0,0,w_5,0) & = (w_0+\lambda w_5,0,w_2,0,0,w_5,0)\\
        u_{\alpha_1}(\lambda) \cdot (w_0,0,w_2,0,0,w_5,0) & = (w_0,0,w_2+\lambda w_5,0, 0, w_5,0),
    \end{align*}
meaning that 
$P^{\alpha_1} \subset S$. Moreover, considering the action of the root subgroups associated to $-\alpha_1$, $-2\alpha_1-\alpha_2$ and $-\alpha_1-\alpha_2$, we have the following :
\begin{align*}
    u_{-\alpha_1}(\lambda) \cdot (w_0,0,w_2,0,0,w_5,0) & = (w_0,0,w_2,  \lambda w_0,0,\lambda w_2+w_5,\lambda^2 w_0),\\
    u_{-2\alpha_1-\alpha_2} (\lambda) \cdot (w_0,0,w_2,0,0,w_5,0) & = (w_0,\lambda w_0,w_2, \lambda w_2,\lambda^2 w_2, w_5, \lambda w_5),\\
    u_{-\alpha_1-\alpha_2} (\lambda) \cdot (w_0,0,w_2,0,0,w_5,0) & = (w_0+\lambda w_2,\lambda^2 w_5, w_2,\lambda w_5,0,w_5,0).
\end{align*}
These computations imply that $\Lie S$ has trivial intersection with the root subspaces associated to short negative roots. Thus by \Cref{cor:hl} $\Lie S = \Lie P^{\alpha_1}$, which allows to conclude that $S= P^{\alpha_1}$.\\
Next, let us consider the action of $G$ on the quotient $W=V/ke$. The second computation just above yields that the intersection $U_{-2\alpha_2-\alpha_1} \cap S^{\prime\prime}$ is trivial, hence $\mathfrak{g}_{-2\alpha_1-\alpha_2}$ is not contained in $\Lie S^{\prime\prime}$ and the latter cannot be equal to $\Lie G$ nor to $\mathfrak{h}$. The other two equalities imply that $U_{-\alpha_1} \cap S^{\prime\prime} = u_{-\alpha_1}(\alpha_p)$ and $U_{-\alpha_1-\alpha_2} \cap S^{\prime\prime} = u_{-\alpha_1-\alpha_2}(\alpha_p)$, meaning that $\Lie S^{\prime\prime} = \mathfrak{l}$. In particular, the positive roots $\gamma$ satisfying $1 \subsetneq U_{-\gamma} \cap S^\prime \subsetneq U_{-\gamma}$ are $\alpha_1$ and $\alpha_1+\alpha_2$ : by \cite{Wenzel}, we have
    \[
    U_{S^{\prime\prime}}^- = \prod_{\gamma \in \Phi^+ \colon U_{-\gamma} \nsubset S^{\prime\prime}} (U_{-\gamma} \cap S^{\prime\prime}) = (U_{-\alpha_1-\alpha_2} \cap S^{\prime\prime}) \cdot (U_{-\alpha_1} \cap S^{\prime\prime}) = L.
    \]
    Thus, we can conclude that $S^{\prime\prime} = U_{S^{\prime\prime}}^- \cdot S^{\prime\prime}_{\text{red}} = L \cdot P^{\alpha_1}$,
    and the latter must coincide with $P_{\mathfrak{l}}$ by definition.
\end{proof}
\begin{corollary}
\label{lem:UPminus_L}
We have $P_{\mathfrak{l}} = L \cdot P^{\alpha_1}$. More precisely,
\[ 
U_{P_{\mathfrak{l}}}^- = P_{\mathfrak{l}} \cap R_u^-(P^{\alpha_1}) = (P_{\mathfrak{l}} \cap U_{-\alpha_1-\alpha_2}) \cdot (P_{\mathfrak{l}} \cap U_{-\alpha_1}) = L.
\]
\end{corollary}

Next, let us determine an explicit equation for the variety $Q$, which will help us describe $X$ geometrically. Recall that - keeping the notation from \Cref{LieN_2} - the reduced parabolic subgroup associated to the short root $\alpha_3$ in type $B_3$, which is denoted $P_3 = P^{\alpha_3} \subset \SO_7$, is the stabilizer of an isotropic subspace of dimension $3$, hence 
\[
P_\mathfrak{l} = \Stab_G([f_{12}\wedge f_{22} \wedge e_{12}]) = G \cap P_3 = G \cap \Stab_{\SO_7} ([f_{12}\wedge f_{22} \wedge e_{12}]).
\]

The following holds in any characteristic and is a consequence of \cite[Lemma 1.3.2]{SV} and of the fact that by definition $\bar{x} = -x$ for all $x \in V$.
\begin{lemma}
    There is a well-defined alternating form 
    \[\nu \colon \Lambda^3 V \longrightarrow k, \quad  (x,y,z) \longmapsto \langle xy, z\rangle.\]
\end{lemma}

Clearly there is an inclusion 
\begin{align}
\label{G2included}
G_2 \subset \{ u \in \SO_7 \colon \nu (ux,uy,uz) = \nu(x,y,z) \, \text{ for all } x,y,z \in V\} \subset \SO_7.
\end{align}

\begin{proposition}  
\label{2X=Z}
    Let us denote as $H$ the hyperplane defined by $\nu=0$ in $\proj(\Lambda^3V)$ and as $H^\prime$ its intersection with the quotient $\SO_7/P_3$. Then we have the equality $H^\prime=2Q$ as Weil divisors on $\SO_7/P_3$.
\end{proposition}

\begin{proof}
    Recall that $Q$ is irreducible of dimension $5$ while $\SO_7/P_3$ is irreducible of dimension $6$. Let us consider $Q$ and $\SO_7/P_3$ respectively as the $G_2$-orbit and the $\SO_7$-orbit of the point $\eta$ 
    in $\proj(\Lambda^3V)$. First, we have $\nu(\eta) = \nu(f_{12},f_{22},e_{12}) = \langle f_{12} f_{22}, e_{12}  \rangle = \langle e_{12}, e_{12} \rangle = 0$, hence by (\ref{G2included}) we have $\nu(Q) =0$, which gives an inclusion of $Q$ into the intersection $H^\prime = H \cap \SO_7/P_3$. Moreover, $\SO_7/P_3$ is not contained in the hyperplane $H$: consider $u \in \SO_7$ given by permuting $f_{12}$ with $f_{21}$ and $f_{22}$ with $f_{11}$,
    then we have \[
    \nu(u \eta ) = \nu (uf_{12}, uf_{22}, ue_{12})= \nu(f_{21}, f_{11}, e_{12}) = \langle f_{21}f_{11}, e_{12} \rangle = \langle e_{21}, e_{12} \rangle= 1.
    \]
    By \cite[III, Corollary 7.9]{Hartshorne}, $H^\prime$ is connected because it is the intersection of a hyperplane in $\proj(\Lambda^3V)$ with a smooth projective variety. Moreover, it is equidimensional of dimension $5$: let us write it as $H^\prime= Q \cup H^{\prime\prime}$ where $H^{\prime\prime}$ is the union of all other irreducible components of $H^\prime$. Since $H^{\prime\prime}$ is $G_2$-stable, if it were not empty it should have trivial intersection with $Q$ because the latter is a $G_2$-orbit. This would contradict the connectedness of $H^\prime$. The variety $Q$ is hence the reduced part of the divisor $H^\prime$ in $\SO_7/P_3$.\\
    Next, let us give an explicit equation for $\nu$: let $x,y,z \in V$ and set
    \begin{align*}
        x & = a_0 f_{12} + a_1 f_{22} + a_2 e_{12} + a_3 e + a_4 f_{21} + a_5 f_{11}+ a_6 e_{21},\\
        y & = b_0 f_{12} + b_1 f_{22} + b_2 e_{12} + b_3 e + b_4 f_{21} + b_5 f_{11}+ b_6 e_{21},\\
        z & = c_0 f_{12} + c_1 f_{22} + c_2 e_{12} + c_3 e + c_4 f_{21} + c_5 f_{11}+ c_6 e_{21},
    \end{align*}
    for appropriate coefficients $a_i$, $b_i$, $c_i \in k$.
    Using the product of octonions - see table (\ref{table:octonions}) - one gets
    \begin{align*}
        xy & = (a_5b_1 + a_2b_4 + a_3b_3 + a_6b_0) e_{11} + (a_0b_1+a_2b_3 +a_3b_2+a_1b_0)e_{12} \\
        & + (a_5b_6 +a_3b_4 + a_4b_3+ a_6b_5)e_{21} + (a_0b_6+a_3b_3+a_4b_2+a_1b_5)e_{22} \\
        & + (a_5b_3+a_0b_4 +a_3b_5+ a_4b_0)f_{11} + (a_0b_3+a_5b_2 + a_2b_5 +a_3b_0)f_{12}\\
        & + (a_3b_6+a_4b_1+a_1b_4 +a_6b_3) f_{21} + (a_3b_1 + a_1b_3 +a_6b_2+a_2b_6)f_{22}.
    \end{align*}

Another direct computation gives
\begin{align*}
    \nu(x,y,z) & = a_0b_1c_4 + a_0b_4c_1 + a_1b_0c_4 +a_1b_4c_0 +a_4b_0c_1 + a_4b_1c_0\\
    & + a_0b_3c_6+a_0b_6c_3+a_3b_0c_6+a_3c_6b_0+a_6b_0c_3+a_6b_3c_0\\
    & + a_1b_3c_5+a_1b_5c_3+a_3b_1c_5+a_3b_5c_1+a_5b_1c_3+a_5b_3c_1\\
    & + a_2b_3c_4+a_2b_4c_3+a_3b_2c_4+a_3b_4c_2+a_4b_2c_3+a_4b_3c_2\\
    & + a_2b_5c_6+a_2b_6c_5+a_5b_2c_6+a_5b_6c_2+a_6b_2c_5+a_6b_5c_2.
\end{align*}
Let us exploit the above computation to obtain an equation for $Q$ in the open Bruhat cell $C$ of $\SO_7/P_3$, which is an affine space of dimension $6$ given by the orbit $U\cdot \eta$ where
\[U = 
\left\{ \begin{pmatrix}
    1 &&&&&&\\
    0 & 1 &&&&&\\
    0 & 0 & 1 &&&&\\
    a_3 & b_3 & c_3 & 1 &&&\\
    a_4 & b_4 & c_4 & h_4 & 1 &&\\
    a_5 & b_5 & c_5 & h_5 & 0 & 1 &\\
    a_6 & b_6 & c_6 & h_6 & 0 & 0 & 1
\end{pmatrix} \in \SO_7 \right\},
\]
where we order $V$ with respect to the basis $(f_{12},f_{22},e_{12}, e, e_{21}, f_{11}, f_{21})$ i.e. we exchange $f_{22}$ and $f_{11}$ from the usual ordering, in order to have an easier block decomposition with the $3$-dimensional subspace $\eta$ in the first block.\\
The condition of such a matrix being in $\SO_7$ implies $a_6=a_3^2$, $b_5= b_3^2$, $c_4 = c_3^2$, $b_6=a_5$, $c_6=a_4$, $c_5=b_4$ and $h_4=h_5=h_6=0$, hence for dimension reasons we have
\[
C = U \cdot \eta \simeq \left\{ \begin{pmatrix}
    1 &&&&&&\\
    0 & 1 &&&&&\\
    0 & 0 & 1 &&&&\\
    \lambda_1 & \lambda_2 & \lambda_3 & 1 &&&\\
    \lambda_4 & \lambda_5 & \lambda_3^2 & 0 & 1 &&\\
    \lambda_6 & \lambda_2^2 & \lambda_5 & 0 & 0 & 1 &\\
    \lambda_1^2 & \lambda_6 & \lambda_4 & 0 & 0 & 0 & 1
\end{pmatrix} \colon \lambda_i \in k \right\} \simeq \AA^6.
\]
By substituting in the above expression for $\nu$, we obtain that the intersection $C \cap H$ is given by the equation
\[
\lambda_3^2+\lambda_6\lambda_3+\lambda_2\lambda_4+\lambda_1\lambda_5 + \lambda_6\lambda_3+\lambda_1\lambda_5 +\lambda_4\lambda_2 +\lambda_6^2 +\lambda_1^2\lambda_2^2 = (\lambda_3+\lambda_6+\lambda_1\lambda_2)^2,
\]
hence $Q \cap C$ has equation $\lambda_3+\lambda_6+\lambda_1\lambda_2= 0$ and we obtain the equality $H^\prime= 2Q$.
\end{proof}

\begin{lemma}
\label{lemma:2L}
Consider the embedding of $Q = G/P^{\alpha_1}$ into $\proj(\Lambda^3V)$ and denote as $\mathcal{L}$ and $\mathcal{M}$ the unique ample generators respectively of the Picard group of $Q$ and of $\SO_7/P_3$. Then the equality $\mathcal{L}^{\otimes 2} = \mathcal{O}_Q(1)$ holds. 
\end{lemma}

\begin{proof}
In \Cref{2X=Z}, the embedding
    \begin{align}
    \label{XYproj}
        Q \longhookrightarrow Y\defeq \SO_7/P_3 \longhookrightarrow \proj(\Lambda^3V)
    \end{align}
    realises the divisor $2Q$ as a hyperplane section of the variety $Y$.\\
    Next, let us express the latter as $Y = \Spin_7/Q_3$, where $Q_3$ is the maximal reduced parabolic subgroup associated to the short simple root in $\Spin_7$. The line in $V$ corresponding to such a root is $V_{2\alpha_1+\alpha_2}=ke_{12}$. Since $\Spin_7$ is simply connected, the Picard group of $Y$ identifies with the group of characters of $Q_3$, which has a unique ample generator $\mathcal{M}$ with associated weight $\varpi_3$, which is the third fundamental weight of $\Spin_7$. Since $Y$ is a standard flag variety (with reduced stabilizer), the line bundle $\mathcal{M}$ is also very ample, and the same holds for its restriction to $X$. Considering the embedding (\ref{XYproj}), the weight at the base point of $\proj(\Lambda^3 V)$ is $2\varpi_3$ hence $\mathcal{O}_Y(1)$ has also associated weight $2\varpi_3$. Restricting to $X$ gives the equality 
    \begin{align}
        \label{2L}
        \left( \mathcal{M}_ {\vert Q} \right)^{\otimes 2} = \mathcal{O}_Q(1).
    \end{align}
    Now, let $\lambda$ be the weight associated to $\lambda$ : then there must be some $n\geq 1$ such that $\mathcal{L}^{\otimes n} = \mathcal{O}_Q(1)$, thus $n\lambda = 2\varpi$. The only two possibilities are $n=1$ and $n=2$. However, assuming $n=1$ would imply that $\mathcal{O}_Q(1)$ generates the Picard group of $X$, which is a contradiction with (\ref{2L}). Hence we can conclude that $n=2$ and in particular $\mathcal{L}$ is the restriction of $\mathcal{M}$ to $Q$, so we are done.
\end{proof}

\begin{corollary}
\label{Q_hyperplane}
    The variety $Q$ is a hyperplane section of $Y$ relative to the very ample line bundle $\mathcal{M}$.
\end{corollary}

\begin{proof}
    We can see the class of $Q$ as a divisor in $\Pic Y = \Z \mathcal{M}$ : this means that, under the embedding of $Y$ into $\proj(H^0(X,\mathcal{L})^\vee)$, it is the class of the zero locus of a global section. The equality $\mathcal{L}^{\otimes 2} = \mathcal{O}_Q(1)$ proved in \Cref{lemma:2L} means that the canonical section of $Q$, which is an element of $H^0(Y,\mathcal{M})$, is a square root of the section $\nu$ described in \Cref{2X=Z}.
\end{proof}

The above description of the variety $Q$ holds in any characteristic. The case of characteristic two is peculiar because there exists an embedding of $G_2$ into $\Sp_6$, together with the very special isogeny described in \Cref{N_SO}. We will now use these two ingredients to get a geometric description of $X$, starting from the above realisation of the variety $Q$ and the natural quotient morphism $Q \rightarrow X$, induced by the 
 inclusion of $P^{\alpha_1} = (P_{\mathfrak{l}})_{\text{red}}$ into $P_{\mathfrak{l}}$.\\

Let us consider the following commutative diagram, which is induced by the quotient $W=V/ke$ and the associated purely inseparable isogeny $\varphi \colon \SO(V) = \SO_7 \rightarrow \Sp_6= \Sp(W)$, with kernel $N \defeq N_{\SO_7}$. Let us recall that, by \Cref{stabilizer}, $Q$ is the $G_2$-orbit of the $3$-dimensional subspace defined by the short root vectors in $\Lambda^3 V$, while $X$ is the $G_2$-orbit of the $3$-dimensional subspace defined by the short root vectors in $\Lambda^3 W$.

\begin{center}
\begin{tikzcd}
    Q = G_2/P^{\alpha_1} \arrow[rr, "g"] \arrow[d, hookrightarrow] && X = G_2 /P_{\mathfrak{l}} \arrow[d, hookrightarrow]\\
    Y \defeq \SO_7/P_3 \arrow[rr, "f"]  \arrow[d, hookrightarrow] && Z \defeq \Sp_6/P_3^\prime = \SO_7/(NP_3) \arrow[d, hookrightarrow]\\
    \proj(\Lambda^3 V) \arrow[rr, dashed] && \proj(\Lambda^3 W)
\end{tikzcd}
\end{center}

\begin{proposition}
\label{end_X}
    The line bundle $\mathcal{O}_Z(X)$ satisfies the equality $\Pic Z = \Z \, \mathcal{O}_Z(X)$. In particular, $X$ can be realised as a hyperplane section of $Z$ with respect to the unique (very) ample generator of $\Pic Z$.
\end{proposition}

\begin{proof}
    By \Cref{Q_hyperplane}, the Picard group of $Y$ is generated by $\mathcal{O}_Y(Q)$, hence $Q$ satisfies $Q \cdot \widetilde{C} = 1$, where we denote respectively as $\widetilde{C}$ and $C$ the Schubert curves (associated to the short simple root $\alpha_3$ in type $B_3$ and the long simple root $\alpha_3^\prime$ in type $C_3$) in $Y$ and in $Z$.\\
    The morphism $f$ is finite locally free of degree $8$, which corresponds to the order of
    \[
    NP_3/P_3 = N/(N \cap P_3).
    \]
    Indeed, as seen in \Cref{N_SO}, the subgroup $N \subset \SO_7$ has height one and Lie algebra $\mathfrak{n}=\mathfrak{g}_<$ of dimension $6$, hence the order of $N$ is $2^{6}$. On the other hand, the order of $N \cap P_3$ is $2^{3}$ because 
    \[
    \Lie (N \cap P_3) = \mathfrak{n} \cap \Lie P_3 = \mathfrak{g}_{-\varepsilon_1-\varepsilon_2} \oplus \mathfrak{g}_{-\varepsilon_1-\varepsilon_3} \oplus \mathfrak{g}_{-\varepsilon_2-\varepsilon_3}
    \]
    has dimension $3$. In particular, this means that $f_\ast f^\ast X = 8X$ seen as elements of $\Pic Z$.\\
    On the other hand, $g$ is finite locally free of degree $4$: the latter is the order of $L$, the unipotent infinitesimal part of $P_{\mathfrak{l}}$. Thus we also have $f_\ast Q= 4X$ : putting the two equalities together implies $f^\ast X = 2Q$ in the Picard group of $Y$.\\
    Next we notice that $\alpha_3$ is a short root in type $B_3$, hence the very special isogeny acts as a Frobenius morphism on the corresponding copy of the additive group in $\SO_7$. In other words, the set theoretic equality $f(\widetilde{C}) = C$ becomes $f_\ast \widetilde{C} = 2C$ on $1$-cycles. In particular,
    \[
    2 = 2 Q \cdot \widetilde{C} = f^\ast X \cdot \widetilde{C} = X \cdot f_\ast \widetilde{C} = 2 X\cdot C.
    \]
    This last computation together with the fact that $\Pic Z \simeq \Z$ allows to conclude that the line bundle associated to $X$ generates the Picard group of $Z$.
\end{proof}

Up to this point we have realized the variety $X= G/P_{\mathfrak{l}}$ using octonions. In particular, this construction provides a new example (besides projective spaces and quadrics) of a hyperplane section $X$ of a homogeneous variety $(Z,\mathcal{L})$, such that $X$ is also homogeneous and $\mathcal{L}$ generates the Picard group of $Z$. 
One might ask whether \Cref{final_rank1} still holds for the variety $X$. Actually this is not the case, as illustrated in the following result.

\begin{proposition}
\label{prop:Pl}
    Let $G$ be simple of type $G_2$ in characteristic $p=2$ and $P_{\mathfrak{l}}$ the parabolic subgroup of \Cref{def:PhPl}. Then $G/P_{\mathfrak{l}}$ is not isomorphic, as a variety, to a quotient of the form $G^\prime /P^\alpha$ for any $G^\prime$ simple and $\alpha \in \Delta(G^\prime)$.
\end{proposition}

In particular, this means that \Cref{final_rank1} does not hold in this case.

\begin{lemma}
    \label{lem:dim5}
    Let $G^\prime$ be simple and let $\alpha$ be a simple root of $G^\prime$. If $\dim (G^\prime/P^\alpha) = 5$, then such a variety is either isomorphic to $Q \subset \proj^6$, to $\proj^5$ or to $G/P^{\alpha_2}$ where $G$ is of type $G_2$ and $\alpha_2$ is the long root.
\end{lemma}

\begin{proof}
 Let us recall that $\dim (G^\prime/P^\alpha) = \vert \Phi^+(G)\vert - \vert \Phi^+(L^\alpha)\vert$, where $L^\alpha = P^\alpha \cap (P^\alpha)^-$ is a Levi subgroup, hence so we can compute this quantity explicitly in each case.\\
 \textbf{Type $A_{n-1}$}: for $1 \leq m \leq n-1$, 
 \[
 \dim(G^\prime /P^{\alpha_m}) = m(n-m)=5
 \]
 when $(n,m)= (6,5)$ or $(6,1)$. In that case, $G^\prime/P^{\alpha_1} = G^\prime /P^{\alpha_5} \simeq \proj^5$.\\
 \textbf{Type $B_n$}: the number of positive roots is $n^2$.\\
 $\bullet$ For $1\leq m \leq n-1$, the Levi subgroup $P^{\alpha_m} \cap (P^{\alpha_m})^-$ is of type $A_{m-1} \times B_{n-m}$, so 
     \[
     \dim(G^\prime /P^{\alpha_m}) = n^2-\frac{m(m-1)}{2} - (n-m)^2 = m \left( \frac{1-m}{2} +2n-m\right) = 5
     \]
     which only has as positive integer solution the pairs $(n,m) = (4,5)$, which is absurd, and $(n,m)= (3,1)$. In that case, $G^\prime = \SO_7$ and by \Cref{demazure77} and \Cref{prop:QtoP5} we have $\SO_7/P^{\alpha_1} \simeq G/P^{\alpha_1} \simeq Q \subset \proj^6$.\\
$\bullet$ Considering the last simple root, $P^{\alpha_n} \cap (P^{\alpha_n})^-$ is of type $A_{n-1}$ and 
     \[
     \dim (G^\prime/P^{\alpha_n}) = n^2-\frac{n(n-1)}{2} = \frac{n(n+1)}{2}
     \]
     is never equal to $5$.\\
 \textbf{Type $C_n$}: the same computations as in type $B_n$ give $(n,m)= (3,1)$, meaning $G^\prime = \PSp_6$ and - again by \Cref{demazure77} - we have $\PSp_6/P^{\alpha_1} = \PSL_6/P^{\alpha_1} \simeq \proj^5$.\\
 \textbf{Type $D_n$}: the number of positive roots is $n(n-1)$.\\
     $\bullet$ For $1\leq m \leq n-4$, the Levi subgroup is of type $A_{m-1} \times D_{n-m}$, so 
     \[
     \dim (G^\prime/P^{\alpha_m} ) = n(n-1) - \frac{m(m-1)}{2} - (n-m)(n-m-1) = m \left( \frac{1-m}{2} +2n-m-1 \right) = 5
     \]
     which has no positive integer solutions $(n,m)$.\\
     $\bullet$ For $m= n-3$, the Levi subgroup is of type $A_{n-4} \times A_3$, so
     \[
     \dim (G^\prime/P^{\alpha_m} ) = n(n-1) - \frac{(n-3)(n-4)}{2} - 6 = 5,
     \]
     which gives $n^2+5n= 34$  hence no integer solutions.\\
     $\bullet$ For $m = n-2$, the Levi subgroup is of type $A_{n-3} \times A_1 \times A_1$, so 
     \[
     \dim (G^\prime/P^{\alpha_m} ) = n(n-1) - \frac{(n-2)(n-3)}{2} -1-1 = 5,
     \]
     which gives $n^2+3n= 20$ hence no integer solutions.\\
     $\bullet$ For $m = n-1$ or $m= n$, the Levi subgroup is of type $A_{n-1}$, so
     \[
     \dim (G^\prime/P^{\alpha_m} ) = n(n-1) -\frac{n(n-1)}{2} = \frac{n(n-1)}{2} = 5,
     \]
     which is never equal to $5$.\\
 \textbf{Type $E_6$}: the number of positive roots is $36$, and the following table
\[
\begin{tabular}{ |c|c|c|c|c|c|c|}
\hline 
$E_6$ & $\alpha_1$ & $\alpha_2$ & $\alpha_3$ & $\alpha_4$ & $\alpha_5$ & $\alpha_6$\\
\hline
$L^\alpha$ & $D_5$ & $A_4 \times A_1 \times A_1$ & $A_2 \times A_2 \times A_1$ & $A_4 \times A_1$ & $D_5$ & $A_5$\\
\hline
$\vert \Phi^+ (L^\alpha) \vert$ & $20$ & $11$ & $7$ & $11$ & $20$ & $15$\\
\hline
$\dim(G/P^\alpha)$ & $16$ & $25$ & $29$ & $25$ & $16$ & $21$\\
\hline
\end{tabular}
\]
shows that the desired quantity is never equal to $5$.\\
  \textbf{Type $E_7$}: the number of positive roots is $63$ and the following table
\[
  \begin{tabular}{ |c|c|c|c|c|c|c|c|}
\hline 
$E_7$ & $\alpha_1$ & $\alpha_2$ & $\alpha_3$ & $\alpha_4$ & $\alpha_5$ & $\alpha_6$ & $\alpha_7$\\
\hline
$L^\alpha$ & $D_6$ & $A_5 \times A_1$ & $A_1 \times A_2 \times A_3$ & $A_4 \times A_2$ & $D_5 \times A_1$ & $E_6$ & $A_6$ \\
\hline
$\vert \Phi^+ (L^\alpha) \vert$ & $30$ & $16$ & $10$ & $13$ & $21$ & $36$ & $21$\\
\hline
$\dim(G/P^\alpha)$ & $33$ & $47$ & $53$ & $50$ & $42$ & $27$ & $42$\\
\hline
\end{tabular}
\]
shows that the desired quantity is never equal to $5$.\\
   \textbf{Type $E_8$}: the number of positive roots is $120$ and the following table
\[
  \begin{tabular}{ |c|c|c|c|c|c|c|c|c|}
\hline 
$E_8$ & $\alpha_1$ & $\alpha_2$ & $\alpha_3$ & $\alpha_4$ & $\alpha_5$ & $\alpha_6$ & $\alpha_7$ & $\alpha_8$\\
\hline
$L^\alpha$ & $D_7$ & $\! A_6 \times A_1\!$ & $\! A_1 \times A_2 \times A_4\!$ & $\!A_4 \times A_3\!$ & $\! D_5 \times A_2\! $ & $\! E_6 \times A_1\!$ & $E_7$ & $A_7$ \\
\hline
$\! \vert \Phi^+ (L^\alpha) \vert\! $ & $42$ & $22$ & $14$ & $16$ & $23$ & $37$ & $63$ & $28$\\
\hline
$\dim(G/P^\alpha)$ & $78$ & $98$ & $106$ & $104$ & $97$ & $83$ & $57$ & $92$\\
\hline
\end{tabular}
\]
shows that the desired quantity is never equal to $5$.\\
 \textbf{Type $F_4$}: a direct computation - see Subsection \ref{F4} - gives
 \[
 \dim(G^\prime/P^{\alpha_1}) = \dim (G^\prime / P^{\alpha_4}) = 15 \quad \text{and} \quad \dim(G^\prime / P^{\alpha_2}) = \dim (G^\prime /P^{\alpha_3}) = 20.
 \]
 \textbf{Type $G_2$}: as we already know, both $G/P^{\alpha_1} = Q$ and $G/P^{\alpha_2}$ have dimension $5$.
\end{proof}

\begin{lemma}
    \label{lem:intersection}
    The variety $X= G/P_{\mathfrak{l}}$ is not isomorphic to $\proj^5$ nor to $Q$.
\end{lemma}

\begin{proof}
    Let us consider the quotient map $f \colon G/P^{\alpha_1} \longrightarrow G/P_{\mathfrak{l}}$. By \Cref{lem:UPminus_L} we have $P_{\mathfrak{l}} = L \cdot P^{\alpha_1}$, hence the morphism $f$ is finite, purely inseparable and of degree $4$. Assume $X \simeq \proj^5$, then we get $f \colon Q \longrightarrow \proj^5$. Considering the line bundle $\mathcal{O}_Q(1) = {\mathcal{O}_{\proj^6}(1)}_{\vert Q}$, we have that $\Pic Q = \Z \cdot \mathcal{O}_Q(1)$ and $f^\ast \mathcal{O}_{\proj^5}(1) = \mathcal{O}_Q(m)$ for some $m >0$, since it has sections. Taking degrees, this gives on the left hand side
    \begin{align*}
    & f^\ast \mathcal{O}_{\proj^5}(1) \cdot f^\ast \mathcal{O}_{\proj^5}(1)  \cdot f^\ast \mathcal{O}_{\proj^5}(1)  \cdot f^\ast \mathcal{O}_{\proj^5}(1)  \cdot f^\ast \mathcal{O}_{\proj^5}(1) \\= & (\deg f) \left( \mathcal{O}_{\proj^5}(1) \cdot \mathcal{O}_{\proj^5}(1) \cdot \mathcal{O}_{\proj^5}(1) \cdot \mathcal{O}_{\proj^5}(1) \cdot \mathcal{O}_{\proj^5}(1) \right) = \deg f,
    \end{align*}
    so we get $\deg f = 4$. On the right hand side, this equals
    \begin{align*}
    & \mathcal{O}_Q(m) \cdot \mathcal{O}_Q(m) \cdot \mathcal{O}_Q(m) \cdot \mathcal{O}_Q(m) \cdot \mathcal{O}_Q(m)\\ & = \rho^\ast \mathcal{O}_{\proj^5}(m) \cdot \rho^\ast \mathcal{O}_{\proj^5}(m) \cdot \rho^\ast \mathcal{O}_{\proj^5}(m)  \cdot \rho^\ast \mathcal{O}_{\proj^5}(m)  \cdot \rho^\ast \mathcal{O}_{\proj^5}(m) 
    \\ & = (\deg \rho) (\mathcal{O}_{\proj^5}(m) \cdot   \mathcal{O}_{\proj^5}(m) \cdot   \mathcal{O}_{\proj^5}(m) \cdot   \mathcal{O}_{\proj^5}(m) \cdot   \mathcal{O}_{\proj^5}(m) ) = (\deg \rho \cdot m^5),
    \end{align*}
    which has degree $2m^5$, where $\rho$ is the projection of \Cref{prop:QtoP5}. Comparing degrees one gets $4 = 2m^5$, which is absurd.

    Now, let us assume instead that $X \simeq Q$, then $f \colon Q \longrightarrow Q$ is of degree $4$ and again $f^\ast \mathcal{O}_Q(1) = \mathcal{O}_Q(r)$ for some $r>0$: the analogous computation of degrees yields $8 = 2r^5$, which is again absurd.
\end{proof}

\begin{lemma}
    \label{lem:alpha2}
    The variety $X = G/P_{\mathfrak{l}}$ is not isomorphic to $G/P^{\alpha_2}$.
\end{lemma}

\begin{proof}
    Assume $X \simeq G/P^{\alpha_2}$, then the $G$-action on $X$ is given by a morphism $\theta \colon G \rightarrow \underline{\Aut}^0_{G/P^{\alpha_2}}$, the latter being equal to $G$ by \Cref{demazure77}. In particular, $\theta$ is an isogeny which satisfies $\theta^{-1}(P^{\alpha_2}) = P_{\mathfrak{l}}$. This means that there is some $g \in G(k)$ such that
    \[
    (\ker \theta) \cdot gP^{\alpha_2}g^{-1} = P_{\mathfrak{l}}.
    \]
    Since $\ker \theta$ is finite, taking the connected component of the identity and the reduced subscheme on both sides implies that $P^{\alpha_2}$ and $P^{\alpha_1}$ are conjugate in $G$, which is a contradiction.
\end{proof}

The above study of $P_{\mathfrak{h}}$ and $P_{\mathfrak{l}}$ does not complete the classification (in characteristic $2$) of homogeneous spaces having as stabilizer a parabolic subgroup whose reduced part is equal to $P^{\alpha_1}$. Let us consider a simple group $G$ of type $G_2$ and a nonreduced parabolic subgroup $P \subset G$ satisfying $P_{\text{red}} = P^{\alpha_1}$, in characteristic $p=2$. Moreover, let us assume that $\Lie P \neq \Lie G$, i.e. that $\Lie P$ is equal to $\mathfrak{h}$  (resp. $\mathfrak{l}$) and let us write it as $P = U_P^- \cdot P_{\text{red}}$, where $U_P^- = P \cap R_u^-(P_{\text{red}})$: in particular, it is contained in $U_{-2\alpha_1-\alpha_2}$ (resp. in $U_{-\alpha_1}\cdot U_{-\alpha_1-\alpha_2}$) and its order is $\vert U_P^-\vert = 2^n$ for some $n\geq 2$, the case $n= 1$ being $P_{\mathfrak{h}}$ treated above.





\subsubsection{End of classification}

Recall that we follow here the notation from \cite{Wenzel} : for a parabolic subgroup $P$, we denote as $U_P^-$ the intersection of $P$ with the unipotent radical of the opposite of $P_{\text{red}}$.



\begin{lemma}
\label{lemma_H}
    Let $P$ be a parabolic subgroup such that $\Lie P = \mathfrak{h}$. Then its unipotent infinitesimal part $U_P^-$ has height one.
\end{lemma}

\begin{proof}
    The reduced part of $P$ is $P^{\alpha_1}$, hence $U_P^-$ must be of the form $u_{-2\alpha_1-\alpha_2}(\alpha_{p^n})$ for some $n$. Let us assume that $n$ is at least equal to $2$. This means that there is some $\lambda \in \Ga$ such that $\lambda^2\neq 0$ and $u_{-2\alpha_1-\alpha_2}(\lambda) \in P$. Let us consider $\mu \in \Ga$ and compute the following commutator, which gives an element of $P$:
    \begin{align*}
        & (u_{-2\alpha_1-\alpha_2}(\lambda),u_{\alpha_1}(\mu)) = u_{-2\alpha_1-\alpha_2}(\lambda) u_{\alpha_1}(\mu) u_{-2\alpha_1-\alpha_2}(-\lambda) u_{\alpha_1}(-\mu)
         = (u_{-2\alpha_1-\alpha_2}(\lambda) u_{\alpha_1}(\mu))^2\\
        & = \left( \begin{pmatrix}
        1 & 0 & 0 & 0 & 0 & 0\\
        \lambda & 1 & 0 & 0 & 0 & 0\\
        0 & 0 & 1 & 0 & 0 & 0\\
        0 & 0 & \lambda^2 & 1 & 0 & 0\\
        0 & 0 & 0 & 0 & 1 & 0\\
        0 & 0 & 0 & 0 & \lambda & 1
    \end{pmatrix} 
    \cdot \begin{pmatrix}
        1 & 0 & 0 & 0 & 0 & \mu^2\\
        0 & 1 & 0 & \mu & 0 & 0\\
        0 & 0 & 1 & 0 & \mu & 0\\
        0 & 0 & 0 & 1 & 0 & 0\\
        0 & 0 & 0 & 0 & 1 & 0\\
        0 & 0 & 0 & 0 & 0 & 1
    \end{pmatrix} \right)^2 =
    \begin{pmatrix}
        1 & 0 & 0 & 0 & \lambda\mu^2 & 0\\
        0 & 1 & \mu\lambda^2 & 0 & 0 & \lambda \mu^2\\
        0 & 0 & 1 & 0 & 0 & 0\\
        0 & 0 & 0 & 1 & \mu\lambda^2 & 0\\
        0 & 0 & 0 & 0 & 1 & 0\\
        0 & 0 & 0 & 0 & 0 & 1
    \end{pmatrix}.
    \end{align*}
    The last quantity, when assuming $\mu^2=0$, coincides with $u_{-3\alpha_1-2\alpha_2}(\mu\lambda^2)$, which is a contradiction with the fact that $\Lie P = \mathfrak{h}$ does not intersect the root subspace associated to the root $-3\alpha_1-2\alpha_2$.
\end{proof}

\begin{lemma}
\label{lemma_L}
    Let $P$ be a parabolic subgroup such that $\Lie P = \mathfrak{l}$. Then its unipotent infinitesimal part $U_P^-$ has height one.
\end{lemma}

\begin{proof}
    As before, the reduced part of $P$ is $P^{\alpha_1}$. Moreover, the unipotent part $U_P^-$ has nontrivial and finite intersection with $U_{-\alpha_1}$ and $U_{-\alpha_1-\alpha_2}$, of height $m_1$ and $m_2$ respectively. Assuming the height of $U_P^-$ to be at least equal to $2$ means we have (up to a reflection by $s_{\alpha_2}$) that $m_2 \geq 2$. Thus, let $\lambda \in \Ga$ such that $\lambda^2\neq 0$ and $\mu \in \alpha_p$, so that $u_{-\alpha_1}(\mu) \in P$. Then the following commutator also belongs to $P$ :
    \begin{align*}
    & (u_{-\alpha_1-\alpha_2}(\lambda),u_{-\alpha_1}(\mu)) = u_{-\alpha_1-\alpha_2}(\lambda) u_{-\alpha_1}(\mu) u_{-\alpha_1-\alpha_2}(-\lambda) u_{-\alpha_1}(-\mu)
         = (u_{-\alpha_1-\alpha_2}(\lambda) u_{-\alpha_1}(\mu))^2\\
        & = \left(
        \begin{pmatrix}
        1 & 0 & \lambda & 0 & 0 & 0\\
        0 & 1 & 0 & 0 & \lambda^2 & 0\\
        0 & 0 & 1 & 0 & 0 & 0\\
        0 & 0 & 0 & 1 & 0 & \lambda\\
        0 & 0 & 0 & 0 & 1 & 0\\
        0 & 0 & 0 & 0 & 0 & 1
    \end{pmatrix}
    \cdot \begin{pmatrix}
        1 & 0 & 0 & 0 & 0 & 0\\
        0 & 1 & 0 & 0 & 0 & 0\\
        0 & 0 & 1 & 0 & 0 & 0\\
        0 & \mu & 0 & 1 & 0 & 0\\
        0 & 0 & \mu & 0 & 1 & 0\\
        \mu^2 & 0 & 0 & 0 & 0 & 1
    \end{pmatrix}
        \right)^2 =  \begin{pmatrix}
        1 & 0 & 0 & 0 & 0 & 0\\
        0 & 1 & \mu\lambda^2 & 0 & 0 & 0\\
        0 & 0 & 1 & 0 & 0 & 0\\
        \lambda\mu^2 & 0 & 0 & 1 & \mu\lambda^2 & 0\\
        0 & 0 & 0 & 0 & 1 & 0\\
        0 & 0 & \lambda\mu^2 & 0 & 0 & 1
    \end{pmatrix}.
        \end{align*}
        The last quantity coincides again with $u_{-3\alpha_1-2\alpha_2}(\mu\lambda^2)$, so we conclude as before.
\end{proof}

\begin{definition}
    For an integer $m \geq 0$, we denote as $H_m$ and $L_m$ the pull-back respectively of the subgroups $H$ and $L$ under an $m$-th iterated Frobenius morphism.
\end{definition}

\begin{proposition}
\label{parabolics_G2}
    Let $G$ be of type $G_2$ in characteristic two.\\
    Then the nonreduced parabolic subgroups of $G$ having $P^{\alpha_1}$ as reduced part are all of the form $G_mP^{\alpha_1}$, $H_mP^{\alpha_1}$ or $L_mP^{\alpha_1}$ for some $m \geq 0$.
\end{proposition}

\begin{proof}
    Let us consider such a subgroup $P$: its Lie algebra contains strictly $\Lie P^{\alpha_1}$, hence by \Cref{cor:hl} it is either equal to $\Lie G$, to $\mathfrak{h}$ or to $\mathfrak{l}$. If $\Lie P = \Lie G$, then there is a unique integer $m \geq 1$ such that the Frobenius kernel $G_m$ is contained in $P$ while $G_{m+1}$ is not. Considering the quotient $P^\prime \defeq P/G_m$ allows to assume that the Lie algebra of $P^\prime$ is strictly contained in the one of $G$. Next, if $\Lie P^\prime = \mathfrak{h}$ (resp. $\mathfrak{l}$), by \Cref{lemma_H} and \Cref{lemma_L}, we have that $P^\prime = P_{\mathfrak{h}}$ (resp. $P_{\mathfrak{l}}$). Thus, the parabolic $P$ is obtained from $P^{\alpha_1}$, $P_{\mathfrak{h}}$ or $P_{\mathfrak{l}}$ by pulling back with an iterated Frobenius morphism, and we are done.
\end{proof}

This completes the proof of \Cref{main_G2} and thus gives a complete classification of homogeneous varieties with Picard group $\Z$, which ends the proof of \Cref{main}.


\begin{remark}
\label{veryample_1}
    The last result, together with \Cref{end_X}, has as consequence the fact that any ample line bundle on an homogeneous variety of Picard rank one is very ample, without any assumption of type nor characteristic.
\end{remark}

\begin{remark}
Let us cite a reason why the geometry of a general projective homogeneous variety of Picard rank one may differ from the one of a generalized flag variety. This comes from the following generalization of a question of Lazarsfeld (see the end of \cite{Lazarsfeld}): if $X= G/P$ has Picard group isomorphic to $\Z$ and there is some surjective morphism $f \colon X \rightarrow Y$, then is $Y$ isomorphic to $X$? First, the iterated Frobenius morphisms $G/P \rightarrow G/G_mP$ do not give a counterexample. However, the maps 
    \[
    G/P^\alpha \longrightarrow G/N_GP^\alpha \quad \text{and} \quad G_2/P^{\alpha_1} \longrightarrow G_2/P_{\mathfrak{l}},
    \]
    defined respectively under the edge hypothesis and in characteristic $2$, are counterexamples. Both these examples are purely inseparable surjective morphisms: the next natural step would be adding the hypothesis for the morphism $f$ to be generically étale.
\end{remark}


\section{Consequences and higher Picard ranks}
\label{sec3}

We state here - in all types but $G_2$ - the desired modification of Wenzel's description of parabolic subgroups having as reduced subgroup a maximal one: they are all obtained by fattening the reduced part with the kernel of a noncentral isogeny, which generalizes to this setting the role of the Frobenius in characteristic $p \geq 5$. We then give a criterion to determine when two homogeneous spaces with Picard rank one have the same underlying variety. Moving on to a different setting, we consider spaces $G/P$ with higher Picard ranks. First, using the Białynicki-Birula decomposition allows us to describe explicitly classes of curves and divisors on such varieties. This description is then used to establish a family of examples - in Picard rank two - of homogeneous spaces which are not isomorphic as varieties to those having a stabilizer a parabolic subgroup of standard type, i.e. of the form $G_{m_1}P^{\alpha_1} \cap \ldots \cap G_{m_r} P^{\alpha_r}$ for some integers $m_i$ and simple roots $\alpha_i$.

\subsection{Consequences in rank one}

In the following Subsection we complete the study in the case of Picard rank one. Due to \Cref{prop:Pl}, let us make the assumption that the group $G$ is not of type $G_2$ in characteristic two. 

\subsubsection{Classification of parabolics with maximal reduced subgroup}

The results in the preceding Section allow us to complete the classification of parabolic subgroups having as reduced subgroup a maximal one. Let us recall that, by \cite{Wenzel}, if the Dynkin diagram of $G$ is simply laced or if $p > 3$, then such subgroups are of the form $P = G_m P^\alpha = (\ker F^m_G) P^\alpha$. 

\begin{proposition}
\label{classification_rank1}
Let $G$ be simple and $P$ be a parabolic subgroup of $G$ such that its reduced subgroup is maximal i.e. of the form $P_{\text{red}}= P^\alpha$ for some simple root $\alpha$. Then there exists an isogeny $\varphi$ with source $G$ such that
\[
P = (\ker \varphi) P^\alpha,
\]
unless $G$ is of type $G_2$ over a field of characteristic $p= 2$ and $\alpha$ is the simple short root.
\end{proposition}

\begin{proof}
First, Propositions \ref{LieN}, \ref{prop:CnPm}, 
\ref{LieN_5}, \ref{LieN_6} and Remarks \ref{rem_lifting} and \ref{rem_lifting2} imply that if $G$ is simple and $P_{\text{red}}$ is a maximal reduced parabolic subgroup, then either $P$ is reduced, or there exists a nontrivial noncentral normal subgroup of height one contained in $P$. This subgroup is either $H=N_G$ - when it is defined - or the image of the Frobenius kernel of the simply connected cover of $G$.\\
Now, let us consider the given parabolic $P$. If it is reduced, then there is nothing to prove. If it is nonreduced, then there is a noncentral subgroup $H_{(1)} \subset P$ normalized by $G$ and of height one. Let us denote as 
\[
\varphi_1 \colon G \longrightarrow G/H_{(1)} =:G_{(1)}
\]
the quotient morphism and replace the pair $(G,P)$ with $(G_{(1)}, P_{(1)}),$ where $P_{(1)}\defeq P/H_{(1)}$. This gives again a parabolic subgroup whose reduced subgroup is maximal, hence either $P_{(1)}$ is reduced or we can repeat the same reasoning to get an isogeny
\[
\varphi_2 \colon G \longrightarrow G/H_{(1)} \longrightarrow G/H_{(2)} =: G_{(2)}.
\]
Setting $P_{(2)} \defeq G/H_{(2)}$ we repeat the same reasoning again. This gives a sequence $(G_{(m)},P_{(m)})$ which ends with a reduced parabolic subgroup in a finite number of steps~: indeed, $P/P_{\text{red}}$ is finite so it is not possible to have an infinite sequence
\[
P_{\text{red}} \subsetneq H_{(1)}P_{\text{red}} \subsetneq \cdots \subsetneq H_{(m)}P_{\text{red}} \subsetneq \cdots \subsetneq P.
\]
Thus, let us set $H\defeq H_{(m)}$ for $m$ big enough and $\varphi \defeq \varphi_m$. Then we claim that $P = HP^\alpha = (\ker \varphi )P^\alpha$.\\
Both $H$ and $P^\alpha$ are subgroups of $P$ by construction, hence $HP^\alpha \subset P$. Quotienting by $H$ then gives
\[
HP^\alpha /H = P^\alpha/(H \cap P^\alpha) \subset P/H = P_{(m)}.
\]
Since both are reduced and have the same underlying topological space, they must coincide hence $HP^\alpha = P$.
\end{proof}

In particular, using our previous results on factorisation of isogenies, we can give a very explicit description of the kernels involved in the classification.

\begin{corollary}
\label{viola}
Keeping the above notation and the ones given in \Cref{def_N}, in the equality $P = (\ker \varphi) P^\alpha$, there are only the two following options:
\begin{enumerate}[(a)]
    \item either $\ker \varphi = \ker F_G^m = G_m$ is the Frobenius kernel,
    \item or, when such a subgroup is defined, $\ker \varphi = \ker (\pi_{G^{(m)}} \circ F^m_G) = N_{m,G}$.
\end{enumerate}
\end{corollary}

\begin{proof}
Let us first assume $G$ to be simply connected and consider the factorisation of the isogeny $\varphi$ given by \Cref{factorisation_isogenies}
\[
\begin{tikzcd}
\varphi \colon G \arrow[r, "\sigma"] & G^{\prime\prime} \arrow[r, "\rho", twoheadrightarrow] & G^\prime,
\end{tikzcd}
\]
where $\sigma = \pi \circ F^m$ and $\rho$ is central. Let $\alpha$, $\alpha^{\prime\prime}$ and $\alpha^\prime$ be simple roots of $G$, $G^{\prime\prime}$ and $G^\prime$ respectively, defined by the equalities
\[
P_{\text{red}}= P^\alpha, \quad \sigma(P^\alpha) = P^{\alpha^{\prime\prime}}, \quad \rho(P^{\alpha^{\prime\prime}}) = P^{\alpha^\prime}.
\]
Then
\[
P = (\ker \rho\sigma)P^\alpha = (\rho\sigma)^{-1} (P^{\alpha^\prime}) = \sigma^{-1}(P^{\alpha^{\prime\prime}}) = (\ker \sigma) P^\alpha,
\]
hence replacing $\varphi$ by $\sigma$ and $G^\prime$ by $G^{\prime\prime}$ gives one of the cases $(a)$ and $(b)$.\\
If $G$ is not simply connected, then we can consider the pull-back $\widetilde{P} \defeq \psi^{-1}(P) \subset \widetilde{G}$ in the simply connected cover. Applying the above reasoning to $\widetilde{P}$ yields
\[
\text{either } P = \psi(\widetilde{P}) = \psi(\widetilde{G}_m P^\alpha) = G_m P^\alpha, \quad \text{ or } P = \psi(\widetilde{P}) = \psi(N_{m,\widetilde{G}} P^\alpha) = N_{m,G} P^\alpha
\]
and we are done.
\end{proof}

\subsubsection{Comparing varieties of Picard rank one}
Let us start by considering a homogeneous variety $X =G/P$ under the action of a simple adjoint group $G$, having Picard group of rank one. Then set 
\[
G_0 \defeq \underline{\Aut}_X^0 \quad \text{and} \quad P_0 \defeq \Stab(x) \subset G_0,
\]
where $x \in X$ is a closed point and where we keep as notation for the automorphism group the same as in \Cref{autX}. Since the radical of $G_0$ is solvable and acts on the projective variety $X$, it has a fixed point: being normal in $G_0$, it is trivial. Analogously, the center of $G_0$ - which is contained in a maximal torus - is trivial. Moreover, the hypothesis $\Pic X = \Z$ together with \Cref{BB_flag} imply that $G_0$ is simple. So the group $G_0$ is simple adjoint and uniquely determined by the variety $X$, while $P_0$ is a parabolic subgroup whose reduced subgroup is maximal. Its conjugacy class is uniquely determined by $X$ up to an automorphism of the Dynkin diagram of $G_0$. Moreover, since the action of $G_0$ on $X$ is faithful, by \Cref{final_rank1_weak} we have that $P_0$ is reduced, hence of the form $P_0 = P^\alpha$ for a simple root $\alpha$.\\
Now, let us consider the action of $G$ on $X$: we want to relate in all possible cases the pair $(G,P)$ to the pair $(G_0,P_0)$. This will give us a way to determine, given two homogeneous spaces $G/P$ and $G^\prime/P^\prime$, whether they are isomorphic as varieties.

\begin{proposition}
If the pair $(G_0,P_0)$ is not exceptional in the sense of Demazure, then one of the following two cases holds :
\begin{enumerate}[(a)]
    \item $G= G_0$ and $P= G_m P^\alpha$, where $P^\alpha = P_0$ up to an automorphism of the Dynkin diagram of $G$,
    \item $G=(\overline{G_0})_{\text{ad}}$ and $P = N_{m,G}P^\alpha$, where $P^\alpha = \pi_{G_0}(P_0)/Z(\overline{G_0})$ up to an automorphism of the Dynkin diagram of $G$.
\end{enumerate}
If $(G_0,P_0)$ is exceptional, then there are two additional possibilities - denoting as $(G_0^\prime, P_0^\prime)$ the associated pair in the sense of Demazure :
\begin{enumerate}[(a')]
    \item $G= G_0^\prime$ and $P= G_m P^\alpha$, where $P^\alpha = P_0^\prime$ up to an automorphism of the Dynkin diagram of $G$,
    \item $G=(\overline{G_0^\prime})_{\text{ad}}$ and $P = N_{m,G}P^\alpha$, where $P^\alpha = \pi_{G_0^\prime}(P_0^\prime)/Z(\overline{G_0^\prime})$ up to an automorphism of the Dynkin diagram of $G$.
\end{enumerate}
\end{proposition}

\begin{proof}
Let us start by assuming that $(G_0,P_0)$ is not exceptional in the sense of Demazure. By Corollary \ref{viola}, either $P= G_m P^\alpha$ or $P= N_{m,G}P^\alpha$ for some $\alpha$. In the first case,
\[
X= G/G_mP^\alpha = G^{(m)} /(P^\alpha)^{(m)} \simeq G/P^\alpha
\]
as varieties, hence by \Cref{demazure77} this implies $G= \underline{\Aut}_X^0 = G_0$ and $P^\alpha = P^0$, leading to $(a)$. In the second case,
\[
X = G/ N_{m,G}P^\alpha = \overline{G}^{(m)} / (P^{\overline{\alpha}})^{(m)} \simeq \overline{G} /P^{\overline{\alpha}} = \overline{G}_{\text{ad}} / \left( P^{\overline{\alpha}} / Z(\overline{G}) \right)
\]
as varieties, hence by \Cref{demazure77} again $\overline{G}_{\text{ad}} = \underline{\Aut}_X^0 =G_0$ and $P_0 = P^{\overline{\alpha}} / Z(\overline{G})$. Considering their respective images by the very special isogeny of $\overline{G}_{\text{ad}}$ gives $(b)$.\\
If $(G_0,P_0)$ is exceptional in the sense of Demazure, \Cref{demazure77} allows for two additional cases: to get the conclusion it is enough to repeat the same reasoning by replacing $(G_0,P_0)$ with $(G_0^\prime,P_0^\prime)$.
\end{proof}

\subsection[Curves and divisors on $G/P$]{Curves and divisors on flag varieties}
\label{sec:divisors}

We give here an explicit basis for $1$-cycles and divisors modulo numerical equivalence on a flag variety $X= G/P$ of any Picard rank, with stabilizer $P$ not necessarily reduced. We do so by describing the cells of an appropriate Białynicki-Birula decomposition of $X$ in terms of the root system of $G$ and of the root system of a Levi subgroup of the reduced part of $P$. 

\subsubsection{Białynicki-Birula decomposition of a $G$-simple projective variety}
\label{BBdecomp}

Flag varieties are normal, projective and equipped with a $G$-action with a unique closed orbit, hence they form a particular class of simple $G$-projective varieties (for short, $G$-simple varieties), as in \cite{Brion02}. Let us review here the main definitions and results concerning the Białynicki-Birula decomposition of such varieties, then specialize to flag varieties. The original work on the subject is \cite{BB}; for a scheme-theoretic statement see \cite[Theorem 13.47]{Milne}. \\\\
Let us consider a $G$-simple variety $X$ and fix a cocharacter $\lambda \colon \Gm \rightarrow T$ such that 
\[
B = \{ g \in G \colon \lim_{t \rightarrow 0} \lambda(t) g \lambda(t^{-1}) \text{ exists in } G\},
\]
which is equivalent to the condition that $\langle \gamma, \lambda \rangle >0$ for all $\gamma \in \Phi^+$. This implies in particular that the set of fixed points under the $\Gm$-action induced by $\lambda$ coincides with the set $X^T$ of $T$-fixed points. Recall that the fixed-point scheme $X^T$ is smooth, see for example \cite[Theorem 13.1]{Milne}. For any connected component $Y \subset X^T$ there are an associated \emph{positive} and a \emph{negative stratum}, defined as
\[
X^+(Y) \defeq \{ x \in X \colon \lim_{t \rightarrow 0} \lambda(t) \cdot x \in Y\} \quad \text{and} \quad X^-(Y) \defeq \{ x \in X \colon \lim_{t \rightarrow 0} \lambda(t^{-1}) \cdot x \in Y\},
\]
equipped with morphisms
\begin{align*}
   &p^+ \colon X^+(Y) \rightarrow Y, \quad x \mapsto \lim_{t \rightarrow 0} \lambda(t) \cdot x,\\
    &p^- \colon X^-(Y) \rightarrow Y, \quad x \mapsto \lim_{t \rightarrow 0} \lambda(t^{-1}) \cdot x.
\end{align*}

\begin{theorem}[Białynicki-Birula decomposition]
\label{thm:BB}
    Let $X$ be a normal $G$-simple projective variety. Then the following hold: 
    \begin{itemize}
        \item The variety $X$ is the disjoint union of the positive (resp. negative) strata as $Y$ ranges over the connected components of $X^T$.
        \item The morphisms $p^+$ and $p^-$ are affine bundles.
        \item The strata $X^+(Y)$ and $X^-(Y)$ intersect transversally along $Y$.
    \end{itemize}
\end{theorem}

Let us remark that the assumption on $\lambda$ implies that positive strata are $B$- invariant, while negative strata are $B^-$-invariant. In particular, the unique open positive stratum $X^+$ is equal to $X^+(x^+)$ where $x^+$ is the unique $B^-$-fixed point, and analogously the unique open negative stratum $X^-$ is equal to $X^-(x^-)$ where $x^-$ is the unique $B$-fixed point. Let us recall here the main results from \cite{Brion02} in the case where $X$ is smooth.

\begin{theorem}
\label{michel}
    Let $X$ be a smooth $G$-simple projective variety, $x^-$ its $B$-fixed point, $X^-=X^-(x^-)$ the open negative cell and $D_1,\ldots,D_r$ the irreducible components of $X\backslash X^-$.
    \begin{enumerate}[(1)]
        \item $D_1,\ldots,D_r$ are globally generated Cartier divisors, whose linear equivalence classes form a basis of $\Pic (X)$.
        \item Every ample (resp. nef) divisor on $X$ is linearly equivalent to a unique linear combination of $D_1,\ldots,D_r$ with positive (resp. non-negative) integer coefficients. In particular, rational and numerical equivalence coincide on $X$ i.e. the natural map $\Pic (X) \rightarrow N^1(X)$ is an isomorphism.
        \item There is a unique $T$-fixed point $x_i^-$ such that $D_i$ is the closure of $X^-(x_i^-)$. Moreover, $x_i^-$ is isolated.
        \item Consider the $B$-invariant curve $C_i\defeq \overline{B\cdot x_i^-}$. Then 
        \[ D_i \cdot C_j = \delta_{ij},\] 
        meaning that $C_j$ intersects transversally $D_j$ and no other $D_i$.
        \item The convex cone of curves $\NE(X)$ is generated by the classes of $C_1,\ldots,C_r$, which form a basis of the rational vector space $N_1(X)_{\Q}$.
        \end{enumerate}
\end{theorem}

\subsubsection{Białynicki-Birula decomposition of flag varieties}

Let us now specialize to our case i.e. interpret the results of the above Section in terms of root systems. The first step consists in recalling the Bruhat decomposition of a flag variety with reduced stabilizer, i.e. $X=G/P_I$ where $I \subset \Delta$ is a basis for the root system of a Levi subgroup of $P_I$. In particular, for a simple root $\alpha$ the subgroup $P^\alpha$ - which has been widely used in the previous Sections - coincides with $P_{\Delta \backslash \{ \alpha \}}$. Let us fix a set of representatives $\dot{w} \in N_G(T)$, for $w \in W=W(G,T)$ and let us recall the following (see \cite[8.3]{Springer}).

\begin{theorem}[Bruhat decomposition]
\label{bruhat}
Let $G \supset B \supset T$ be a reductive group, a Borel subgroup and a maximal torus, and $W=W(G,T)$. Then the following hold.
\begin{enumerate}[(1)]
    \item $G$ is the disjoint union of the double cosets $Bw B$, for $w \in W$.
    \item Let $\Phi_w$ be the set of positive roots $\gamma$ such that $w^{-1}\gamma$ is negative. Then
    \[
    U_w \defeq \prod_{\gamma \in \Phi_w} U_\gamma
    \] is a subgroup of the unipotent radical of $B$, with the product being taken in any order.
    \item The map $U_w \times B \rightarrow BwB$ given by $(u,b) \mapsto u\dot{w}b$ is an isomorphism of varieties.
\end{enumerate}
\end{theorem}

This gives a decomposition of $G/B$ into the disjoint union of the cells $BwB/B$, which are isomorphic to $U_w$ i.e. to affine spaces of dimension equal to the length of $w$. Since we want to work with $G/P_I$ instead of $G/B$, we shall not consider the whole Weyl group but its quotient by the subgroup $W_I$ generated by the reflections corresponding to simple roots in $I$.
\begin{lemma}
\label{w_minimal}
In any left coset of $W_I$ in $W$ there is a unique element $w$ characterized by the fact that $wI \subset \Phi^+$ or by the fact that the element $w$ is of minimal length in $wW_I$.
\end{lemma}

\begin{proof}
    See \cite[Proposition 3.9]{BorelTits}.
\end{proof}

We denote the set of such representatives as $W^I$. In particular, denoting $w_0$ and $w_{0,I}$ the element of longest length of $W$ and $W_I$ respectively, then $w_0^I \defeq w_0w_{0,I}$ is the element of longest length in $W^I$. 

\begin{proposition}[Generalized Bruhat decomposition]
\label{generalized_bruhat}
 For a fixed $I \subset \Delta$, the group $G$ is the disjoint union of the double cosets $Bw P_I$, where $w$ ranges over the set $W^I$.   
\end{proposition}

In order to get a similar statement as $(3)$ in \Cref{bruhat}, let us consider for any $w \in W^I$ the sets 
\begin{align}
\label{lem:phi_I}
& \Phi_w^I \defeq \{ \gamma \in \Phi^+ \colon w^{-1}\gamma \notin \Phi^+ \text{ and } w^{-1} \gamma \notin \Phi_I\},\\
& \Phi_{w,I} \defeq \Phi_w \backslash \Phi^I_w = \Phi_w \cap \Phi_I^+.
\end{align}


\begin{lemma}
With the above notation, let us fix $w \in W^I$.
\begin{enumerate}[(1)]
    \item The groups $U_\gamma$, with $\gamma$ ranging over $\Phi_w^I$ (resp. $\Phi_{w,I}$) generate two subgroups of the unipotent radical of $B$, 
    \[
    U_w^I = \prod_{\gamma \in \Phi_w^I} U_\gamma \quad \text{and} \quad U_{w,I} = \prod_{\gamma \in \Phi_{w,I}} U_\gamma,
    \]
    with the product being taken in any order. 
    \item The product map $U_w^I \times P_I \rightarrow BwP_I$ given by $(u,h) \mapsto u\dot{w}h$ is an isomorphism of varieties. 
\end{enumerate}
\end{lemma}

\begin{proof}
    To prove $(1)$, let us recall that for any pair of roots $\gamma, \delta \in \Phi$ there exist constants $c_{ij}$ such that 
    \begin{align*}
    (u_\gamma(x), u_\delta(y)) = \prod_{i,j>0, \, i\gamma+j\delta \in \Phi} u_{i\gamma+j\delta} (c_{ij}x^iy^j), \qquad \text{for all } x,y \in \Ga
    \end{align*}
    (see \cite[Proposition 8.2.3]{Springer}). If $\gamma$ and $\delta$ are both in $\Phi^I_w$, then $w^{-1}(i\gamma+j\delta)$ is still negative and not belonging to $\Phi_I$, hence by \Cref{lem:phi_I} the product of the root subgroups with roots ranging over $\Phi^I_w$ is a group. The same reasoning holds for the second product.\\
    Moving on to $(2)$, let us consider an element $x \in BwP_I$. Let us fix an order on $\Phi_w^I = \{ \gamma_1,\ldots, \gamma_l\}$ and on $\Phi_{w,I}= \{\delta_1,\ldots,\delta_m\}$ . By \Cref{bruhat}, there are a unique $w^\prime \in W_I$, a unique $u= u_{\gamma_1}(x_1) \cdot \ldots \cdot u_{\gamma_l}(x_l) \in U_w^I$, a unique $u^\prime=u_{\delta_1}(y_1)\cdot \ldots \cdot u_{\delta_m}(y_m) \in U_{w,I}$ and a unique $b \in B$ such that $x= uu^\prime \dot{w} \dot{w}^\prime b \in Bww^\prime B$. Moreover, by \cite[8.1.12(2)]{Springer}, there exist constants $c_i \in k$ such that
    \[
    u^\prime \dot{w} = \left( \prod_{i=1}^m u_{\delta_i}(y_i) \right) \dot{w} = \dot{w} \left( \prod_{i=1}^m \dot{w}^{-1} u_{\delta_i}(y_i) \dot{w} \right) = \dot{w} \prod_{i=1}^m u_{w^{-1}\delta_i} (c_iy_i) =: \dot{w}u^{\prime\prime}
    \]
    Since $w^{-1}\delta_i$ is in $\Phi_I$ for all $i$, the product $u^{\prime\prime}$ is an element of $P_I$, as well as $h\defeq u^{\prime\prime}\dot{w}^\prime b$ because $w^\prime \in W_I$. This gives a unique way to write $x$ as product $u\dot{w}h$ for some $u \in U_w^I$ and $h \in P_I$.
\end{proof}

Next, let us go back to our original setting: consider a sequence $G \supset P \supset P_{\text{red}} = P_I \supset B \supset T$ and look at the map 
\[
\begin{tikzcd}
    \widetilde{X} \defeq G/P_I \arrow[rr, "\sigma"] && G/P =: X,
\end{tikzcd}
\]
in order to relate the geometry of $X$ to that of $\widetilde{X}$. The morphism $\sigma$ is finite, purely inseparable and hence a homeomorphism between the underlying topological spaces. Let us denote as $\tilde{o} \in \widetilde{X}$ and $o \in X$ the respective base points.

The decomposition of \Cref{generalized_bruhat} allows us to express the variety $\widetilde{X}$ as the disjoint union of the cells $BwP_I /P_I = Bw\tilde{o}$ as $w \in W^I$. Let us remark that $W^I$ corresponds to the set of isolated points under the $T$-action, i.e. that
\[
({\widetilde{X}})^T = \{ w\tilde{o} \, \colon w \in W/W_I\}
\]
and the same holds for $X$. It is hence natural if such a decomposition coincides with the Białynicki-Birula decomposition of \Cref{thm:BB}. This is useful because the advantage of the first one is that it is more explicit and easier to manipulate, while the second can be defined also on $X$, independently of the smoothness of the stabilizer. Let us denote as ${\widetilde{X}}^+_w$ (resp. $X_w^+$) the positive Białynicki-Birula strata associated to the $T$-fixed point $w\tilde{o}$ (resp. $wo$), and the analogous notation for negative strata.

\begin{lemma}
    For any $w \in W/W_I$, we have
    \[
    Bw\tilde{o} = {\widetilde{X}}_w^+ \quad \text{and} \quad Bwo = X_w^+.
    \]
\end{lemma}

\begin{proof}
    For the first equality, $w\tilde{o}$ belongs to $\widetilde{X}^+_w$ because it is a $T$-fixed point. Moreover, positive strata are $B$-invariant which means that $Bw\tilde{o} \subseteq \widetilde{X}^+_w$. The other inclusion comes from the fact that $\widetilde{X}$ can be expressed as the disjoint union of both the strata of the two decompositions with the same index set.\\
    Next, let us consider $Bwo = \sigma(Bw\tilde{o})$, which equals $\sigma(\widetilde{X}^+_w)$ by what we just proved. The inclusion $\sigma(\widetilde{X}^+_w) \subset X^+_w$ comes from the fact that $\sigma$ being $T$-equivariant respects the Białynicki-Birula decomposition, while the other inclusion is due to the fact that
    \[
    \bigsqcup_{w \in W^I} Bwo = X = \bigsqcup_{w \in W^I} X^+_w.
    \]
    because $\sigma$ is an homeomorphism.
\end{proof}

\begin{remark}
How can we visualize the morphism $\sigma$ on cells? By \Cref{generalized_bruhat}, the Bruhat cell associated to $w \in W^I$ in $\widetilde{X}$ is an affine space of dimension $l$, equal to the cardinality of $\Phi^I_w = \{\gamma_1,\ldots, \gamma_l\}$. Let us consider the integers $n_i$, which we recall are associated to the roots in $\Phi_w^I$ via the equality
\[
U_{-\gamma_i} \cap P = u_{-\gamma_i} (\alpha_{p^{n_i}}).
\]
If we denote as $Y_i$ the coordinate on the affine line given by $U_{\gamma_i}$, then the morphism $\sigma$ acts on such a line as an $n_i$-th iterated Frobenius morphism, hence its behavior on the cell $Bw\tilde{o} = \widetilde{X}^+_w$ can be summarized in the following diagram
\[
\begin{tikzcd}
    U_w^I \simeq \widetilde{X}_w^+ = \Spec k[Y_1,\ldots, Y_l] \simeq \AA^l \arrow[dd, "\sigma"] \arrow[rr, hookrightarrow] && G/P_I \arrow[dd, "\sigma"] \\&&\\
    X_w^+ = \Spec k[Y_1^{p^{n_1}}, \ldots, Y_l^{p^{n_l}}] \simeq \AA^l \arrow[rr, hookrightarrow] && G/P
\end{tikzcd}
\]
\end{remark}

We reinterpret all the ingredients of \Cref{michel} in order to specialize and state it in the case of flag varieties. First, $X=G/P$ is indeed smooth, projective and $G$-simple. Its unique $B$-fixed point is $x^-=o$ the base point, which gives as open cell $B^- o = B w_0o = B w_0^I o = X^+_{w_0^I}$. Moreover, the irreducible components of $X\backslash X_{w_0^I}$ are the closures of the strata of codimension one, i.e. the cells $B wo$ with $w \in W^I$ of length $l(w) = l(w_0^I)-1$. Those are exactly of the form $w= w_0s_\alpha w_{0,I}$ for $\alpha \in \Delta \backslash I$, since for $\alpha \in I$ we have that $w_0s_\alpha$ is in the same left coset as $w_0^I$. In particular, the divisors in the statement of \Cref{michel} are 
\[
D_\alpha = \overline{B w_0s_\alpha w_{0,I} o} = \overline{B w_0s_\alpha o} = \overline{B^- s_\alpha o}, \quad \text{for } \alpha \in \Delta \backslash I,
\]
hence the unique $T$-fixed point $x_\alpha^-$ such that $D_\alpha$ is the closure of $X^-(x_\alpha^-)$ is $x_\alpha^-= s_\alpha o$, and we are led to consider the $B$-invariant curves 
\[
C_\alpha = \overline{B x_\alpha^-} = \overline{B s_\alpha o}.
\]
We are now able to reformulate the results of \Cref{BBdecomp} in the following:

\begin{theorem}
\label{BB_flag}
    Let us consider a sequence $G \supset P \supset P_{\text{red}}=P_I \supset B \supset T$ and let $X=G/P$ with base point $o$ and open cell $X^-= B^-o$. Then the following hold:
    \begin{enumerate}[(1)]
        \item The irreducible components of $X \backslash X^-$ are the closures $D_\alpha$ of the negative cells associated to the points $s_\alpha o$ for $\alpha \in \Delta \backslash I$. Moreover, they are globally generated Cartier divisors, whose linear equivalence classes form a basis for $\Pic (X)$. 
        \item Every ample (resp. nef) divisor on $X$ is linearly equivalent to a unique linear combination of the $D_\alpha$'s with positive (resp. non-negative) integer coefficients. In particular, the natural map $\Pic(X) \rightarrow N^1(X)$ is an isomorphism.
        \item Considering the $B$-invariant curves $C_\alpha$'s defined above, the intersection numbers satisfy $D_\alpha \cdot C_\beta = \delta_{\alpha\beta}$.
        \item The convex cone of curves $\NE(X)$ is generated by the classes of the $C_\alpha$'s, which form a basis of the rational vector space $N_1(X)_{\Q}$.
    \end{enumerate}
\end{theorem}

\subsubsection{Contractions}

\Cref{BB_flag} tells us in particular that the Picard group of a flag variety $X=G/P$ is a free $\Z$-module of rank the number of simple roots not belonging to the root system of a Levi factor of $P_{\text{red}}$. This gives a motivation to the study, done in \Cref{sec2}, of parabolic subgroups having maximal reduced part. In order to move on to higher ranks by exploiting the previous results in rank one, we adopt the following strategy : we define a finite collection of 
 morphisms which behave nicely, arise naturally from the variety $X$, and whose targets are homogeneous spaces of Picard rank one. As a first step towards such a construction, we recall the notion of a contraction between varieties and some of its properties.

\begin{definition}
Let $X$ and $Y$ be varieties over an algebraically closed field $k$. A \emph{contraction} between them is a proper morphism $f \colon X \rightarrow Y$ such that $f^\hash \colon \mathcal{O}_Y \longrightarrow f_\ast \mathcal{O}_X$ is an isomorphism.
\end{definition}
We will make use of the following results (stated here for reference).

\begin{theorem}
\label{contractions}
Let $f \colon X \rightarrow Y$ be a contraction between projective varieties over $k$. 
Then $f$ is uniquely determined, up to isomorphism, by the convex subcone $\NE(f)$ of $\NE(X)$ generated by the classes of curves which it contracts. Moreover, if $Y^\prime$ is a third projective variety and $f^\prime \colon X \rightarrow Y^\prime$ satisfies $\NE(f) \subset \NE(f^\prime)$, then there is a unique morphism $\psi \colon Y \rightarrow Y^\prime$ such that $f^\prime = \psi \circ f$.
\[
\begin{tikzcd}
X \arrow[rr,"f"] \arrow[rd, "f^\prime"] && Y \arrow[dl, "\psi", dotted]\\
& Y^\prime & 
\end{tikzcd}
\]
\end{theorem}

\begin{proof}
See \cite[Proposition $1.14$]{Debarre}.
\end{proof}

\begin{theorem}[Blanchard's Lemma]
\label{blanchard}
Let $f \colon X \rightarrow Y$ be a contraction between projective varieties over $k$. Assume 
that $X$ is equipped with an action of a connected algebraic group $G$. Then there exists a unique $G$-action on $Y$ such that the morphism $f$ is $G$-equivariant.
\end{theorem}

\begin{proof}
See \cite[$7.2$]{Brion17}.
\end{proof}
The following construction is done here for any globally generated line bundle and is then applied to $\mathcal{O}_X(D_\alpha)$ to define the desired family of contractions.

\begin{lemma}
\label{ProjY}
Let $X$ be a projective variety over $k$ and $\mathcal{L}$ a line bundle over $X$ which is generated by its global sections. Then
\begin{enumerate}[(a)]
    \item There is a well-defined contraction
    \[
    f \colon X \longrightarrow Y \defeq \Proj \bigoplus_{n=0}^\infty H^0(X,\mathcal{L}^{\otimes n}).
    \]
    \item A curve $C$ in $X$ is contracted by $f$ if and only if $\mathcal{L} \cdot C =0$.
\end{enumerate}
\end{lemma}

\begin{proof}
$(a)$ : Let us denote as $S$ the graded ring on the right hand side and denote as $S_d = H^0(X,\mathcal{L}^{\otimes d})$ its homogeneous part of degree $d$. The schemes $X$ and $Y = \Proj S$ are covered by the open subset
\[
D(t) = \Spec\left( \bigcup_{n=0}^\infty \frac{H^0(X, \mathcal{L}^{\otimes n})}{t^n}\right) \quad \text{and} \quad  X_t \defeq \{ x \in X , \, t_x \notin \mathfrak{m}_x \mathcal{L}_x\} = X \backslash Z(t),
\]
for $t$ homogeneous in $\oplus_{d\geq 1} S_d$, because by hypothesis $\mathcal{L}$ is globally generated. This allows to define $f$ via the inclusion
\begin{align}
\label{inclusion_rings}
\bigcup_{n=0}^\infty \frac{H^0(X,\mathcal{L}^{\otimes nd}) }{t^n} \subset \mathcal{O}_X(X_t), \quad \text{ for } t \in S_d.
\end{align}
Moreover, \cite[II, Lemma $5.14$]{Hartshorne} - applied to the coherent sheaf $\mathcal{O}_X$ and the line bundle $\mathcal{L}^{\otimes nd}$ - implies that (\ref{inclusion_rings}) is an equality, which gives the condition $f_\ast \mathcal{O}_X \simeq \mathcal{O}_Y$.\\
$(b)$ : Let us consider the sheaf $\mathcal{O}_Y(1)$ defined as in \cite[II, Proposition $5.11$]{Hartshorne}, fix some global section $s \in H^0(X,\mathcal{L})$ and assume the open set on which it does not vanish is affine i.e. $X_s = \Spec \mathcal{O}_X(X_s)$. Recall that we have the trivialization $\mathcal{L}_{\vert X_s} \simeq s \mathcal{O}_{X_s}$, so considering sections over $X_s$ gives
\[
H^0(X_s,f^\ast \mathcal{O}_Y(1)) = \left( \bigcup_{n=0}^\infty \frac{S_{n+1}}{s^n} \right) \otimes_{\bigcup_{n=0}^\infty \frac{S_n}{s^n} } \mathcal{O}_X(X_s) = \frac{s^{n+1}\mathcal{O}_X(X_s)}{s^n} = H^0(X_s,\mathcal{L}).
\]
By covering $X$ with the open sets $X_s$ as $s$ ranges over the global sections, this gives the condition $f^\ast \mathcal{O}_Y(1) =\mathcal{L}$, hence 
\begin{align*}
H^0(X,\mathcal{L}) & = H^0(X, f^\ast \mathcal{O}_Y(1) )= H^0(Y, f_\ast f^\ast \mathcal{O}_Y(1)) = H^0(Y, \mathcal{O}_Y(1) \otimes_{\mathcal{O}_Y} f_\ast \mathcal{O}_X)\\ & = H^0(Y, \mathcal{O}_Y(1)),
\end{align*}
where the last equality comes from $f$ being a contraction. In particular, $\mathcal{O}_Y(1)$ is ample over $Y$, thus it must have strictly positive intersection with any effective $1$-cycle by Kleinman's criterion. In other words, given a nonzero class $C \in \NE(X)$, $f_\ast C =0$ if and only if
\[
0 = \mathcal{O}_Y(1) \cdot f_\ast C = f^\ast \mathcal{O}_Y(1) \cdot C = \mathcal{L} \cdot C,
\]
by the projection formula, and we are done.
\end{proof}

Before going back to our particular case, let us prove a criterion for a morphism between homogeneous spaces to be a contraction.

\begin{lemma}
\label{criterion_contraction}
Consider a chain of algebraic groups $H \subset H^\prime \subset G$ over $k$. The morphism $f \colon G/H \longrightarrow G/H^\prime$
is a contraction if and only if $H^\prime/H$ is proper over $k$ and $\mathcal{O}(H^\prime/H)=k$.
\end{lemma}

\begin{proof}
Let us consider $q \colon G \rightarrow G/H$ and $q^\prime \colon G \rightarrow G/H^\prime$ to be the quotient maps and $m \colon G \times H/H^\prime \rightarrow G/H$ the morphism given by the group multiplication and then by quotienting by $H$: by \cite[Proposition $7.15$]{Milne} we have a cartesian square
\[
\begin{tikzcd}
G \times H^\prime /H \arrow[rr, "pr_G"] \arrow[d, "m"] && G \arrow[d, "q^\prime"]\\
G/H \arrow[rr, "f"] && G/H^\prime
\end{tikzcd}
\]
Since $q^\prime$ is faithfully flat and $pr_G$ is obtained as base change of $f$ via such a morphism, $f$ being proper is equivalent to $pr_G$ being proper; now, the latter is obtained as base change of $H^\prime/H \rightarrow \Spec k$ via the structural morphism of $G$, which is also fppf, hence it is proper if and only if $H^\prime/H$ is proper over $k$. This shows the first condition. \\
Moreover, the formation of the direct image of sheaves also commutes with fppf extensions: more precisely, applying this to the structural sheaves in our case yields
\[
(q^\prime)^\ast f_\ast \mathcal{O}_{G/H} = (pr_G)_\ast \mathcal{O}_{G \times H^\prime /H} = \mathcal{O}_G \otimes \mathcal{O}_{H^\prime/H} (H^\prime /H),
\]
hence by taking $q_\ast^\prime$ on both sides one gets
\[
f_\ast \mathcal{O}_{G/H} = \mathcal{O}_{G/H^\prime} \quad \Longleftrightarrow \quad \mathcal{O}_{H^\prime/H} (H^\prime /H) = k,
\]
which gives the second condition.
\end{proof}

\begin{remark}
\label{ideas}
Let us consider again a fixed parabolic subgroup $P$. We now construct a collection of morphisms $f_\alpha \colon X \rightarrow G/Q_\alpha$, for $\alpha \in \Delta \backslash I$, such that
\begin{enumerate}[(1)]
    \item the target $G/Q_\alpha$ is defined in a concrete geometrical way,
    \item each $f_\alpha$ is a contraction,
    \item the stabilizer $Q_\alpha$ coincides with the smallest subgroup scheme of $G$ containing both $P$ and $P^\alpha$: in particular, $(Q_\alpha)_{\text{red}}$ is a maximal reduced parabolic subgroup,
    \item the collection $(f_\alpha)_{\alpha \in \Delta \backslash I}$ "tells us a lot" about the variety $X$.
\end{enumerate}
\end{remark}

The reason why $Q_\alpha$ is not directly defined as being the algebraic subgroup generated by $P$ and $P^\alpha$ is that this notion does not behave well since $P$ is nonreduced in general.\\
Let us apply \Cref{ProjY} to the variety $X=G/P$ and the line bundle $\mathcal{L}= \mathcal{O}_X(D_\alpha)$, which can be done thanks to \Cref{BB_flag}. This gives a contraction
\begin{align}
\label{definition_falpha}
f_\alpha \colon X \longrightarrow Y_\alpha \defeq  \Proj \bigoplus_{n=0}^\infty H^0(X,\mathcal{O}_X(nD_\alpha)).
\end{align}
By \Cref{blanchard}, there is a unique $G$-action on $Y_\alpha$ such that $f_\alpha$ is equivariant. Moreover, since $f_\alpha$ is a dominant morphism between projective varieties, it is surjective, hence the target must be of the form $Y_\alpha = G/Q_\alpha$ for some subgroup scheme $P \subseteq Q_\alpha \subsetneq G$. We take this construction as the definition of the subgroup $Q_\alpha$, so that conditions $(1)$ and $(2)$ are already satisfied. Moreover, by \Cref{BB_flag} and \Cref{ProjY}, a curve $C$ is contracted by $f_\alpha$ if and only if $D_\alpha \cdot C=0$, meaning that this map contracts all $C_\beta$ for $\beta \neq \alpha$ while it restricts to a finite morphism on $C_\alpha$. This leaves one more condition to show.

\begin{lemma}
The smallest subgroup scheme of $G$ containing both $P$ and $P^\alpha$ is $Q_\alpha$.
\end{lemma}

\begin{proof}
By definition of $Y_\alpha$ we have the inclusion $P \subset Q_\alpha$.\\
Let $H$ be the subgroup scheme of $G$ generated by $P$ and $P^\alpha$. Since
\[P_{\text{red}} = P_I = \bigcap_{\alpha \in \Delta \backslash I} P^\alpha,
\]
the subgroup generated by $P_{\text{red}}$ and $P^\alpha$ is just $P^\alpha$. Next, consider the quotient map $\widetilde{\pi} \colon \widetilde{X} \rightarrow G/P^\alpha$ and the composition $f_\alpha \circ \sigma \colon \widetilde{X} \rightarrow G/Q_\alpha$: the latter contracts, by the above discussions, all curves $\widetilde{C}_\beta$ for $\beta \neq \alpha$, hence $\NE(\widetilde{\pi})\subset \NE(f_\alpha \circ \sigma)$. Moreover, $\widetilde{\pi}$ is a contraction by \Cref{criterion_contraction}, because its fiber at the base point is $P^\alpha/P_I$ which is proper and has no nonconstant global regular functions. By \Cref{contractions}, there exists a unique morphism $\varphi$ making the diagram
\[
\begin{tikzcd}
\widetilde{X} = G/P_{\text{red}} \arrow[d, "\sigma"] \arrow[rr, "\widetilde{\pi}"] && G/P^\alpha \arrow[d, "\varphi", dotted]\\
X=G/P \arrow[rr, "f_\alpha"] && G/Q_\alpha
\end{tikzcd}
\]
commute: this shows $P^\alpha \subset Q_\alpha$ hence $H \subset Q_\alpha$.\\
Conversely, let us consider the projection $\pi \colon X \rightarrow G/H$. We already know by \Cref{BB_flag} that $\widetilde{\pi}$ contracts all $\widetilde{C}_\beta$ for $\beta \neq \alpha$; moreover, the square on the left in the following diagram is commutative and its horizontal arrows are both homeomorphisms. This implies that $\pi$ contracts all $C_\beta$ for $\beta \neq \alpha$. In other words, the inclusion $\NE(f_\alpha) \subset \NE(\pi)$ holds.
\[
\begin{tikzcd}
\widetilde{X} \arrow[r, "\sigma"] \arrow[d, "\widetilde{\pi}"] & X \arrow[rr, "f_\alpha"] \arrow[d, "\pi"] && G/Q_\alpha \arrow[dll, "\psi", dotted] \\
G/P^\alpha \arrow[r] & G/H &
\end{tikzcd}
\]
 Since $f_\alpha$ is a contraction by definition, this gives a factorisation by $\psi$ - again by \Cref{contractions} - which means that $Q_\alpha \subset H$.
\end{proof}

\begin{remark}
\label{j_closedimm}
The homogeneous space $X$ is now equipped with a finite number of contractions $f_\alpha$ such that the target of each morphism has Picard group $\Z$, with a unique canonical ample generator, corresponding to the image of $D_\alpha$. The inclusion
\begin{align}
\label{inclusion}
P \subseteq \bigcap_{\alpha \in \Delta} Q_\alpha
\end{align}
holds by definition of $Q_\alpha$. If the characteristic is $p \geq 5$, by \cite{Wenzel} there are nonnegative integers $m_\alpha$ for $\alpha \in \Delta \backslash I$ such that $P$ is the intersection of the $G_{m_\alpha}P^\alpha$, hence $P \subset Q_\alpha \subset G_{m_\alpha}P^\alpha$ and the inclusion (\ref{inclusion}) becomes an equality. Geometrically, this corresponds to saying that the product map 
\[
f \defeq \prod_{\alpha \in \Delta} f_\alpha \colon X \longrightarrow \prod_{\alpha \in \Delta} G/Q_\alpha
\]
is a closed immersion, realizing $X$ as the unique closed orbit of the $G$-action on the target.
\end{remark}

\subsection{Examples in Picard rank two}

Let us consider a simple simply connected algebraic group $G$ over $k$, having Dynkin diagram with an edge of multiplicity equal to the characteristic $p \in \{ 2,3 \}$, so that the definitions and properties of Subsection \ref{subsection_N} apply. In what follows, we call a parabolic subgroup \emph{of standard type} if it is of the form $G_{m_1}P^{\alpha_1} \cap \ldots \cap G_{m_r} P^{\alpha_r}$ for some integers $m_i$ and simple roots $\alpha_i$, while a homogeneous space is said to be \emph{of standard type} its underlying variety is isomorphic to some $G^\prime/P^\prime$, where $P^\prime$ is a parabolic subgroup of standard type.\\
The main result in this part is the following, which provides us with a first family of homogeneous projective varieties (in types $B_n$, $C_n$ and $F_4$) which are not of standard type.

\begin{proposition}
\label{main:rank2}
    Let $p=2$ and consider a simple, simply connected group $G$ and two distinct simple roots $\alpha$ and $\beta$ such that: either $G$ is of type $B_n$ or $C_n$ and the pair $(\alpha,\beta)$ is of the form $(\alpha_j,\alpha_i)$ with $i < j < n$ or $j=n$ and $i <n-1$, or $G$ is of type $F_4$ and the pair $(\alpha,\beta)$ is one among
    \[
    (\alpha_1,\alpha_4), \quad (\alpha_2,\alpha_1), \quad (\alpha_2,\alpha_4), \quad (\alpha_3,\alpha_1), \quad (\alpha_3,\alpha_4), \quad (\alpha_4,\alpha_1).
    \]
    Then the homogeneous space $X= G/(N_{r,G}P^\alpha \cap P^\beta)$
    is \emph{not} of standard type.
\end{proposition}

First, we give a motivation to the fact that we look for an example in rank two, then we prove \Cref{main:rank2} in two consecutive steps.\\

Let us fix a simple root $\alpha \in \Delta$. In order to find a parabolic subgroup not of standard type, the easiest and more natural idea is to consider the very special isogeny $\pi_G \colon G \longrightarrow \overline{G}$ and the subgroup $P \defeq N_G P^\alpha$. Its reduced part $P_{\text{red}} = P^\alpha$ is maximal, but $P$ is not of the form $G_m P^\alpha$ for any $m$. Indeed, its associated function $\varphi_P \colon \Phi^+ \rightarrow \N \cup \{ \infty \}$ is given by
\begin{align*}
    \gamma \longmapsto \infty & \quad \text{ if } \alpha \notin \Supp (\gamma)\\
    \gamma \longmapsto 0 & \quad \text{ if } \alpha \in \Supp (\gamma) \text{ and } \gamma \in \Phi_> \\
    \gamma \longmapsto 1 & \quad \text{ if } \alpha \in \Supp (\gamma) \text{ and } \gamma \in \Phi_<
\end{align*}
while the function associated to a parabolic subgroup of standard type satisfies $\varphi_{G_m P^\alpha}(\gamma) = m$ for all roots $\gamma$ containing $\alpha$ in their support, regardless of their length. There always exist both a short and a long root containing any simple root $\alpha$ in their support, namely
\begin{align}
    \label{bbb}
    & \bullet \, \text{in type } B_n, \, \Supp(\varepsilon_1) 
    = \Supp(\varepsilon_1+\varepsilon_2) = \Delta ;\\
    & \bullet \, \text{in type } C_n, \,\Supp(2\varepsilon_1)
     = \Supp(\varepsilon_1+\varepsilon_2)
     = \Delta ;\\
     \label{fff}
     & \bullet \, \text{in type } F_4, \, \Supp(\alpha_1+2\alpha_2+3\alpha_3+2\alpha_4)=\Supp(\alpha_1+2\alpha_2+4\alpha_3+2\alpha_4) = \Delta;\\
     & \bullet \, \text{in type } G_2, \, \Supp(2\alpha_1+\alpha_2)= \Supp(3\alpha_1+2\alpha_2)=\Delta.
\end{align}
Let us remark that the above roots can be constructed in a uniform way: they are respectively the highest short root and the highest (long) root. Thus, we can conclude that $\varphi_P \neq \varphi_{G_mQ}$ for all $m$, proving that $P$ is a parabolic subgroup not of standard type. However, $X= G/P$ is isomorphic as a variety to $\overline{G}/P_{\overline{\alpha}}$, hence the homogeneous space $X$ is still of standard type.\\
The same reasoning applies when one considers the product of a parabolic subgroup of standard type and of a kernel of a noncentral isogeny with source $G$: this might define a new parabolic subgroup, but an homogeneous space which is still of standard type. Together with \Cref{classification_rank1}, this implies that it is not possible to find examples of homogeneous spaces not of standard type having Picard rank one, when the characteristic satisfies the edge hypothesis (see \Cref{subsection_N}). This provides a motivation to the study of the rank two case, which means considering parabolic subgroups whose reduced part is of the form $P^\alpha \cap P^\beta$ for two distinct simple roots $\alpha$ and $\beta$. In such a context we are able to find the desired class of examples.

\begin{lemma}
\label{alphabeta}
Let us consider a simple, simply connected group $G$ having Dynkin diagram with an edge of multiplicity $p$, fix two distinct simple roots $\alpha$ and $\beta$ and an integer $r \geq 0$. Both the parabolic
\[
P \defeq N_{r,G} P^\alpha \cap P^\beta
\]
and its pull-back via the very special isogeny $\pi_{\overline{G}} \colon \overline{G} \rightarrow G$ are \emph{not} of standard type if and only if one of the following conditions is satisfied :
\begin{enumerate}[(i)]
    \item $G$ is of type $B_n$ or $C_n$ and the pair $(\alpha,\beta)$ is of the form $(\alpha_j,\alpha_i)$ with $i < j < n$ or $j=n$ and $i <n-1$ ;
    \item $G$ is of type $F_4$ and the pair $(\alpha,\beta)$ is one amongst
    \[
    (\alpha_1,\alpha_4), \quad (\alpha_2,\alpha_1), \quad (\alpha_2,\alpha_4), \quad (\alpha_3,\alpha_1), \quad (\alpha_3,\alpha_4), \quad (\alpha_4,\alpha_1).
    \]
\end{enumerate}
In particular, this situation can only happen when $p=2$.
\end{lemma}

\begin{proof}
Let us take a look at the function $\varphi_P \colon \Phi^+ \rightarrow \N \cup \{ \infty \}$ associated to the parabolic $P$ - recall that it is determined by the equality
\[
U_{-\gamma} \cap P = u_{-\gamma}(\alpha_{p^{\varphi(\gamma)}}), \quad \gamma \in \Phi^+
\]
- and let us compare it to the one associated to some $Q= G_m P^\alpha \cap G_nP^\beta$ (i.e. a parabolic of standard type), which is necessarily of this form because $Q_{\text{red}}= P_{\text{red}} = P^\alpha \cap P^\beta$. Our aim is to find in which cases there is a contradiction with the equality $P=Q$. First of all, assuming $\varphi_P(\beta) = \varphi_Q(\beta)$ leads to $n=0$. Now, let us write down the values that $\varphi_P$ and $\varphi_Q$ assume on all positive roots in the following table.
\begin{center}
\begin{tabular}{ |c|c|c|c|c|}
\hline
& $\alpha, \beta \in \Supp(\gamma)$ & \multirow{2}{9em}{$\alpha \in \Supp(\gamma)$, $\beta \notin \Supp(\gamma)$, $\gamma$ short} & \multirow{2}{9em}{$\alpha \in \Supp(\gamma)$, $\beta \notin \Supp(\gamma)$, $\gamma$ long} & $\beta \in \Supp(\gamma)$ \\
& & & & \\
\hline
$\varphi_Q(\gamma)$ & $\infty$ & $m$ & $m$ & $0$ \\
\hline
$\varphi_P(\gamma)$ & $\infty$ & $r+1$ & $r$ & $0$ \\
\hline
\end{tabular}
\end{center}
Thus, the two functions can never coincide if and only if there exist at least one long root and one short root containing $\alpha$ and not $\beta$ in their respective supports. Let us examine each root system to determine when this is the case.
\begin{itemize}
    \item If $G$ is of type $G_2$ in characteristic $p=3$, then all roots distinct from $\alpha_1$ and $\alpha_2$ contain both simple roots in their support, hence the desired condition is never satisfied. Thus from now on we can assume that $p=2$.
    \item If $G$ is of type $B_n$, let $\alpha= \alpha_j$ and $\beta = \alpha_i$ for some $1\leq i,j \leq n$. A positive short root is of the form $\varepsilon_m = \alpha_m + \ldots + \alpha_{n-1} +2\alpha_n$ for $m <n$ or $\varepsilon_n=\alpha_n$: hence if $j<i$ then a short root containing $\alpha$ in its support also contains $\beta$. Let us then assume $i<j$: in this case $\gamma = \varepsilon_j$ satisfies the condition. Moving on to long roots, if $i<j<n$ then $\gamma = \alpha = \varepsilon_j-\varepsilon_{j+1}$ is as wanted, while if $j=n$ then $\gamma = \varepsilon_{n-1} + \varepsilon_n = \alpha_n + 2\alpha_{n-1}$ satisfies the condition when $i<n-1$, while if $i = n-1$ then there is no such $\gamma$.
    \item If $G$ is of type $C_n$, let $\alpha= \alpha_j$ and $\beta = \alpha_i$ for some $1\leq i,j \leq n$. A positive long root is of the form $2\varepsilon_m = 2(\alpha_m + \ldots + \alpha_{n-1} +\alpha_n)$ for $m <n$ or $2\varepsilon_n=\alpha_n$: hence if $j<i$ then a long root containing $\alpha$ in its support also contains $\beta$. Let us then assume $i<j$: in this case $\gamma = 2\varepsilon_j$ satisfies the condition. Moving on to short roots, if $i<j<n$ then $\gamma = \alpha = \varepsilon_j-\varepsilon_{j+1}$ is as wanted, while if $j=n$ then $\gamma = \varepsilon_{n-1}+\varepsilon_n = \alpha_n + \alpha_{n-1}$ satisfies the condition when $i<n-1$, while if $i = n-1$ then there is no such $\gamma$. This completes condition $(i)$.
    \item If $G$ is of type $F_4$, there is no short root containing $\alpha_1$ (resp. $\alpha_1$, resp. $\alpha_2$) in its support and not containing $\alpha_2$ (resp. $\alpha_3$, resp. $\alpha_3$); moreover, there is no long root containing $\alpha_3$ (resp. $\alpha_4$, resp. $\alpha_4$) in its support and not containing $\alpha_2$ (resp. $\alpha_2$, resp. $\alpha_3$). This can be seen by directly looking at the list of positive roots in such a system, recalled at the beginning of Subsection \ref{F4}. The remaining pairs are listed below, which gives condition $(ii)$.
\end{itemize}
\begin{center}
\begin{tabular}{ |c|c|c|c|}
\hline
$\alpha$ & $\beta$ & a short $\gamma \colon \alpha \in \Supp(\gamma), \beta \notin \Supp(\gamma)$ & a long $\gamma \colon \alpha \in \Supp(\gamma), \beta \notin \Supp(\gamma)$ \\
\hline
$\alpha_1$ & $\alpha_4$ & $\alpha_1+\alpha_2+\alpha_3$ & $\alpha_1$ \\
\hline
$\alpha_2$ & $\alpha_1$ & $\alpha_2+\alpha_3$ & $\alpha_2$ \\
\hline
$\alpha_2$ & $\alpha_4$ & $\alpha_2+\alpha_3$ & $\alpha_2$ \\
\hline
$\alpha_3$ & $\alpha_1$ & $\alpha_3$ & $\alpha_2+2\alpha_3$\\
\hline
$\alpha_3$ & $\alpha_4$ & $\alpha_3$ & $\alpha_2+2\alpha_3$ \\
\hline
$\alpha_4$ & $\alpha_1$ & $\alpha_4$ & $\alpha_2+2\alpha_3+2\alpha_4$ \\
\hline
\end{tabular}
\end{center}
Up to this point we have only shown that the parabolic $P$ is not of standard type if and only if conditions $(i)$ or $(ii)$ are satisfied. Now, let us consider the pull-back 
\[
\pi_{\overline{G}}^{-1}(P) = \pi_{\overline{G}}^{-1} (N_{r,G}P^\alpha \cap P^\beta) = \overline{G}_{r+1}P^{\overline{\alpha}} \cap N_{\overline{G}}P^{\overline{\beta}}
\]
and compare it with $Q = \overline{G}_m P^{\overline{\alpha}} \cap \overline{G}_n P^{\overline{\beta}}$, analogously as before. This gives in particular, considering a root $\gamma \in \Phi^+$ satisfying $\overline{\alpha}, \overline{\beta} \in \Supp(\gamma)$, that $\varphi_Q(\gamma) = \min (m,n)$ for all $\gamma$, while $\varphi_{\pi^{-1}(P)}(\gamma)$ is equal to $1$ if $\gamma$ is short, and equal to $0$ if $\gamma$ is long. To show that those two parabolics can never coincide it is enough to have both such a long and a short root. This is always the case, as recalled at the beginning of this Subsection in (\ref{bbb})- (\ref{fff}), hence this concludes the proof.
\end{proof}


\begin{lemma}
\label{rank2}
Keeping the above notations, consider two distinct simple positive roots $\alpha$ and $\beta$ satisfying one of the conditions of \Cref{alphabeta}.
Then the parabolic $P \defeq N_{r,G}P^\alpha \cap P^\beta$ gives a variety $X\defeq G/P$ which is \emph{not} of standard type.
\end{lemma}


\begin{proof} 
The reduced part of the parabolic subgroup $P$ is $P_{\text{red}} = P^\alpha \cap P^\beta$: by \Cref{BB_flag}, the convex cone of curves of the variety $X$ is generated by the classes of the curves
\[
C_\alpha = \overline{Bs_\alpha o} \quad \text{and} \quad C_\beta = \overline{Bs_\beta o}.
\]
Next, let us consider the two contractions 
\[
f_\alpha \colon X \longrightarrow G/Q_\alpha \quad \text{and} \quad f_\beta \colon X \longrightarrow G/Q_\beta
\]
defined by (\ref{definition_falpha}). Clearly, $Q_\beta = \langle Q,P^\beta \rangle = P^\beta $ is smooth because $P \subset P^\beta$.
On the other hand, let us show that $Q_\alpha = N_{r,G}P^\alpha$. Since both $P$ and $P^\alpha$ are subgroups of the right hand term, the inclusion $Q_\alpha \subset N_{r,G}P^\alpha$ holds. To prove the other inclusion, let us notice that the hypothesis on $\alpha$ and $\beta$, as shown in the proof of \Cref{alphabeta}, guarantees the existence of some short positive root $\gamma$ containing $\alpha$ and not $\beta$ in its support. In particular, this implies that
\[
P \cap U_{-\gamma} = (N_{r,G}P^\alpha \cap U_{-\gamma}) \cap (P^\beta \cap U_{-\gamma}) = u_{-\gamma}(\alpha_{p^{r+1}}),
\]
hence $Q_\alpha \cap U_{-\gamma}$ is the image of a Frobenius kernel of height at least equal to $r+1$. By the factorisation of isogenies in \Cref{factorisation_isogenies}, the only two possibilities are thus $Q_\alpha= G_{r+1}P^\alpha$ and $Q_\alpha= N_{r,G}P^\alpha$, which allows to conclude that $Q_\alpha= N_{r,G}P^\alpha$. 
This means that the product of the contractions 
\[
\begin{tikzcd}
f = f_\alpha \times f_\beta \colon X \arrow[rr, hookrightarrow] && X_\alpha \times X_\beta
\end{tikzcd}
\]
is a closed immersion, where $X_\alpha$ (resp. $X_\beta$) is the underlying variety of $G/N_GP^\alpha$ (resp. $G/P^\beta$). Moreover, these maps are - up to a permutation - uniquely determined by the variety $X$, because the monoid $\N C_\alpha \oplus \N C_\beta \subset N_1(X)$ of effective $1$-cycles does not depend on the group action on it: the two contractions are uniquely determined by its two generators and by the fact that the first is a nonsmooth morphism while the second is smooth.\\
The following step consists in studying the automorphisms of the varieties $X$ and $X_\beta$. First, we can apply \Cref{demazure77} to the variety $X_\beta = G/P^\beta$ since its stabilizer is smooth and since by \Cref{alphabeta} the pair $(G_{\ad},P^\beta/Z(G))$ is not associated to any of the exceptional pairs, except in the case of $G= \Sp_{2n}$ and $(\alpha,\beta) = (\alpha_j, \alpha_1)$, which we treat later.
This implies
\[
\underline{\Aut}_{X_\beta}^0 = G_{\ad}.
\]
Next, let us consider the group $\underline{\Aut}_X^0$: its natural action on $X$ gives, applying \Cref{blanchard} to the contraction $f_\beta \colon X \longrightarrow X_\beta$, an action on $X_\beta$ i.e. a morphism
\[
\begin{tikzcd}
\underline{\Aut}^0_X \arrow[rr, "\xi"] && \underline{\Aut}_{X_\beta}^0 = G_{\ad}.
\end{tikzcd}
\]
In particular, the isogeny $\xi$ is a section of the natural morphism given by the action of $G_{\ad}$ on $X$, thus giving a semidirect product $\underline{\Aut}_X^0 = G_{\text{ad}} \rtimes \ker\xi$. Since $\underline{\Aut}_X^0$ is reduced, $\ker \xi$ must be finite, smooth and connected, so it is trivial and we conclude that $\underline{\Aut}_X^0 = G_{\ad}.$\\
Finally, let us consider another action of a semisimple, simply connected $G^\prime$ onto the variety $X$; realizing it as a quotient $G^\prime / P^\prime$ for some parabolic subgroup $P^\prime$. Since it is simply connected, $G^\prime$ is either simple or the direct product $G_{(1)} \times \cdots \times G_{(l)}$ where each $G_{(i)}$ is simple.\\
$\bullet$ If $G^\prime$ is simple, then its action on $X$ induces a morphism $G^\prime \longrightarrow \underline{\Aut}_X^0=G_{\text{ad}}$, which is in particular an isogeny. By \Cref{factorisation_isogenies}, this morphism can be factorised as
\[
\begin{tikzcd}
G^\prime \arrow[rr, "F^m"] && G \arrow[r, twoheadrightarrow]  & \underline{\Aut}_X^0 & \text{or} & G^\prime \arrow[rr, "F^m \circ \pi"] && G \arrow[r, twoheadrightarrow]  &  \underline{\Aut}_X^0,
\end{tikzcd}
\]
where the second possibility only can happen whenever $G$ satisfies the edge hypothesis. The stabilizer of the $G^\prime$-action is the preimage of the stabilizer of the $G$-action via such an isogeny, hence it is either of the form $G_m P$ for some $m$ or of the form $\overline{G}_m \pi^{-1}(P)$. Now, a parabolic $Q$ is of standard type if and only if $G_m Q $ is for any integer $m$, since the associated functions satisfy $\varphi_Q(\gamma) + m = \varphi_{G_m Q}(\gamma)$. This means that $P^\prime$ is of standard type if and only if $P$ (resp. $\pi^{-1}(P)$) is. This remark, together with \Cref{alphabeta} allows us to conclude that, due to our choice of roots $\alpha$ and $\beta$, $P^\prime$ is still a parabolic subgroup not of standard type.\\
If $G=\Sp_{2n}$ and $P^\beta= P^{\alpha_1}$, then \Cref{demazure77} yields $\underline{\Aut}_{X_{\beta}}^0= \PGL_{2n}$. Repeating the above reasoning implies that $\underline{\Aut}_X^0= \PGL_{2n}$ as well, hence the isogeny with source $G^\prime$ is necessarily the composition of an iterated Frobenius and a central isogeny. This implies that the stabilizer of the $G^\prime$-action is of the form $P^\prime = G_mP$ hence still not of standard type.\\
$\bullet$ If $G^\prime = G_{(1)} \times \cdots \times G_{(l)}$ is not simple, consider the morphism
\[
\begin{tikzcd}
G_{(1)} \times \cdots \times G_{(l)} \arrow[rr, "\phi"] && G \arrow[r, twoheadrightarrow] & G_{\text{ad}}
\end{tikzcd}
\]
determined by the action: then $H\defeq \ker \phi$ is a normal subgroup of $G^\prime$ and the image of $\phi$ is simple, thus $H$ is necessarily of the form
\[
H = \prod_{i \neq i_0} G_{(i)} \times K, \quad \text{for some } K \subset Z(G_{(i_0)}),
\]
thus $K$ is trivial because the quotient $G$ is also simply connected.
In particular, denoting as $P_{(i_0)} \defeq P^\prime \cap G_{(i_0)}$, we have 
\[
X = G^\prime/P^\prime = G^\prime / \left( \prod_{i \neq i_0} G_{(i)} \times P_{(i_0)}\right) = G_{(i_0)} / P_{(i_0)}
\]
Applying the reasoning above to $G_{(i_0)}$ instead of $G^\prime$ leads to the conclusion that the associated function of $P_{(i_0)}$ is not of standard type, hence the same is true for the stabilizer $P^\prime = \prod_{i \neq i_0} G_{(i)} \times P_{(i_0)}$.
\end{proof}

Notice that, except for the group of type $G_2$ in characteristic $2$, \Cref{rank2} covers the classification of all homogeneous spaces of Picard rank two that "we know the existence of" i.e. those of the form $G/P$ with $P = (\ker \varphi)P^\alpha \cap (\ker \psi) P^\beta$ for a couple of isogenies $\varphi$ and $\psi$ with source $G$.\\
Indeed, \Cref{factorisation_isogenies} implies that one of the two kernels must be contained in the other, hence up to permuting $\alpha$ and $\beta$ the inclusion $\ker\psi \subset \ker \varphi$ holds. Taking the quotient by $\ker \psi$ allows to assume either $P = G_r P^\alpha \cap P^\beta$, which is the standard type case, or $P = N_{r,G}P^\alpha \cap P^\beta$ for some $r \geq 0$. The latter gives a variety not of standard type if and only if $p=2$ and the above hypothesis on roots is satisfied.\\


\textbf{Problem}: Let us consider a simple group and a parabolic subgroup $P \subset G$ with reduced part $P_{\text{red}}=P_I$ which is not maximal. The associated family of contractions give a natural inclusion
\[
P \subset \bigcap_{\alpha \in \Delta \backslash I} Q_\beta.
\]
The question whether there exist a parabolic subgroup $P$ for which the inclusion is strict is still open. At this point, we are neither able to exclude their existence nor to exhibit an explicit example.

\vfill\pagebreak

\section{Appendix}
\label{sec4}

Let us resume here a short description of the Chevalley embedding of the group of type $G_2$, which holds in any characteristic. We will then specialize to characteristic two which is the interesting one for our purposes. First we describe its action on the algebra of octonions, then we use it to compute some of the root subgroups of such a group, which are fundamental in order to study the parabolic subgroups $P_{\mathfrak{h}}$ and $P_{\mathfrak{l}}$ (see \Cref{def:PhPl}).

\subsection{The Chevalley embedding of $G_2$}
\label{G2_description}

Let $G$ be the simple group of type $G_2$ in characteristic $p >0$. It can be viewed - as illustrated in \cite{SV}, from which we will keep most of the notation - as the automorphism group of an octonion algebra. The latter is the algebra
\[
\mathbb{O} = \left\{ (u,v) \colon u,v \text{ are } 2\times 2 \text{ matrices}\right\},
\]
with basis
\begin{align*}
    & e_{11}=  \left(\begin{pmatrix} 1 & 0 \\ 0 & 0 \end{pmatrix}, \begin{pmatrix} 0 & 0 \\ 0 & 0 \end{pmatrix} \right), & e_{12}=  \left(\begin{pmatrix} 0 & 1 \\ 0 & 0 \end{pmatrix}, \begin{pmatrix} 0 & 0 \\ 0 & 0 \end{pmatrix} \right),\\
    & e_{21}=  \left(\begin{pmatrix} 0 & 0 \\ 1 & 0 \end{pmatrix}, \begin{pmatrix} 0 & 0 \\ 0 &  0 \end{pmatrix} \right), & e_{22}=  \left(\begin{pmatrix} 0 & 0 \\ 0 & 1 \end{pmatrix}, \begin{pmatrix} 0 & 0 \\ 0 & 0 \end{pmatrix} \right),\\
    & f_{11}=  \left(\begin{pmatrix} 0 & 0 \\ 0 & 0 \end{pmatrix}, \begin{pmatrix} 1 & 0 \\ 0 & 0 \end{pmatrix} \right), & f_{12}=  \left(\begin{pmatrix} 0 & 0 \\ 0 & 0 \end{pmatrix}, \begin{pmatrix} 0 & 1 \\ 0 & 0 \end{pmatrix} \right),\\
    & f_{21}=  \left(\begin{pmatrix} 0 & 0 \\ 0 & 0 \end{pmatrix}, \begin{pmatrix} 0 & 0 \\ 1 & 0 \end{pmatrix} \right), & f_{22}=  \left(\begin{pmatrix} 0 & 0 \\ 0 & 0 \end{pmatrix}, \begin{pmatrix} 0 & 0 \\ 0 & 1 \end{pmatrix} \right),
\end{align*}
unit $e=(1,0)= e_{11}+e_{22}$, and which is equipped with a norm
\[
q(u,v) = \det (u) - \det(v).
\]
Let us write here for reference a table of products of the basis vectors :
\begin{align}
\label{table:octonions}
\begin{tabular}{ |c|c|c|c|c|c|c|c|c|}
\hline 
$\diagdown$ & $e_{11}$ & $e_{21}$ & $e_{12}$ & $e_{22}$ & $f_{11}$ & $f_{21}$ & $f_{12}$ & $f_{22}$\\
\hline
$e_{11}$ & $e_{11}$ & $0$ & $e_{12}$ & $0$ & $f_{11}$ & $f_{21}$ & $0$ & $0$\\
\hline
$e_{21}$ & $e_{21}$ & $0$ & $e_{22}$ & $0$ & $0$ & $0$ & $f_{11}$ & $f_{21}$\\
\hline
$e_{12}$ & $0$ & $e_{11}$ & $0$ & $e_{12}$ & $f_{12}$ & $f_{22}$ & $0$ & $0$\\
\hline
$e_{22}$ & $0$ & $e_{21}$ & $0$ & $e_{22}$ & $0$ & $0$ & $f_{12}$ & $f_{22}$\\
\hline
$f_{11}$ & $0$ & $0$ & $-f_{12}$ & $f_{11}$ & $0$ & $-e_{21}$ & $0$ & $e_{11}$\\
\hline
$f_{21}$ & $0$ & $0$ & $-f_{22}$ & $f_{21}$ & $e_{21}$ & $0$ & $-e_{11}$ & $0$\\
\hline
$f_{12}$ & $f_{12}$ & $-f_{11}$ & $0$ & $0$ & $0$ & $-e_{22}$ & $0$ & $e_{12}$\\
\hline
$f_{22}$ & $f_{22}$ & $-f_{21}$ & $0$ & $0$ & $e_{22}$ & $0$ & $-e_{12}$ & $0$\\
\hline
\end{tabular}
 \end{align}
An embedding of the group $G_2$ into $\SO_7$ - which gives an irreducible representation in all characteristics but two - can be seen as follows: let us consider its action on the vector space 
\begin{align}
\label{e_orthogonal}
V \defeq e^\perp = \{ (u,v) \colon \det(1+u) - \det(u) = 1\} = \{ (u,v) \colon u_{11} + u_{22} = 0 \}.
\end{align}
By \cite[Lemma $2.3.1$]{SV}, a maximal torus of $G$ - with respect to the basis $(e_{12}, e_{21},f_{11}, e_{11}-e_{22}, f_{12}, f_{21}, f_{22})$ of $W$ - acts on $V$ as 
\[
{\Gm}^2 \ni (\xi, \eta) \longmapsto \diag(\xi \eta, \xi^{-1}\eta^{-1}, \eta^{-1}, 1, \xi, \xi^{-1}, \eta) \in \GL_7
\]
Let us re-parameterize it with $\xi= a$, $\eta=ab$, this gives the torus 
\[
{\Gm}^2 \ni (a,b) \longmapsto \diag(a^2b,a^{-2}b^{-1},a^{-1}b^{-1},1,a,a^{-1},ab) =: t \in \GL_7,
\]
and the basis of simple roots we fix is $\alpha_1(t) \defeq a$ and $\alpha_2(t) \defeq b$. Such a torus acts on $V$ with the following weight spaces :
\begin{align*}
V_0 = k(e_{11}-e_{22}), \, & V_{\alpha_1} = kf_{12}, \, V_{-\alpha_1}= kf_{21}, \, V_{\alpha_1+\alpha_2} = kf_{22},\\
& V_{-\alpha_1-\alpha_2}= kf_{11}, \, V_{2\alpha_1+\alpha_2} = ke_{12}, \, V_{-2\alpha_1-\alpha_2} = ke_{21},
\end{align*}
which correspond to $0$ and the short roots. Re-arranging $V$ as 
\begin{align}
\label{V_roots}
V = kf_{12} \oplus kf_{11} \oplus ke_{12} \oplus k(e_{11}-e_{22}) \oplus ke_{21} \oplus kf_{22} \oplus kf_{21}
\end{align}
gives the maximal torus $T$
\begin{align}
\label{T_typeG}
{\Gm}^2 \ni (a,b) \longmapsto \diag(a,a^{-1}b^{-1},a^2b, 1,a^{-2}b^{-1}, ab,a^{-1}) = t \in T \subset \GL_7.
\end{align}
This way, $T$ can be identified with the maximal torus in \cite[page 13]{Heinloth}: in his description of the embedding $G \subset \GL_7$, the group $G$ is generated by the two following copies of $\GL_2$,
\[
\theta_1 \colon A \longmapsto \begin{pmatrix} A && \\ & \Sym^2(A) \det A^{-1} & \\ && A\end{pmatrix} \quad \! \text{and} \! \quad \theta_2 \colon B \longmapsto \begin{pmatrix} \det B^{-1} &&&& \\ & \widetilde{B} &&& \\ && 1 && \\ &&& B & \\ &&&& \det B\end{pmatrix},
\]
where \[
\widetilde{A} = \begin{pmatrix} 0 & 1 \\ 1 & 0 \end{pmatrix} \, ^tA^{-1} \begin{pmatrix} 0 & 1 \\ 1 & 0 \end{pmatrix}.\]
However, in characteristic $p=2$, due to the fact that $e \in V$ and that $G$ acts on the quotient $W=V/ke$, these become the following two copies embedded in $\GL(W) = \GL_6$:
\[
\theta_1 \colon A \longmapsto \begin{pmatrix} A && \\ & A^{(1)}\det A^{-1} & \\ && A\end{pmatrix} \quad \text{and} \quad \theta_2 \colon B \longmapsto \begin{pmatrix} \det B^{-1} &&& \\ & B && \\ && B & \\ &&& \det B\end{pmatrix},
\]
where $A^{(1)}$ denotes the Frobenius twist applied to $A$. 

\begin{lemma}
    The subgroups $\theta_1(\GL_2)$ and $\theta_2(\GL_2)$ have root system with positive root respectively $\beta_1 \defeq 2\alpha_1+\alpha_2$ and $\beta_2 \defeq -3\alpha_1-2\alpha_2$.
\end{lemma}

\begin{proof}
    See the computation of the root homomorphisms associated respectively to $\beta_1$ and $\beta_2$, done in \Cref{rootsubspaces:G2}: these are respectively the intersection of $\theta_1(\GL_2)$ and $\theta_2(\GL_2)$ with the upper triangular matrices of $\GL_7$.
\end{proof}

Let us remark that $\{ \beta_1, \beta_2\}$ is indeed a basis for the root system of type $G_2$, with corresponding set of positive roots being 
\[
-3\alpha_1-2\alpha_2, \, \alpha_1-\alpha_2, \, -\alpha_2, \, \alpha_2, \, 3\alpha_1+\alpha_2, \, 2\alpha_1+\alpha_2
\]
and with Borel subgroup given by the intersection of $G$ with the upper triangular matrices in $\GL_7$.


\subsection{Root subgroups}

Let us move on to the explicit computation of some of the root subgroups in type $G_2$. As before, we will do everything considering the action on a $7$-dimensional vector space - the orthogonal of the identity element of $\mathbb{O}$ - so that the computations hold in any characteristic, then at the end we will summarize what we get in characteristic $2$.\\

Let us consider the group $G$ acting on the vector space $V$ arranged as in (\ref{V_roots}). Denoting as $x_0,\ldots,x_6$ the coordinates on $V$, the norm becomes 
\begin{align}
\label{norm:q}
q(x) = -x^2_3 -x_2x_4-x_1x_5+x_0x_6,
\end{align}
while the maximal torus $T$ given in (\ref{T_typeG}) acts on $V$ through this table of characters
\begin{align}
\label{table:G2}
\begin{tabular}{ |c|c|c|c|c|c|c|}
\hline
$1$ & $a^2b$ & $a^{-1}b^{-1}$ & $a$ & $a^3b$ & $b^{-1}$ & $a^2$ \\
\hline
$a^{-2}b^{-1}$ & $1$ & $a^{-3}b^{-2}$ & $a^{-1}b^{-1}$ & $a$ & $a^{-2}b^{-2}$ & $b^{-1}$ \\
\hline
$ab$ & $a^3b^2$ & $1$ & $a^2b$ & $a^4b^2$ & $a$ & $a^3b$ \\
\hline
$a^{-1}$ & $ab$ & $a^{-2}b^{-1}$ & $1$ & $a^2b$ & $a^{-1}b^{-1}$ & $a$ \\
\hline
$a^{-3}b^{-1}$ & $a^{-1}$ & $a^{-4}b^{-2}$ & $a^{-2}b^{-1}$ & $1$ & $a^{-3}b^{-2}$ & $a^{-1}b^{-1}$ \\
\hline
$b$ & $a^2b^2$ & $a^{-1}$ & $ab$ & $a^3b^2$ & $1$ & $a^2b$ \\
\hline
$a^{-2}$ & $b$ & $a^{-3}b^{-1}$ & $a^{-1}$ & $ab$ & $a^{-2}b^{-1}$ & $1$ \\
\hline
\end{tabular}
\end{align}
The idea is the following: we know that - for any root $\gamma \in \Phi$ - the root subgroup $U_\gamma \subset G$ is determined by being the unique subgroup of $\GL(V)$ (resp. $\GL(W)$ in characteristic $2$), which is smooth, unipotent, is acted on by $T$ via the character $\gamma$, and whose elements are automorphisms of octonions. We will impose some of these necessary condition - such as $u_\gamma(\lambda)$ being an isometry for any $\lambda \in \Ga$ - to determine the root homomorphism $u_\gamma \colon \Ga \longrightarrow U_\gamma$.

$\bullet$ First, let us consider the root $\alpha_1$. By (\ref{table:G2}) and the condition for $u_{\alpha_1}$ to be a group homomorphism, there exist some constants $\eta_1,\ldots,\eta_5 \in k$ such that for any $\lambda \in \Ga$, $u_{\alpha_1}(\lambda)$ acts on $V$ as
\[
\begin{pmatrix}
    1 & 0 & 0 & \eta_1\lambda & 0 & 0 & \eta_5 \lambda^2\\
    0 & 1 & 0 & 0 & \eta_2\lambda & 0 & 0\\
    0 & 0 & 1 & 0 & 0 & \eta_3\lambda & 0\\
    0 & 0 & 0 & 1 & 0 & 0 & \eta_4\lambda\\
    0 & 0 & 0 & 0 & 1 & 0 & 0\\
    0 & 0 & 0 & 0 & 0 & 1 & 0\\
    0 & 0 & 0 & 0 & 0 & 0 & 1
\end{pmatrix}.
\]
Moreover, $u_{\alpha_1}(\lambda)$ being an isometry means, by (\ref{norm:q}), that
\begin{align*}
    q(x) & = q(u_{\alpha_1}(\lambda) \cdot x) = q(x_0+\eta_1\lambda x_3 + \eta_5\lambda^2 x_6, x_1+\eta_2\lambda x_4, x_2+\eta_3\lambda x_5, x_3+\eta_4\lambda x_6, x_4, x_5, x_6)\\
    & = q(x) + (-2\eta_4+\eta_1)\lambda x_3 x_6 + (\eta_5-\eta_4^2)\lambda^2 x_6^2 - (\eta_3+\eta_2)\lambda x_4x_5,
\end{align*}
hence $\eta_1 = 2\eta_4$, $\eta_5= \eta_4^2$ and $\eta_2= -\eta_3$. This still leaves two independent parameters $\eta_3$ and $\eta_4$ instead of one, so let us also impose the condition of $u_{\alpha_1}(\lambda)$ respecting the product $f_{22}e_{21} = -f_{21}$ - see (\ref{table:octonions}) :
\begin{align*}
    (u_{\alpha_1}(\lambda) \cdot f_{22}) (u_{\alpha_1}(\lambda)\cdot e_{21} ) = & \, u_{\alpha_1}(\lambda) \cdot (-f_{21})\\
    (\eta_3\lambda(e_{11}-e_{22})+f_{22}) (-\eta_3\lambda f_{11} +e_{21}) = & -\eta_4^2\lambda^2 f_{12}-\eta_4 \lambda(e_{11}-e_{22})-f_{21}\\
    -\eta_3^2\lambda^2 f_{12} -\eta_3 \lambda (e_{11}-e_{22}) -f_{21} = & -\eta_4^2\lambda^2 f_{12}-\eta_4 \lambda(e_{11}-e_{22})-f_{21},
\end{align*}
implying $\eta_3= \eta_4$. Let us reparametrise the root homomorphism such that $\eta_3=1$: this, together with an analogous computation for $-\alpha_1$, gives the desired representations, of the form
\[
u_{\alpha_1} \colon \lambda \mapsto
\begin{pmatrix}
    1 & 0 & 0 & 2\lambda & 0 & 0 & \lambda^2\\
    0 & 1 & 0 & 0 & -\lambda & 0 & 0\\
    0 & 0 & 1 & 0 & 0 & \lambda & 0\\
    0 & 0 & 0 & 1 & 0 & 0 & \lambda\\
    0 & 0 & 0 & 0 & 1 & 0 & 0\\
    0 & 0 & 0 & 0 & 0 & 1 & 0\\
    0 & 0 & 0 & 0 & 0 & 0 & 1
\end{pmatrix}, \quad 
u_{-\alpha_1} \colon \lambda \mapsto
\begin{pmatrix}
    1 & 0 & 0 & 0 & 0 & 0 & 0\\
    0 & 1 & 0 & 0 & 0 & 0 & 0\\
    0 & 0 & 1 & 0 & 0 & 0  & 0\\
    \lambda & 0 & 0 & 1 & 0 & 0 & 0\\
    0 & \lambda & 0 & 0 & 1 & 0 & 0\\
    0 & 0 & -\lambda & 0 & 0 & 1 & 0\\
    \lambda^2 & 0 & 0 & 2\lambda & 0 & 0 & 1
\end{pmatrix}.
\]

$\bullet$ Let us consider the root $\alpha_2$.  By (\ref{table:G2}) and the condition for $u_{\alpha_2}$ to be a group homomorphism, there exist some constants $\eta_1$ and $\eta_2 \in k$ such that for any $\lambda \in \Ga$, $u_{\alpha_2}(\lambda)$ acts on $V$ as
\[
u_{\alpha_2}(\lambda) \cdot x = (x_0, x_1, x_2, x_3, x_4, \eta_1\lambda x_0 x_5, \eta_2\lambda x_1+x_6).
\]
Moreover, the isometry condition means that
\begin{align*}
    q(x) & = q(u_{\alpha_2}(\lambda) \cdot x) = -x_3^2-x_2x_4 -\eta_1\lambda x_0x_1 -x_1x_5 +\eta_2\lambda x_0x_1 + x_0x_6\\ 
    & = q(x) +(\eta_2-\eta_1) \lambda x_0x_1,
\end{align*}
hence $\eta_1= \eta_2$. As before, we can conclude that the associated root subgroups are of the form 
\[
u_{\alpha_2} \colon \lambda \mapsto
\begin{pmatrix}
    1 & 0 & 0 & 0 & 0 & 0 & 0\\
    0 & 1 & 0 & 0 & 0 & 0 & 0\\
    0 & 0 & 1 & 0 & 0 & 0 & 0\\
    0 & 0 & 0 & 1 & 0 & 0 & 0\\
    0 & 0 & 0 & 0 & 1 & 0 & 0\\
    \lambda & 0 & 0 & 0 & 0 & 1 & 0\\
    0 & \lambda & 0 & 0 & 0 & 0 & 1
\end{pmatrix}, \quad 
u_{-\alpha_2} \colon \lambda \mapsto
\begin{pmatrix}
    1 & 0 & 0 & 0 & 0 & \lambda & 0\\
    0 & 1 & 0 & 0 & 0 & 0 & \lambda\\
    0 & 0 & 1 & 0 & 0 & 0  & 0\\
    0 & 0 & 0 & 1 & 0 & 0 & 0\\
    0 & 0 & 0 & 0 & 1 & 0 & 0\\
    0 & 0 & 0 & 0 & 0 & 1 & 0\\
    0 & 0 & 0 & 0 & 0 & 0 & 1
\end{pmatrix}.
\]

$\bullet$ Let us consider the root $2\alpha_1+\alpha_2$.  By (\ref{table:G2}) and the condition for $u_{2\alpha_1+\alpha_2}$ to be a group homomorphism, there exist some constants $\eta_1, \ldots, \eta_5 \in k$ such that for any $\lambda \in \Ga$, $u_{2\alpha_1+\alpha_2}(\lambda)$ acts on $V$ as
\[
\begin{pmatrix}
    1 & \eta_1\lambda & 0 & 0 & 0 & 0 & 0\\
    0 & 1 & 0 & 0 & 0 & 0 & 0\\
    0 & 0 & 1 & \eta_2\lambda & \eta_5\lambda^2 & 0 & 0\\
    0 & 0 & 0 & 1 & \eta_3\lambda & 0 & 0\\
    0 & 0 & 0 & 0 & 1 & 0 & 0\\
    0 & 0 & 0 & 0 & 0 & 1 & \eta_4\lambda\\
    0 & 0 & 0 & 0 & 0 & 0 & 1
\end{pmatrix}.
\]
Moreover, the isometry condition implies
\begin{align*}
    q(u_{2\alpha_1+\alpha_2}(\lambda)& \cdot x) = q(x_0+\eta_1\lambda x_1, x_1, x_2+\eta_2 \lambda x_3 + \eta_5 \lambda^2 x_4, x_3 + \eta_3 \lambda x_4, x_4, x_5+\eta_4 \lambda x_6, x_6)\\
    & = q(x) + (\eta_1-\eta_4)\lambda x_1 x_6 - (\eta_3^2+\eta_5)\lambda^2 x_4^2 - (2\eta_3-\eta_2)\lambda x_3x_4 = q(x),
\end{align*}
hence $\eta_1 = \eta_4$, $\eta_5 = -\eta_3^2$ and $\eta_2 = 2\eta_3$. This still leaves two independent parameters $\eta_3$ and $\eta_4$ instead of one, so let us also impose the condition of $u_{2\alpha_1+\alpha_2}(\lambda)$ respecting the product $f_{22}e_{21} = -f_{21}$ :
\begin{align*}
    (u_{2\alpha_1+\alpha_2}(\lambda) \cdot f_{22}) (u_{2\alpha_1+\alpha_2}(\lambda)\cdot e_{21} ) = & \, u_{2\alpha_1+\alpha_2}(\lambda) \cdot (-f_{21})\\
    f_{22} (-\eta_3^2 \lambda^2 e_{12} + \eta_3 \lambda (e_{11}-e_{22})+e_{21}) = & -\eta_4\lambda f_{22} +f_{21}\\
    \eta_3\lambda f_{22} -f_{21} = & -\eta_4\lambda f_{22} -f_{21},
\end{align*}
implying $\eta_3= - \eta_4$, so we can conclude that the associated root subgroups are of the form 
\begin{align*}
u_{2\alpha_1+\alpha_2} \colon \lambda \mapsto
\begin{pmatrix}
    1 & \lambda & 0 & 0 & 0 & 0 & 0\\
    0 & 1 & 0 & 0 & 0 & 0 & 0\\
    0 & 0 & 1 & 2\lambda & -\lambda^2 & 0 & 0\\
    0 & 0 & 0 & 1 & -\lambda & 0 & 0\\
    0 & 0 & 0 & 0 & 1 & 0 & 0\\
    0 & 0 & 0 & 0 & 0 & 1 & \lambda\\
    0 & 0 & 0 & 0 & 0 & 0 & 1
\end{pmatrix}, \, 
u_{-2\alpha_1-\alpha_2} \colon \lambda \mapsto
\begin{pmatrix}
    1 & 0 & 0 & 0 & 0 & 0 & 0\\
    \lambda & 1 & 0 & 0 & 0 & 0 & 0\\
    0 & 0 & 1 & 0 & 0 & 0  & 0\\
    0 & 0 & -\lambda & 1 & 0 & 0 & 0\\
    0 & 0 & -\lambda^2 & 2\lambda & 1 & 0 & 0\\
    0 & 0 & 0 & 0 & 0 & 1 & 0\\
    0 & 0 & 0 & 0 & 0 & \lambda & 1
\end{pmatrix}.
\end{align*}

$\bullet$ Let us consider the root $\alpha_1+\alpha_2$.  By (\ref{table:G2}) and the condition for $u_{\alpha_1+\alpha_2}$ to be a group homomorphism, there exist some constants $\eta_1, \ldots, \eta_5 \in k$ such that for any $\lambda \in \Ga$, $u_{\alpha_1+\alpha_2}(\lambda)$ acts on $V$ as
\[
\begin{pmatrix}
    1 & 0 & 0 & 0 & 0 & 0 & 0\\
    0 & 1 & 0 & 0 & 0 & 0 & 0\\
    \eta_1\lambda & 0 & 1 & 0 & 0 & 0  & 0\\
    0 & \eta_2\lambda & 0 & 1 & 0 & 0 & 0\\
    0 & 0 & 0 & 0 & 1 & 0 & 0\\
    0 & \eta_5\lambda^2 & 0 & \eta_3\lambda & 0 & 1 & 0\\
    0 & 0 & 0 & 0 & \eta_4\lambda & 0 & 1
\end{pmatrix}.
\]
Moreover, the isometry condition implies
\begin{align*}
    q(x) = & q(u_{\alpha_1+\alpha_2}(\lambda) \cdot x )= q(x_0,x_1,\eta_1\lambda x_0+x_2, \eta_2\lambda x_1+x_3, x_4, \eta_5\lambda^2 x_1 + \eta_3\lambda x_3 + x_5, \eta_4\lambda x_4+x_6)\\
    = & q(x) - (2\eta_2+\eta_3)\lambda x_1x_3 + (\eta_4-\eta_1)\lambda x_0x_4 - (\eta_2^2+\eta_5)\lambda^2 x_1^2,
\end{align*}
hence $\eta_3= -2\eta_2$, $\eta_1= \eta_4$ and $\eta_5 = -\eta_2^2$. Reasoning as in the above cases, let us also impose the condition of $u_{\alpha_1+\alpha_2}(\lambda)$ respecting the product $f_{11}f_{21} = -e_{21}$ :
\begin{align*}
    (u_{\alpha_1+\alpha_2}(\lambda) \cdot f_{11}) (u_{\alpha_1+\alpha_2}(\lambda) \cdot f_{21}) & =  u_{\alpha_1+\alpha_2}(\lambda) \cdot (-e_{21})\\
    (f_{11} +\eta_2\lambda(e_{11}-e_{22})-\eta_2^2\lambda^2f_{22}) f_{21} & = -e_{21} -\eta_4\lambda f_{21}\\
    -e_{21} +\eta_2\lambda f_{21} = -e_{21} -\eta_4\lambda f_{21},
\end{align*}
implying $\eta_2 = -\eta_4$. Reparametrizing and doing an analogous computation for the negative root allows to conclude that the root subgroups are as follows :
\[
u_{\alpha_1+\alpha_2} \colon \lambda \mapsto
\begin{pmatrix}
    1 & 0 & 0 & 0 & 0 & 0 & 0\\
    0 & 1 & 0 & 0 & 0 & 0 & 0\\
    \lambda & 0 & 1 & 0 & 0 & 0 & 0\\
    0 & -\lambda & 0 & 1 & 0 & 0 & 0\\
    0 & 0 & 0 & 0 & 1 & 0 & 0\\
    0 & -\lambda^2 & 0 & 2\lambda & 0 & 1 & 0\\
    0 & 0 & 0 & 0 & \lambda & 0 & 1
\end{pmatrix}, \quad 
u_{-\alpha_1-\alpha_2} \colon \lambda \mapsto
\begin{pmatrix}
    1 & 0 & \lambda & 0 & 0 & 0 & 0\\
    0 & 1 & 0 & 2\lambda & 0 & -\lambda^2 & 0\\
    0 & 0 & 1 & 0 & 0 & 0  & 0\\
    0 & 0 & 0 & 1 & 0 & -\lambda & 0\\
    0 & 0 & 0 & 0 & 1 & 0 & \lambda\\
    0 & 0 & 0 & 0 & 0 & 1 & 0\\
    0 & 0 & 0 & 0 & 0 & 0 & 1
\end{pmatrix}.
\]

$\bullet$ As last computation, let us consider the root $-3\alpha_1-2\alpha_2$. By (\ref{table:G2}) and the condition for $u_{-3\alpha_1-2\alpha_2}$ to be a group homomorphism, there exist some constants $\eta_1$ and $\eta_2 \in k$ such that for any $\lambda \in \Ga$, $u_{-3\alpha_1-2\alpha_2}(\lambda)$ acts on $V$ as
\[
u_{-3\alpha_1-2\alpha_2}(\lambda) \cdot x= (x_0,x_1+\eta_1\lambda x_2,x_2,x_3,x_4+\eta_2\lambda x_5, x_5, x_6).
\]
The isometry condition implies
\begin{align*}
    q(x) & =  q(u_{-3\alpha_1-2\alpha_2}(\lambda) \cdot x) = -x_3^2-x_2x_4-\eta_2\lambda x_2 x_5 -x_1x_5 -\eta_1\lambda x_2x_5 +x_0x_6\\
     & = q(x) + (\eta_2+\eta_1)\lambda x_2x_5,
\end{align*}
hence $\eta_2= -\eta_1$ and we can conclude that the root subgroups have the following form :
\begin{align}
\label{U_beta2}
u_{-3\alpha_1-2\alpha_2} \colon \lambda \mapsto
\begin{pmatrix}
    1 & 0 & 0 & 0 & 0 & 0 & 0\\
    0 & 1 & \lambda & 0 & 0 & 0 & 0\\
    0 & 0 & 1 & 0 & 0 & 0 & 0\\
    0 & 0 & 0 & 1 & 0 & 0 & 0\\
    0 & 0 & 0 & 0 & 1 & -\lambda & 0\\
    0 & 0 & 0 & 0 & 0 & 1 & 0\\
    0 & 0 & 0 & 0 & 0 & 0 & 1
\end{pmatrix}, \quad 
u_{3\alpha_1+2\alpha_2} \colon \lambda \mapsto
\begin{pmatrix}
    1 & 0 & 0 & 0 & 0 & 0 & 0\\
    0 & 1 & 0 & 0 & 0 & 0 & 0\\
    0 & -\lambda & 1 & 0 & 0 & 0  & 0\\
    0 & 0 & 0 & 1 & 0 & 0 & 0\\
    0 & 0 & 0 & 0 & 1 & 0 & 0\\
    0 & 0 & 0 & 0 & \lambda & 1 & 0\\
    0 & 0 & 0 & 0 & 0 & 0 & 1
\end{pmatrix}.
\end{align} 


\begin{remark}
\label{rootsubspaces:G2}
Let us recall that in characteristic $2$ the group $G$ acts on $W=V/ke$, giving an embedding $G \subset \Sp_6$: we list below what the root subspaces we need become in that case.
\begin{align*}
     u_{\alpha_1}(\lambda) & = 
    \begin{pmatrix}
        1 & 0 & 0 & 0 & 0 & \lambda^2\\
        0 & 1 & 0 & \lambda & 0 & 0\\
        0 & 0 & 1 & 0 & \lambda & 0\\
        0 & 0 & 0 & 1 & 0 & 0\\
        0 & 0 & 0 & 0 & 1 & 0\\
        0 & 0 & 0 & 0 & 0 & 1
    \end{pmatrix}, 
    & u_{-\alpha_1}(\lambda) = 
    \begin{pmatrix}
        1 & 0 & 0 & 0 & 0 & 0\\
        0 & 1 & 0 & 0 & 0 & 0\\
        0 & 0 & 1 & 0 & 0 & 0\\
        0 & \lambda & 0 & 1 & 0 & 0\\
        0 & 0 & \lambda & 0 & 1 & 0\\
        \lambda^2 & 0 & 0 & 0 & 0 & 1
    \end{pmatrix}\\
     u_{\alpha_2}(\lambda) & =
    \begin{pmatrix}
        1 & 0 & 0 & 0 & \lambda & 0\\
        0 & 1 & 0 & 0 & 0 & \lambda \\
        0 & 0 & 1 & 0 & 0 & 0\\
        0 & 0 & 0 & 1 & 0 & 0\\
        0 & 0 & 0 & 0 & 1 & 0\\
        0 & 0 & 0 & 0 & 0 & 1
    \end{pmatrix}, 
    & u_{-\alpha_2}(\lambda) =
    \begin{pmatrix}
        1 & 0 & 0 & 0 & 0 & 0\\
        0 & 1 & 0 & 0 & 0 & 0\\
        0 & 0 & 1 & 0 & 0 & 0\\
        0 & 0 & 0 & 1 & 0 & 0\\
        \lambda & 0 & 0 & 0 & 1 & 0\\
        0 & \lambda & 0 & 0 & 0 & 1
    \end{pmatrix}\\
    u_{2\alpha_1+\alpha_2}(\lambda) & = 
    \begin{pmatrix}
        1 & \lambda & 0 & 0 & 0 & 0\\
        0 & 1 & 0 & 0 & 0 & 0\\
        0 & 0 & 1 & \lambda^2 & 0 & 0\\
        0 & 0 & 0 & 1 & 0 & 0\\
        0 & 0 & 0 & 0 & 1 & \lambda\\
        0 & 0 & 0 & 0 & 0 & 1
    \end{pmatrix}, 
   & u_{-2\alpha_1-\alpha_2}(\lambda) = 
    \begin{pmatrix}
        1 & 0 & 0 & 0 & 0 & 0\\
        \lambda & 1 & 0 & 0 & 0 & 0\\
        0 & 0 & 1 & 0 & 0 & 0\\
        0 & 0 & \lambda^2 & 1 & 0 & 0\\
        0 & 0 & 0 & 0 & 1 & 0\\
        0 & 0 & 0 & 0 & \lambda & 1
    \end{pmatrix}\\
    u_{\alpha_1+\alpha_2}(\lambda) & = 
    \begin{pmatrix}
        1 & 0 & 0 & 0 & 0 & 0\\
        0 & 1 & 0 & 0 & 0 & 0\\
        \lambda & 0 & 1 & 0 & 0 & 0\\
        0 & 0 & 0 & 1 & 0 & 0\\
        0 & \lambda^2 & 0 & 0 & 1 & 0\\
        0 & 0 & 0 & \lambda & 0 & 1
    \end{pmatrix}, 
    & u_{-\alpha_1-\alpha_2}(\lambda) = 
    \begin{pmatrix}
        1 & 0 & \lambda & 0 & 0 & 0\\
        0 & 1 & 0 & 0 & \lambda^2 & 0\\
        0 & 0 & 1 & 0 & 0 & 0\\
        0 & 0 & 0 & 1 & 0 & \lambda\\
        0 & 0 & 0 & 0 & 1 & 0\\
        0 & 0 & 0 & 0 & 0 & 1
    \end{pmatrix}\\
    u_{3\alpha_1+2\alpha_2}(\lambda) & = 
    \begin{pmatrix}
        1 & 0 & 0 & 0 & 0 & 0\\
        0 & 1 & 0 & 0 & 0 & 0\\
        0 & \lambda & 1 & 0 & 0 & 0\\
        0 & 0 & 0 & 1 & 0 & 0\\
        0 & 0 & 0 & \lambda & 1 & 0\\
        0 & 0 & 0 & 0 & 0 & 1
    \end{pmatrix}, 
    & u_{-3\alpha_1-2\alpha_2}(\lambda) =
    \begin{pmatrix}
        1 & 0 & 0 & 0 & 0 & 0\\
        0 & 1 & \lambda & 0 & 0 & 0\\
        0 & 0 & 1 & 0 & 0 & 0\\
        0 & 0 & 0 & 1 & \lambda & 0\\
        0 & 0 & 0 & 0 & 1 & 0\\
        0 & 0 & 0 & 0 & 0 & 1
    \end{pmatrix}\\
\end{align*}
\end{remark}

\vfill\pagebreak

\end{document}